\documentclass[11pt,leqno]{article}

\usepackage{latexsym,amssymb,amsthm}
\usepackage{amsmath}
\usepackage{hyperref}
\usepackage{nicefrac}
\usepackage[usenames]{color}
\usepackage{todonotes}
\usepackage{tikz}
\usetikzlibrary{calc}
\usepackage{pgfplots}
\usetikzlibrary{positioning}
\usepackage{subcaption}
\usepackage{wrapfig}
\usepackage[export]{adjustbox}
\usepackage{makecell}

\numberwithin{equation}{section}

\newcommand{\dd}{\mathrm{d}}
\newcommand{\bp}{\begin{proof}}
\newcommand{\ep}{\end{proof}}

\newcommand{\Rand}[1]{\marginpar{#1}}
\renewcommand{\Rand}[1]{}

\newcommand{\be}[1]{\Rand{\vspace{0,6cm}\tt #1}\begin{equation}\label{#1}}
\newcommand{\beL}[1]{\Rand{\vspace{0,6cm}\tt #1}\begin{lemma}\label{#1}}
\newcommand{\el}{\end{lemma}}
\newcommand{\belC}[2]{\Rand{\vspace{0,6cm}\tt #1}\begin{lemma}[#2]\label{#1}}
\newcommand{\beP}[1]{\Rand{\vspace{0,6cm}\tt #1}\begin{proposition}\label{#1}}
\newcommand{\bePC}[2]{\Rand{\vspace{0,6cm}\tt #1}\begin{proposition}[#2]\label{#1}}
\newcommand{\beD}[1]{\Rand{\vspace{0,6cm}\tt #1}\begin{definition}\label{#1}}
\newcommand{\beT}[1]{\Rand{\vspace{0,6cm}\tt #1}\begin{theorem}\label{#1}}
\newcommand{\beC}[1]{\Rand{\vspace{0,6cm}\tt #1}\begin{corollary}\label{#1}}
\newcommand{\bea}[1]{\Rand{\vspace{0,7cm}\tt #1}\begin{eqnarray}\label{#1}}

\newcommand{\wt}{\widetilde}

\def\CA{\mathcal{A}}
\def\CB{\mathcal{B}}
\def\CC{\mathcal{C}}

\def\CM{\mathcal{M}}

\def\CF{\mathcal{F}}

\def\CG{\mathcal{G}}
\def\CH{\mathcal{H}}
\def\CI{\mathcal{I}}
\def\CL{\boldsymbol{\mathcal{L}}}

\def\CU{\mathcal{U}}

\def\CW{\mathcal{W}}
\def\CX{\mathcal{X}}

\DeclareMathAlphabet{\mathpzc}{OT1}{pzc}{m}{it}

\newcommand{\uu}{\underline{u}}

\newcommand{\uur}{\underline{\underline{r}}}
\newcommand{\uuh}{\underline{\underline{h}}}

\newcommand{\e}{\mathrm{e}}

\renewcommand{\d}{{\rm d}}

\def\G{\mathbb{G}}

\def\M{\mathbb{M}}
\def\N{\mathbb{N}}
\def\R{\mathbb{R}}

\def\E{\mathbb{E}}

\def\T{\mathbb{T}}
\def\U{\mathbb{U}}

\newcommand{\Tto}{{_{\D \Longrightarrow \atop t \to \infty}}
}

\newcommand{\Tno}{{_{\D \Longrightarrow \atop n \to \infty}}
}

\newcommand{\ntoo}{{_{\D \longrightarrow \atop n \to \infty}}
}
\newcommand{\utoo}{{_{\D \Longrightarrow \atop u \to \infty}}
}

\newcommand{\mtoo}{{_{\D \Longrightarrow \atop m \to \infty}}
}
\newcommand{\D}{\displaystyle}

\newcommand{\wh}{\widehat}
\newcommand{\eea}{\end{eqnarray}}
\newcommand{\bean}{\begin{eqnarray*}
}
\newcommand{\eean}{\end{eqnarray*}
}

\newcommand{\Prodl}{\prod\limits}

\DeclareMathOperator{\supp}{supp}

 \makeatletter
 \@addtoreset{equation}{section}
 \makeatother

 \newcounter{extralabel}[section]

 \newtheorem{ittheorem}{Theorem}
 \newtheorem{itlemma}{Lemma}
 \newtheorem{itproposition}{Proposition}
 \newtheorem{itdefinition}{Definition}
 \newtheorem{itcorollary}{Corollary}
 \newtheorem{itconjecture}{Conjecture}
 \newtheorem{itremark}{Remark}
 \newtheorem{itassumption}{Assumption}
 \newtheorem{itexample}{Example}

 \newenvironment{theorem}{\addtocounter{extralabel}{1}
 \begin{ittheorem}}{\end{ittheorem}}

 \newenvironment{lemma}{\addtocounter{extralabel}{1}
 \begin{itlemma}}{\end{itlemma}}

 \newenvironment{proposition}{\addtocounter{extralabel}{1}
 \begin{itproposition}}{\end{itproposition}}

 \newenvironment{definition}{\addtocounter{extralabel}{1}
 \begin{itdefinition}}{\end{itdefinition}}

 \newenvironment{corollary}{\addtocounter{extralabel}{1}
 \begin{itcorollary}}{\end{itcorollary}}

 \newenvironment{remark}{\addtocounter{extralabel}{1}
 \begin{itremark}}{\end{itremark}}

  \newenvironment{example}{\addtocounter{extralabel}{1}
  \begin{itexample}}{\end{itexample}}

\newtheorem{zz}{\bf zzz}

\newtheorem{yy}{\bf yyy}

\usepackage{cases}

\usepackage{scalerel,stackengine}
\stackMath
\newcommand\reallywidehat[1]{%
\savestack{\tmpbox}{\stretchto{%
\scaleto{%
\scalerel*[\widthof{\ensuremath{#1}}]{\kern-.6pt\bigwedge\kern-.6pt}%
{\rule[-\textheight/2]{1ex}{\textheight}}%
}{\textheight}%
}{0.5ex}}%
\stackon[1pt]{#1}{\tmpbox}%
}
\begin{document}

%%%%%% TITLE PAGE %%%%%%%%%%%%%%%%%%%%%

\title{Continuum graph dynamics via population dynamics:\\
well-posedness, duality and equilibria}

\author{
Andreas Greven$^1$,
Frank den Hollander$^2$,
Anton Klimovsky$^3$,
Anita Winter$^4$
}

\date{\today}

\maketitle

\small
\begin{abstract}

This paper introduces \textit{graphemes}, a novel framework for constructing and analyzing stochastic processes that describe the evolution of large dynamic graphs.  Unlike graphons, which are well-suited for studying static dense graphs and which are closely related to the Aldous-Hoover representation of exchangeable random graphs, graphemes allow for a modeling of the full space-time evolution of \textit{dynamic} dense graphs, beyond the exchangeability and the subgraph frequencies used in graphon theory. A grapheme is defined as an equivalence class of triples, consisting of a Polish space, a symmetric $\{0,1\}$-valued connection function on that space (representing edges connecting vertices), and a sampling probability measure.  

We focus on graphemes embedded in \textit{ultrametric} spaces, where the ultrametric encodes the \textit{genealogy} of the graph evolution, thereby drawing a direct connection to population genetics. The grapheme framework emphasizes the embedding, in particular, in Polish spaces, and uses stronger notions of equivalence (homeomorphism and isometry) than the exchangeability underlying the Aldous-Hoover representation. We construct grapheme-valued Markov processes that arise as limits of finite graph evolutions, driven by rules analogous to the Fleming-Viot, Dawson-Watanabe and McKean-Vlasov processes from population genetics. We establish that these grapheme dynamics are characterized by well-posed martingale problems, leading to strong Markov processes with the Feller property and continuous paths (i.e., diffusions). We further derive duality relations by using coalescent processes, and identify the equilibria of dynamic graphemes, showing that these are linked to classical distributions arising in population genetics and can therefore be non-trivial. 

Our approach extends and modifies previous work on graphon dynamics \cite{AHR21}, by providing a more general framework that includes a natural representation of the history of the graph. This allows for a rigorous treatment of the dynamics via martingale problems, and yields a characterization of non-trivial equilibria.
\end{abstract}

\emph{Acknowledgement:}
AG was supported by the Deutsche Forschungsgemeinschaft (through grant DFG-GR 876/16-2 of SPP-1590). FdH was supported by the Netherlands Organisation for Scientific Research (through NWO Gravitation Grant NETWORKS-024.002.003), and by the Alexander von Humboldt Foundation. AK was supported by the Deutsche Forschungsgemeinschaft (through Project-ID 443891315 within SPP 2265 and through Project-ID 412848929). AW was supported by the Deutsche Forschungsgemeinschaft (through grant DFG-SPP-2265).
\normalsize

\noindent
\hspace*{0.3cm} {\footnotesize $^{1)}$
Department Mathematik, Universit\"at Erlangen-N\"urnberg, Cauerstrasse 11,
91058 Erlangen, Germany\\
greven@mi.uni-erlangen.de}\\
\hspace*{0.3cm} {\footnotesize $^{2)}$
Mathematisch Instituut, Universiteit Leiden, Niels Bohrweg 1, 2333 CA  Leiden, The Netherlands\\
denholla@math.leidenuniv.nl}\\
\hspace*{0.3cm} {\footnotesize $^{3)}$
Institut f\"ur Mathematik, Emil-Fischer-Strasse 30, 97074 W\"urzburg, Germany\\
anton.klymovskiy@mathematik.uni-wuerzburg.de}\\
\hspace*{0.3cm} {\footnotesize $^{4)}$
Universit\"at Essen, Fakult\"at f\"ur Mathematik, Thea-Leymann-Strasse, 45141 Essen,  Germany\\
anita.winter@uni-due.de}

\small
\tableofcontents
\normalsize

\pagestyle{myheadings}

\markright{\jobname}

%%%%%%%%%%% SECTION 1 %%%%%%%%%%%%%%%%%%%%%%%%%%

\section{Introduction}
\label{s.introduction}

This paper addresses the challenge of constructing and analyzing stochastic processes that describe the evolution of large dynamic graphs.  Classical approaches to graph limits, such as graphons \cite{LovaszSzegedy2006graphon}, are well-suited for describing \textit{static} properties of dense graphs (see Appendix \ref{ss.graphons}). However, they do not naturally capture the \textit{dynamics} of how a graph changes over time, particularly the full history of these changes. Our goal is to develop a rigorous framework -- called grapheme -- for studying graph-valued Markov processes that captures both their short-term behavior and their long-term equilibria. A key requirement is a suitable state space that encodes not only the current graph structure but also its \textit{genealogy}, reflecting how the graph evolved from its initial state to its present state. Such a framework has potential applications in areas where dynamic networks play a crucial role, including complex systems.

In the remainder of this section we argue why graphemes are needed to capture the dynamics of finite and countable graphs (Section \ref{ss.need}), provide an example of a grapheme (Section \ref{ss.example}), provide a formal definition of a grapheme (Section \ref{ss.defgrapheme}), define the class of grapheme dynamics that will be the focus of our paper (Section \ref{ss.dyn}), list the tools that we need along the way (Section \ref{ss.tools}), and give an outline of the remainder of the paper (Section \ref{ss.outline}). 

%%%

\subsection{The need for graphemes}
\label{ss.need}

To address this challenge, we introduce a novel concept called a \textit{grapheme}. Intuitively, a grapheme represents a (finite or) countable graph \textit{embedded} in a suitably chosen Polish space\footnote{= a complete, separable metric space}. The choice of embedding is crucial. Since we would ultimately like to model the evolution of the graph, the embedding should allow us to track how individual vertices and edges change \textit{location} within a continuous embedding space as the graph evolves, capturing the \textit{space-time path process}. The embedding itself is dynamic, evolving in concert with the graph to faithfully record its historical trajectory. We will consider embeddings that evolve in a Markovian way equivalent to the evolution of stochastic processes in populations genetics.

%%%

\subsection{An example}
\label{ss.example}

\begin{example}[A Dynamic Social Network]
\label{exm:dynamic-social-network}
{\rm Consider a social network where individuals (vertices) exert influence on each other (edges). New individuals can join the network, and existing individuals can change their connections based on influence.  Let us trace a simplified evolution (see Figure~\ref{fig:exm_graph_evolution_arcs_line}):

%%%%%%%%%%%%%%%%%%%%%%%%%%%%%%%%%%%%%%%%%%%%%%%%%
\begin{figure}[htbp]
\centering
\begin{subfigure}{0.48\linewidth}
    \centering
    \begin{tikzpicture}[node/.style={circle, draw, fill=black, minimum size=0.2cm, inner sep=0pt}, scale=0.8]
        \draw[thick] (0,0) -- (4,0);
        \node[node, label=below:$x_1$] (x1) at (0,0) {};
        \node[node, label=below:$x_2$] (x2) at (1,0) {};
        \node[node, label=below:$x_3$] (x3) at (2,0) {};
    \end{tikzpicture}
    \caption{\small Initial state: Three vertices, no connections.}
    \label{fig:fv_graph_arcs_t0}
\end{subfigure}\hfill
\begin{subfigure}{0.48\linewidth}
    \centering
    \begin{tikzpicture}[node/.style={circle, draw, fill=black, minimum size=0.2cm, inner sep=0pt},scale=0.8]
        \draw[thick] (0,0) -- (4,0);
        \node[node, label=below:$x_1$] (x1) at (0,0) {};
        \node[node, label=below:$x_2$] (x2) at (1,0) {};
        \node[node, label=below:$x_3$] (x3) at (2,0) {};
        \draw[thick] (x1) to[bend left=45] (x2);
    \end{tikzpicture}
    \caption{\small $t=1$: $x_1$ influences $x_2$.}
    \label{fig:fv_graph_arcs_t1}
\end{subfigure}

\bigskip

\begin{subfigure}{0.48\linewidth}
    \centering
    \begin{tikzpicture}[node/.style={circle, draw, fill=black, minimum size=0.2cm, inner sep=0pt},scale=0.8]
        \draw[thick] (0,0) -- (4,0);
        \node[node, label=below:$x_1$] (x1) at (0,0) {};
        \node[node, label=below:$x_2$] (x2) at (1,0) {};
        \node[node, label=below:$x_3$] (x3) at (2,0) {};
        \draw[thick] (x1) to[bend left=45] (x2);
        \draw[thick] (x2) to[bend left=45] (x3);
        \draw[thick] (x1) to[bend left=45] (x3);
    \end{tikzpicture}
    \caption{\small $t=2$: $x_1$ influences $x_3$.}
    \label{fig:fv_graph_arcs_t2}
\end{subfigure}\hfill
\begin{subfigure}{0.48\linewidth}
    \centering
    \begin{tikzpicture}[node/.style={circle, draw, fill=black, minimum size=0.2cm, inner sep=0pt},scale=0.8]
        \draw[thick] (0,0) -- (4,0);
        \node[node, label=below:$x_1$] (x1) at (0,0) {};
        \node[node, label=below:$x_2$] (x2) at (1,0) {};
        \node[node, label=below:$x_3$] (x3) at (2,0) {};
        \node[node, label=below:$x_4$] (x4) at (3,0) {};
        \draw[thick] (x1) to[bend left=45] (x2);
        \draw[thick] (x1) to[bend left=45] (x3);
        \draw[thick] (x2) to[bend left=45] (x3);
    \end{tikzpicture}
    \caption{\small $t=3$: $x_4$ is added (no connections).}
    \label{fig:fv_graph_arcs_t3}
\end{subfigure}

\bigskip

\begin{subfigure}{\linewidth}
    \centering
    \begin{tikzpicture}[node/.style={circle, draw, fill=black, minimum size=0.2cm, inner sep=0pt},scale=0.7]
        \draw[thick] (0,0) -- (4,0);
        \node[node, label=below:$x_1$] (x1) at (0,0) {};
        \node[node, label=below:$x_2$] (x2) at (0.8,0) {};
        \node[node, label=below:$x_3$] (x3) at (1.6,0) {};
        \node[node, label=below:$x_4$] (x4) at (2.4,0) {};
        \node[node, label=below:$x_5$] (x5) at (3.2,0) {};
        \draw[thick] (x1) to[bend left=45] (x2);
        \draw[thick] (x1) to[bend left=45] (x3);
        \draw[thick] (x1) to[bend left=45] (x5);
        \draw[thick] (x2) to[bend left=45] (x3);
        \draw[thick] (x2) to[bend left=45] (x5);
        \draw[thick] (x3) to[bend left=45] (x5);
    \end{tikzpicture}
    \caption{\small $t=4$: $x_5$ joins and connects to $x_1$.}
    \label{fig:fv_graph_arcs_t4}
\end{subfigure}
\caption{\small Evolution of the graph. Vertices are represented as black dots on the segment $[0,1]$, and edges are represented by arcs above the segment.}
\label{fig:exm_graph_evolution_arcs_line}
\end{figure}
%%%%%%%%%%%%%%%%%%%%%%%%%%%%%%%%%%%%%%%%%%%%%%%%

\begin{enumerate}
\item 
\textbf{Initial State (t=0):} We start with three individuals, $x_1$, $x_2$, $x_3$, represented as points in the unit interval $[0,1]$. Denote by $\mu_0 := \frac{1}{3}(\delta_{x_1}+\delta_{x_2}+\delta_{x_3})$ the \textit{sampling measure} of the vertices. Initially, there are no connections (no influence). We also define an initial embedding \textit{ultrametric} $r_0$ on these points. Since there is no influence, all individuals are genealogically distant: $r_0(x_i, x_j) = \infty$ for $i \neq j$ and $r_0(x_i,x_i):=0$. For practical purposes and to ensure boundedness, we work with the transformed ultrametric $1 - \e^{-r_t}$, and so, initially, the transformed distances between $x_i$ and $x_j$ are $1$ for $i \neq j$ and $0$ for $i=j$.
\item 
\textbf{Time Step 1:}  Suppose that $x_1$ influences $x_2$. This creates an edge between them in the graph.  Critically, we update the (transformed) ultrametric to reflect the new connection: $r_1(x_1, x_2):=1-\e^{-1}$.  This value represents the shortened genealogical distance due to the influence event.  The distances to $x_3$ remain at $r_1(x_i,x_3):=1$, $i=1,2$, and the sampling measure does not evolve: $\mu_1 := \mu_0$.
\item 
\textbf{Time Step 2:} Suppose that $x_3$ is influenced by $x_1$. The graph becomes fully connected. The (transformed) ultrametric is updated: $r_2(x_i,x_j) := 1 - \e^{-1}$, $1\le i<j\le 3$. The sampling measure still does not change: $\mu_2 := \mu_1$.
\item 
\textbf{Time Step 3:} A new vertex $x_4$ joins the network, initially unconnected to others. We update the sampling measure: $\mu_3 := \frac{1}{4}(\delta_{x_1}+\delta_{x_2}+\delta_{x_3}+\delta_{x_4})$. The (transformed) ultrametric distance from $x_4$ to the existing vertices is $r_3(x_4, x_i) := 1$ for $i\in\{1,2,3\}$, while we keep $r_3(x_i,x_j):=1-\e^{-1}$ for $1\le i<j\le 3$.
\item 
\textbf{Time Step 4:} A new vertex $x_5$ is added and connects to $x_1$, $x_2$, $x_3$ and $x_4$.  We update the sampling measure: $\mu_4 := \frac{1}{5}(\delta_{x_1}+\delta_{x_2}+\delta_{x_3}+\delta_{x_4}+\delta_{x_5})$. The (transformed) ultrametric is also updated: $r_4(x_i, x_j) := 1-\e^{-1}$, $i\not= j\in\{1,2,3,5\}$, while we keep $r_3(x_4, x_i) := 1$ for $i\in\{1,2,3,5\}$.
\end{enumerate}
Consider a \textit{connection function} $h_t \colon [0,1]^2 \to \{0, 1\}$ encoding the time-$t$ connectivity of the network. The embedding $r_t$ and the connection function can be derived from $r_t$: for instance, $h_t(x_i, x_j) = 1$ if and only if
$r_t(x_i, x_j)=1-\e^{-t}$ (or some other threshold related to $t$).
} \hfill $\spadesuit$
\end{example}

The above example illustrates how the grapheme, which is represented by the triple $([0,1], h_t, \mu_t)$, and the associated (transformed) ultrametric $r_t$ evolve \textit{together} and capture the dynamic relationships in the network. The ultrametric distance path $(r_t(x_i, x_j))_{t\ge 0}$ encodes the genealogical history of influence between individuals $x_i$ and $x_j$.

%%%

\subsection{Definition of graphemes}
\label{ss.defgrapheme}

Example \ref{exm:dynamic-social-network} motivates the following definition.

\begin{definition}[Grapheme]
\label{def.grapheme}
{\rm A grapheme is an equivalence class of triples $(\mathcal{I}^*, h, \mu)$, where:
\begin{itemize}
\item 
$\mathcal{I}^* = (\mathcal{I}, \tau)$ is a \textit{Polish space}: $\mathcal{I}^*$ provides the underlying space in which the vertices of the graph are embedded.
\item 
$h$ is a measurable, symmetric, $\{0,1\}$-valued \textit{connection function} on $(\mathcal{I} \times \mathcal{I}) \setminus D_{\mathcal{I}\times\mathcal{I}}$, where $D_{\mathcal{I}\times\mathcal{I}} = \{(x,x)\colon\, x \in \mathcal{I}\}$ is the diagonal:  $h$ determines the presence or absence of edges: $h(x, y) = 1$ if there is an edge between vertices at locations $x$ and $y$; $h(x, y) = 0$ if there is no edge.
\item 
$\mu$ is a \textit{sampling measure} on $\mathcal{I}^*$: $\mu$ is a probability measure that tells us how to sample vertices from the underlying space $\mathcal{I}$; we use $\mu$ to relate finite subgraphs to infinite graphemes, allowing us to study the properties of the grapheme through its finite substructures.
\end{itemize}
We will consider three types of equivalence relations between tuples $(\mathcal{I}^*, h, \mu)$, each with increasing levels of strictness:
\begin{enumerate}
\item 
\textbf{$\langle \mathcal{I}^*, h, \mu \rangle$ (Modulo distributions of finite subsamples\footnote{= Modulo exchangeability (see Remark~\ref{rem:exchangeable-arrays-graphemes})}):} Two graphemes are equivalent under this relation if they have the same \textit{statistical} properties, i.e., the distribution of finite subgraphs (sampled according to $\mu$) is the same. This notion is conceptually similar to the graphon plus exchangeability perspective (see Section~\ref{ss.graphons} and Remark~\ref{rem:exchangeable-arrays-graphemes}), but, crucially, we \textit{require} $h$ to be $\{0,1\}$-valued, representing the definite presence or absence of edges. The latter is in contrast to graphons (see Section~\ref{ss.graphons}), which allow for edge probabilities different from $\{0, 1\}$.
\item 
\textbf{$\{ \mathcal{I}^*, h, \mu \}$ (Modulo homeomorphism):} Two graphemes are equivalent under this relation if there exists a measure-preserving \textit{homeomorphism} (a continuous bijection with a continuous inverse) between their underlying Polish spaces that also preserves the connection function $h$.  This type of equivalence respects the \textit{topology} of the embedding space.
\item 
\textbf{$[\mathcal{I}^*, h, \mu ]$ (Modulo isometry):} Two graphemes are equivalent under this relation if there exists a measure-preserving \textit{isometry} (a distance-preserving bijection) between their underlying Polish spaces that also preserves the connection function $h$. This equivalence respects the \textit{metric} structure of the embedding space.
\end{enumerate}
}\hfill$\spadesuit$
\end{definition}

We will denote the sets of equivalence classes under these relations by $\mathbb{G}^{\langle\rangle}$, $\mathbb{G}^{\{\}}$, $\mathbb{G}^{[]}$, respectively.  We will primarily work with the strongest equivalence $\mathbb{G}^{[]}$, because it is a Polish space (see Section~\ref{s.evol}) and allows us to utilize powerful tools from stochastic analysis, particularly the theory of martingale problems, in order to construct and analyze our dynamic graph processes.

A key element of our approach is the use of \textit{ultrametric}\footnote{We expect that this is not the only class of interesting embeddings, but it is a natural one for tracking history.} measure spaces for the embedding space $\mathcal{I}^*$, generated by the genealogy of the evolving graph. An ultrametric space is a metric space where the triangle inequality is strengthened to the \textit{ultrametric inequality}:
\begin{align}
d(x, z) \leq \max\{d(x, y), d(y, z)\} \quad \text{for all } x, y, z \in \mathcal{I}^*.
\end{align}
This choice is deeply connected to the \textit{genealogy} of models in population genetics, which will serve as the foundation for our graph evolution rules. In these models, individuals have ancestors and descendants, and the ultrametric distance between two individuals can represent their \textit{genealogical distance}, for instance, how far back in time we must go to find their \emph{most recent common ancestor}.

By embedding our graphs in ultrametric spaces (specifically, spaces from the class $\mathbb{U}_1$ studied in \cite{GrevenPfaffelhuberWinter2013}), we are essentially encoding the \textit{history} or the \textit{genealogy} of the evolution of the graph in the metric itself. This allows us to naturally capture the space-time path process of the evolving graph. Local changes in the graph, corresponding to events like vertex resampling or birth/death, are reflected in changes of the genealogy and, consequently, in the ultrametric.

%%%

\subsection{Dynamics of graphemes}
\label{ss.dyn}

Next, we outline the dynamics driving grapheme evolution and preview our key findings. We will construct stochastic processes on the space of graphemes $\mathbb{G}^{[]}$ that arise as limits of finite graph evolutions as the number of vertices tends to infinity. These evolutions will be based on rules inspired by models in population genetics, providing a powerful link. See Table~\ref{tab:evolution_rules} for a summary, and below for more details and examples.

%%%%%%%%%%%%%%%%%%%%%%%%%%%%%%%%%%%%%%%%%
\begin{table}[htbp]
\centering
\begin{tabular}{|l|l|l|}
\hline
\textbf{Evolution Rule} & \textbf{Population Dynamics Analogue} & \textbf{Graph Effect} \\ \hline
Fleming-Viot & Resampling & Components merging \\
Dawson-Watanabe & Birth-Death & Vertex addition/removal \\
McKean-Vlasov & Immigration/Emigration & \makecell{Vertex replacement from\\external source} \\ \hline
\end{tabular}
\caption{\small Analogy between grapheme evolution rules and population dynamics mechanisms.}
\label{tab:evolution_rules}
\end{table}
%%%%%%%%%%%%%%%%%%%%%%%%%%%%%%%%%%%%%%%%%%%%%%%

We will consider several classes of dynamics,:
\begin{itemize}
\item 
\textbf{Fleming-Viot (FV) type dynamics:}  These dynamics are analogous to resampling processes in population genetics and allows for infinitely many types (= connected components). In the graph context, this corresponds to selecting pairs of vertices, choosing one vertex of the two, called looser, and rewiring the looser to the connected component of the winner. The vertex set is not changing over time. See Figure~\ref{fig:fv_graph_evolution_arcs_line_new}.

%%%%%%%%%%%%%%%%%%%%%%%%%%%%%%%%%%%%%%%%%%%%%
\begin{figure}[htbt]
\centering
\begin{adjustbox}{minipage=\linewidth, scale=0.8} 
\begin{subfigure}{0.48\linewidth}
    \centering
    \begin{tikzpicture}[node/.style={circle, draw, fill=black, minimum size=0.2cm, inner sep=0pt},scale=0.8]
        \draw[thick] (0,0) -- (4,0);
        \node[node, label=below:$x_1$] (x1) at (0,0) {};
        \node[node, label=below:$x_2$] (x2) at (1.33,0) {};
        \node[node, label=below:$x_3$] (x3) at (2.66,0) {};
        \node[node, label=below:$x_4$] (x4) at (4,0) {};
    \end{tikzpicture}
    \caption{Initial state: Four vertices, no connections.}
    \label{fig:fv_graph_arcs_t0_new}
\end{subfigure}\hfill
\begin{subfigure}{0.48\linewidth}
    \centering
    \begin{tikzpicture}[node/.style={circle, draw, fill=black, minimum size=0.2cm, inner sep=0pt},scale=0.8]
        \draw[thick] (0,0) -- (4,0);
        \node[node, label=below:$x_1$] (x1) at (0,0) {};
        \node[node, label=below:$x_2$] (x2) at (1.33,0) {};
        \node[node, label=below:$x_3$] (x3) at (2.66,0) {};
        \node[node, label=below:$x_4$] (x4) at (4,0) {};
        \draw[thick] (x1) to[bend left=45] (x2);
    \end{tikzpicture}
    \caption{\small $t=1$: $x_1$ influences $x_2$.}
    \label{fig:fv_graph_arcs_t1_new}
\end{subfigure}

\bigskip

\begin{subfigure}{0.48\linewidth}
    \centering
    \begin{tikzpicture}[node/.style={circle, draw, fill=black, minimum size=0.2cm, inner sep=0pt},scale=0.8]
        \draw[thick] (0,0) -- (4,0);
        \node[node, label=below:$x_1$] (x1) at (0,0) {};
        \node[node, label=below:$x_2$] (x2) at (1.33,0) {};
        \node[node, label=below:$x_3$] (x3) at (2.66,0) {};
        \node[node, label=below:$x_4$] (x4) at (4,0) {};
        \draw[thick] (x1) to[bend left=45] (x2);
        \draw[thick] (x3) to[bend left=45] (x4);
    \end{tikzpicture}
    \caption{\small $t=2$: $x_3$ influences $x_4$.}
    \label{fig:fv_graph_arcs_t2_new}
\end{subfigure}\hfill
\begin{subfigure}{0.48\linewidth}
    \centering
    \begin{tikzpicture}[node/.style={circle, draw, fill=black, minimum size=0.2cm, inner sep=0pt},scale=0.8]
        \draw[thick] (0,0) -- (4,0);
        \node[node, label=below:$x_1$] (x1) at (0,0) {};
        \node[node, label=below:$x_2$] (x2) at (1.33,0) {};
        \node[node, label=below:$x_3$] (x3) at (2.66,0) {};
        \node[node, label=below:$x_4$] (x4) at (4,0) {};
        \draw[thick] (x1) to[bend left=45] (x2);
        \draw[thick] (x1) to[bend left=45] (x3);
        \draw[thick] (x2) to[bend left=45] (x3);
    \end{tikzpicture}
    \caption{\small $t=3$: $x_1$ influences $x_3$.}
    \label{fig:fv_graph_arcs_t3_new}
\end{subfigure}
\end{adjustbox}
\caption{\small Evolution of the graph under the Fleming-Viot type dynamics. Vertices are black dots on $[0,1]$, edges are arcs.}
\label{fig:fv_graph_evolution_arcs_line_new}
\end{figure}
%%%%%%%%%%%%%%%%%%%%%%%%%%%%%%%%%%%%%%%%%%%

\item 
\textbf{Dawson-Watanabe (DW) type dynamics:} These dynamics are generalizations of birth-death processes. In the graph context, the latter corresponds to adding new vertices (birth) or removing existing vertices (death), along with their associated edges, which affects both the size (= the number of edges) and the connectivity of the graph. See Figure \ref{fig:dw_graph_evolution_arcs_line}.

%%%%%%%%%%%%%%%%%%%%%%%%%%%%%%%%%%%%%%%%%%
\begin{figure}[htbt]
\centering
\begin{adjustbox}{minipage=\linewidth, scale=0.8} 
\begin{subfigure}{0.48\linewidth}
    \centering
    \begin{tikzpicture}[node/.style={circle, draw, fill=black, minimum size=0.2cm, inner sep=0pt},scale=0.8]
        \draw[thick] (0,0) -- (3,0);
        \node[node, label=below:$x_1$] (x1) at (0,0) {};
        \node[node, label=below:$x_2$] (x2) at (1.5,0) {};
        \node[node, label=below:$x_3$] (x3) at (3,0) {};
	\draw[thick] (x1) to[bend left=45] (x2);
    \end{tikzpicture}
    \caption{\small Initial state: Three vertices, one connection}
    \label{fig:dw_graph_t0}
\end{subfigure}\hfill
\begin{subfigure}{0.48\linewidth}
    \centering
    \begin{tikzpicture}[node/.style={circle, draw, fill=black, minimum size=0.2cm, inner sep=0pt},scale=0.8]
        \draw[thick] (0,0) -- (4,0);
        \node[node, label=below:$x_1$] (x1) at (0,0) {};
        \node[node, label=below:$x_2$] (x2) at (1.33,0) {};
        \node[node, label=below:$x_3$] (x3) at (2.66,0) {};
        \node[node, label=below:$x_4$] (x4) at (4,0) {};
        \draw[thick] (x1) to[bend left=45] (x2);
        \draw[thick] (x1) to[bend left=45] (x4);
        \draw[thick] (x2) to[bend left=45] (x4);
    \end{tikzpicture}
    \caption{\small $t=1$: $x_4$ is added (``born'' to $x_1$) and connects to $x_1$ and its friend $x_2$.}
    \label{fig:dw_graph_t1}
\end{subfigure}

\bigskip

\begin{subfigure}{\linewidth}
    \centering
    \begin{tikzpicture}[node/.style={circle, draw, fill=black, minimum size=0.2cm, inner sep=0pt},scale=0.8]
        \draw[thick] (0,0) -- (4,0);
        \node[node, label=below:$x_2$] (x2) at (1.33,0) {};
        \node[node, label=below:$x_3$] (x3) at (2.66,0) {};
        \node[node, label=below:$x_4$] (x4) at (4,0) {};
        \draw[thick] (x2) to[bend left=45] (x4); 
    \end{tikzpicture}
    \caption{\small $t=3$: $x_1$ is removed (``dies'').}
    \label{fig:dw_graph_t4}
\end{subfigure}
\end{adjustbox}
\caption{\small Evolution of the graph under the Dawson-Watanabe dynamics. Vertices are black dots on $[0,1]$, edges are arcs. Birth events add a new vertex, initially unconnected. Death events remove a vertex and all its incident edges.}
\label{fig:dw_graph_evolution_arcs_line}
\end{figure}

\item 
\textbf{McKean-Vlasov (McKV) type dynamics:} These dynamics involve replacing vertices with new vertices drawn according to an external source distribution. This is analogous to immigration/emigration in population models, and allows to model an influx of new structures into the graph. See Figure~\ref{fig:mkv_graph_evolution_arcs_line}.

%%%%%%%%%%%%%%%%%%%%%%%%%%%%%%%%%%%%%%%%%%%%%%%%%%%%%%%    
\begin{figure}[htbt]
\centering
\begin{adjustbox}{minipage=\linewidth, scale=0.8} 
\begin{subfigure}{0.48\linewidth}
    \centering
    \begin{tikzpicture}[node/.style={circle, draw, fill=black, minimum size=0.2cm, inner sep=0pt},scale=0.8]
        \draw[thick] (0,0) -- (3,0);
        \node[node, label=below:$x_1$] (x1) at (0,0) {};
        \node[node, label=below:$x_2$] (x2) at (1.5,0) {};
        \node[node, label=below:$x_3$] (x3) at (3,0) {};
        \draw[thick] (x2) to[bend right=45] (x1);
        \draw[thick] (x2) to[bend left=45] (x3);
    \end{tikzpicture}
    \caption{\small Initial state: Three individuals, no connections.}
    \label{fig:mkv_graph_t0}
\end{subfigure}\hfill
\begin{subfigure}{0.48\linewidth}
    \centering
    \begin{tikzpicture}[node/.style={circle, draw, fill=black, minimum size=0.2cm, inner sep=0pt},scale=0.8]
        \draw[thick] (0,0) -- (3,0);
        \node[node, label=below:$x_1$] (x1) at (0,0) {};
        \node[node, label=below:$x_3$] (x3) at (3,0) {};
        \node[node, label=below:$x_4$] (x4) at (1,0) {};
    \end{tikzpicture}
    \caption{\small $t=1$: $x_2$ is removed and replaced by the ``newcomer'' $x_4$}
    \label{fig:mkv_graph_t1}
\end{subfigure}
\end{adjustbox}
\caption{\small Evolution of the graph under the McKean-Vlasov dynamics. Vertices are black dots on $[0,1]$, edges are arcs. Vertices are replaced by new vertices from an external source.}
\label{fig:mkv_graph_evolution_arcs_line}
\end{figure}
%%%%%%%%%%%%%%%%%%%%%%%%%%%%%%%%%%%%%%%%%

\end{itemize}

We will also consider extensions of these basic dynamics:

\begin{itemize}
\item 
\textbf{Marked Graphemes:} Vertices can have associated types (drawn from a Polish space $V$), and the evolution rules can depend on these types, allowing for graphs with heterogeneous node properties.
\item 
\textbf{Random-size Graphemes:} The number of vertices can change over time, allowing for graphs that grow or shrink.
\item 
\textbf{Non-completely connected components:} After introducing mechanisms for adding and removing edges independently, we can model graphs where connected components are not necessarily complete subgraphs.
\end{itemize}

%%%

\subsection{Tools}
\label{ss.tools}

Our work builds upon and extends previous research on graph limits, particularly the theory of graphons \cite{LovaszSzegedy2006graphon} and recent work on graphon dynamics \cite{AHR21}. While graphons provide a powerful framework for studying static properties of dense graphs, they do not naturally capture the dynamics of graph evolution.

The work in \cite{AHR21} introduced diffusion-like graphon-valued processes, also drawing inspiration from population genetics, but it left open several important questions:
\begin{itemize}
\item 
Are the graphon-valued processes strong Markov?
\item 
Can the graphon-valued processes be described by a generator acting on a dense class of test functions?
\item 
Do there exist non-trivial equilibria?
\end{itemize}
Our grapheme approach settles these questions. By embedding graphs in Polish spaces and leveraging the genealogical structure of population models, we are able to construct well-behaved Markov processes with continuous paths, characterize them by using martingale problems, establish duality relations, and identify non-trivial equilibria.  The use of ultrametric spaces, driven by the underlying genealogy, is crucial for capturing the space-time evolution of the graph in a \emph{natural and minimal way}.

Our main results show that grapheme dynamics are mathematically well-behaved.  Specifically, we show the following:
\begin{enumerate}
\item 
\textbf{Well-Posed Martingale Problems:} Grapheme dynamics are characterized by well-posed martingale problems.
\item 
\textbf{Strong Markov Processes with the Feller Property:} Grapheme processes are strong Markov processes with the Feller property.
\item 
\textbf{Diffusion Processes:} Grapheme processes are diffusions, meaning that they have continuous paths in the space of graphemes.
\item 
\textbf{Equilibria and Long-Term Behavior:} Grapheme processes have equilibria that can be non-trivial.
\item 
\textbf{Duality:} Grapheme processes are dual to coalescent processes.
 \item \textbf{Approximation by Finite Graphs:} Grapheme processes can be approximated by sequences of finite graph evolutions.
\end{enumerate}

%%%

\subsection{Outline}
\label{ss.outline}

The remainder of the paper is organized as follows:
\begin{itemize}
\item 
\textbf{Section \ref{s.main-results}} states our main theorems.
\item 
\textbf{Section \ref{s.evol}} formally defines the state spaces for graphemes and the associated topologies.
\item 
\textbf{Section \ref{s.exmart}} specifies the martingale problems for the grapheme dynamics, focusing on the operators that play the role of generators.
\item 
\textbf{Sections \ref{s.state}--\ref{s.compo}} provides the proofs of our main theorems, including results on well-posedness, duality, equilibria, and path properties. (We leverage results from the theory of genealogy-valued diffusions \cite{GrevenPfaffelhuberWinter2013,DGP12,DG19,ggr_GeneralBranching} to establish our claims.)
\item 
\textbf{Section \ref{s.discussion}} discusses our findings, relates them to graphon dynamics, and outlines future research directions.
\item 
\textbf{Appendix \ref{ss.graphons}} briefly reviews related literature on graphon dynamics.
\item
\textbf{Appendix \ref{s.connections}} list connections to the literature.
\end{itemize}

%%%%%%%%%%%% SECTION 2 %%%%%%%%%%%%%%%%%%

\section{Main results}
\label{s.main-results}

In this section we present our main results concerning the evolution of finite $n$-graphs under stochastic rules and their convergence to grapheme-valued limiting processes. We define Markovian evolution rules inspired by population dynamics -- namely, Fleming-Viot (FV), Dawson-Watanabe (DW), and McKean-Vlasov (McKV) processes -- and establish their convergence to grapheme-valued diffusions, characterized via martingale problems. Our results generalize prior work \cite{AHR21} by incorporating genealogical structures, with detailed statements organized as evolution rules (Section~\ref{ss.examples}), main theorems (Section~\ref{sss.theorem}), consequence and extensions (Section \ref{sss.conex}).

%%%

\subsection{Evolution rules for finite graphs}
\label{ss.examples}

The finite graphs evolve as \emph{Markov pure-jump processes} inducing a grapheme $([n],h,\mu)$, where $h$ is the connection matrix and $\mu$ is the sampling measure, with $\supp(\mu) = [n]$. Write $\G_n$ to denote the set of associated equivalence classes, and put $\G_\infty = \cup_{n\in\N} \G_n$ (which is needed when the number of vertices may vary). We refer to the elements as \emph{finite} graphemes. The evolution rules below describe transitions for vertices and edges occurring at certain rates. In Section~\ref{sss.compcase}, we give the evolution rules on which we focus. In Section~\ref{sss.cmd} we look at the connection matrix. In Section~\ref{sss.goalsgen} we state our objectives. In Section~\ref{sss.link}, we make the link with processes of evolving populations and their genealogies, define our grapheme process $\CG$ and lift our graph dynamics specified by the simple rules to a $\G^{[]}_n$-valued stochastic grapheme process.

%%%

\subsubsection{Two classes of evolution rules from population genetics}
\label{sss.compcase}

We now specify the evolution rules governing finite graphs, cf.~Table~\ref{tab:evolution_rules}, categorized into two classes: those preserving completely connected components and extensions allowing non-complete connectivity.

\medskip\noindent
${\bf (I)}$ We first focus on evolution rules for graphs with \emph{completely connected components} (= all clusters or cliques are complete subgraphs). These have a version with a fixed number of vertices (Fleming-Viot) and a version with a randomly evolving number of vertices (Dawson-Watanabe), and each has a corresponding rule of immigration and emigration (McKean-Vlasov rule), both for fixed size and variable size. The evolution rules carry these names because for a population sizes tending to infinity these processes are the \emph{diffusion limits} of the Moran dynamics and Galton-Watson dynamics, respectively.

We next specify the \emph{rules} and the \emph{initial states}.

%%%

\paragraph*{$\blacktriangleright$ The Fleming-Viot evolution rule.}

At rate $d \in (0,\infty)$ pick a pair of vertices and perform the following transition:
\begin{itemize}
\item[-]
If the vertices belong to the same component, then do nothing.
\item[-]
If the vertices belong to different components, then throw a fair coin to decide which vertex is the winner, respectively, the looser, remove all the edges of the looser and replace them by edges to all the vertices in the connected component of the winner.
\end{itemize}

%%%

\paragraph*{$\blacktriangleright$ The Dawson-Watanabe evolution rule.}

At rate $b \in (0,\infty)$ pick a vertex and perform the following transition:
\begin{itemize}
\item[-]
Throw a fair coin. If the outcome is $1$, then add a new vertex and add edges between the new vertex and all the vertices in the component of the chosen vertex. If the outcome is $0$, then delete the chosen vertex and all its edges.
\end{itemize}

%%%

\paragraph*{$\blacktriangleright$ The McKean-Vlasov evolution rule.}

$\mbox{}$

\medskip\noindent
$\bullet$ \emph{Fixed size}:
At rate $c \in [0,\infty)$ pick a vertex, \emph{remove} the chosen vertex and all its edges, and \emph{add} a new vertex drawn from a source $S$ of labels according to a probability measure $\theta$ on $S$ (which is a Polish space), and \emph{connect} it by edges to all the vertices with the same label. Note that for $\theta$ diffusive, the new vertex enters without edges, which leads to a simple rule: the new vertex becomes a new connected component and forgets its label. For atomic $\theta$, on the other hand, a vertex is potentially added to an existing completely connected component, and the weights of the atoms in $\theta$ become relevant parameters.

\medskip\noindent
$\bullet$ \emph{Variable size}: Perform the addition and removal of vertices independently at rate $c$.

\medskip
For FV + McKV (= resampling + immigration and emigration) the number of vertices is preserved, while for DW + McKV (= birth or death + immigration and emigration) the number of vertices varies. All four rules have the property that they preserve \emph{complete connectedness}.

%%%

\paragraph*{$\blacktriangleright$ Possible extensions.}

We can add further evolution rules with interesting new features. To do so, we give vertices a \emph{type} drawn from a fixed type-set $V$. Then we can reformulate the rules in (I) by requiring edges to be present if and only if the two vertices share the same type, and in (II) by adding new rules for the addition and removal of edges.

\medskip
It is possible to add a \emph{bias of types} in the resampling mechanism. For instance, the fair coin in the FV-rule is replaced by a biased coin, with a bias that depends on the types of the pair of vertices chosen (= \emph{selection}). Another option is to add to resampling in the marked case a change of type: a vertex receives a type drawn independently according to a probability measure $\theta_v \in \CM_1(V)$ that depends on the type $v$ of the chosen vertex (= \emph{mutation}). It is also possible to evolve via resampling in larger sweeps. For instance, in the FV-rule a positive fraction of the vertices is chosen, rather than a pair of vertices, and all are connected to the winner (= \emph{Cannings} resampling).

\medskip\noindent
${\bf (II)}$ We next focus on evolution rules for graphs with \emph{non-completely connected components} (= not all clusters or cliques are complete subgraphs). Here we use the marked version of (I) with some new rules.

%%%

\paragraph*{$\blacktriangleright$ Reformulation of (I).}

\begin{itemize}
\item[--]
Grow the graph according to the FV-rule, with the modification that the looser adopts the type of the winner.
\item[--]
Grow the graph according to the DW-rule, with the modification that the new vertex adopts the type of the chosen vertex.
\item[--]
Grow the graph according to the McKV-rule, with the modification that the new vertex receives a type that is drawn independently according to a probability measure $\theta \in \CM_1(V)$.
\end{itemize}

\paragraph*{$\blacktriangleright$ Addition of edges.}
Apply the same rules as above. In addition, for each pair of vertices carrying the same type but having no edge between them, at rate $a^+ \in [0,\infty)$ add an edge.

%%%

\paragraph*{$\blacktriangleright$ Removal of edges.}

Apply the same rules as above. In addition, for each pair of vertices carrying the same type and having an edge between them, at rate $a^-\in [0,\infty)$ remove the edge.

\medskip
The latter two evolution rules do \emph{not} preserve complete connectedness. We can run the evolution as in (I), but with the two rules above added (which are switched off when $a^\pm=0$). The new Markov pure-jump processes allow for equilibria in which the components of a single type are random in number and in size, and in which edges evolve randomly.

%%%

\paragraph*{Generalizations of the addition and removal mechanism, non-Markovian $\G^{\langle\rangle}$-valued process.}

An interesting extension is to consider non-Markovian processes with values in $\G^{\langle\rangle}$ obtained by constructing a Markov process with values in $\G^{[]}$ and letting $a^+,a^-$ be \emph{bounded measurable functions} of the state that depend on the age and on the current size of the subpopulation. This is treatable with the same methods, but generates a non-Markovian process on $\G^{\langle\rangle}$. Note that age and subfamily size are observables in the current state of the genealogy, namely, distance and corresponding measure charged to a ball with a radius equal to the distance.

%%%

\paragraph*{Initial conditions in (I) and (II).}

For finite graphs we can work with all the elements of $\G_n$ or $\G_\infty$ as initial condition. It is only in the $n \to \infty$ limit that we need to impose certain restrictions.

\begin{remark}[Connection to~\cite{AHR21}]
\label{r.1685}
{\rm Adding types in $V=[0,1]$ to the underlying Fleming-Viot process and setting $h((i,v),(j,v))=1$ for $i,j \in U,v \in V$ and $0$ elsewhere, we obtain a model with types and values in $\G^{[],V}$. Now, compared to Fleming-Viot, further vertices are connected in the case of the Fisher-Wright rule, since they have different ancestors at time $0$ but have the \emph{same type} (recall that descendants have the type of the ancestor). This gives the \emph{Fisher-Wright} process of~\cite{AHR21}, as opposed to the Fleming-Viot process we focus on here.} \hfill$\spadesuit$
\end{remark}

%%%

\subsubsection{Connection matrix distribution}
\label{sss.cmd}

Let $\CI^*=(\CI,\tau)$ be a \emph{Polish space}, with $\CI \neq \emptyset$ a set and $\tau$ a topology, $\CB_\CI$ the Borel-$\sigma$-algebra of subsets of $\CI$, $h$ a $\CB_\CI$-measurable symmetric $\{0,1\}$-valued \emph{connection function} on $(\CI \times \CI) \setminus D_{\CI\times\CI}$, with $D_{\CI\times\CI}$ the diagonal of $\CI\times\CI$, and $\mu$ a \emph{sampling measure}\footnote{a probability measure} on $\CI^*$. (For graphemes it is crucial that the connection function $h$ is $\{0,1\}$-valued, which is \emph{different} from graphons where $\wt h \in \wt \CW$ is based on a function $h$ that is $[0,1]$-valued.)

If we draw an $\N$-sample by using a \emph{non-atomic} (= diffuse) sampling measure $\mu$, setting
\begin{equation}
X = (x_i)_{i \in \N}
\end{equation}
drawn from $\CI$ according to $\mu^{\otimes\N}$. In other words $x_i$, $i \in \N$ are i.i.d.~$\mu$-distributed $\CI$-valued random variables. Consider the matrix
\begin{equation}
H = \big(h(x_i,x_j)\big)_{(i,j) \in (\N \times \N) \setminus D_{\N\times\N}},
\end{equation}
then we obtain a \emph{countable graph} $\CG$ embedded in $\CI$, of which the vertices are represented by $X$ and the edges are represented by $H$, in the sense that the vertices are labelled by $\N$ while vertices $i$ and $j$ are connected by an edge if and only if $h(x_i,x_j) = 1$  (see Fig.~\ref{fig-grapheme}).

%%%%%%%%%%%%%%%%%%%%%%%%%%%%%%%%%%%%%%
\begin{figure}[htbp]
\begin{center}
\setlength{\unitlength}{1cm}
\begin{picture}(10,3)(0,-1.5)
%%%
{\thicklines
\qbezier(0,0)(5,0)(10,0)
\qbezier(1,0)(1.5,0.5)(2,0)
\qbezier(1.5,0)(3.5,1)(5.5,0)
\qbezier(4,0)(6,1)(8,0)
\qbezier(8,0)(8.5,0.5)(9,0)
}
\put(10.5,-.1){$\CI$}
\put(5,-0.6){$X$}
\put(5,0.6){$H$}
%%%
\put(1,0){\circle*{0.25}}
\put(1.5,0){\circle*{0.25}}
\put(2,0){\circle*{0.25}}
\put(3.5,0){\circle*{0.25}}
\put(4,0){\circle*{0.25}}
\put(5.5,0){\circle*{0.25}}
\put(6.5,0){\circle*{0.25}}
\put(8,0){\circle*{0.25}}
\put(9,0){\circle*{0.25}}
\end{picture}
\end{center}
\vspace{-1.3cm}
\caption{\small Representation of a graph as a grapheme: $\CI$ space, $X$ vertices, $H$ edges.}
\label{fig-grapheme}
\vspace{0.3cm}
\end{figure}
%%%%%%%%%%%%%%%%%%%%%%%%%%%%%%%%%%%%%%

\noindent
Note that without loss of generality we may assume that $\supp(\mu) = \CI$. Also note that graphs are automatically simple: no self-loops and no multiple edges are present. If we allow $\mu$ to have atoms, then we have to sample \emph{without replacement} in order to preserve this property. The connection function can be trivially extended to $\CI \times \CI$ by setting $h=0$ on $D_{\CI \times \CI}$, which we do henceforth.

Two cases are of interest: the set of vertices generated by the sampled infinite sequence for $\CG$ is
\begin{itemize}
\item[(I)]
\emph{deterministic} when $\mu$ is atomic;
\item[(II)]
\emph{random} when $\mu$ is non-atomic (= diffuse) and all sample points are different.
\end{itemize}
In case (I), the cardinality of the vertex set of $\CG$ equals $|\supp(\mu)|$, which can be finite or countably infinite. In case (II), the cardinality is countably infinite. In case (II), a sampled sequence contains only different points and all sequences that we can sample are $\mu^{\otimes \N}$-a.s.\ equivalent, in the sense that they are statistically indistinguishable: all finite samples have the same distribution and so the \emph{equivalence class is deterministic}. The mixed case is of no interest, as we will see below.

For $n\in\N$, the $n$-sample $(x_i)_{i \in [n]}$ is drawn from $\CI$ according to $\mu^{\otimes n,\downarrow}$, i.e., without replacement, which generates a random finite graph $\CG_n$. For $m \in [n]$, draw $m$ vertices \emph{uniformly at random without replacement} from the $n$-sample. Denote the distribution of its connection matrix by $\nu^{(m),n}$. As $n\to\infty$, $\nu^{(m),n}$ converges weakly $\mu^{\otimes \N}$-a.s.\ to a connection-matrix distribution $\nu_m$. Moreover, as $m \to \infty$, $\nu_m$ converges weakly $\mu^{\otimes \N}$-a.s.\ (as a projective limit) to a \emph{connection-matrix distribution}
\begin{equation}
\label{e510}
\nu \in \CM_1(\{0,1\}^{(\N\times \N) \setminus D}),
\end{equation}
which is the \emph{characteristic object} describing the equivalence class of statistically indistinguishable countable graphs. Therefore, instead of $(\CI,\tau),h,\mu$ labelling a grapheme, we might have taken $(\CI,\tau),\nu,\mu$, because $h$ is the random variable on $\CI \times \CI$ realising $\nu$. Thus, in case (II), even though the graph generated by $\CG$ is random, its connection-matrix distribution $\nu$ is \emph{deterministic}, contains all the information on the edges, and does not use properties of the embedding other than its existence. Note that if we have a \textit{diffuse} and an atomic component, then the latter is statistically not visible in the countable graph, and $\nu$ just depends on the sample we draw from the diffuse part of the measure, because the graph is \emph{simple}.

Note that $\nu$ determines all the subgraph densities of $\CG$ (see~\eqref{eq:homdensity}), as these are normalised expectations under the sampling distribution for samples of size $k \in [n]$, arising from the law of large numbers as the sample size tends to infinity. They also determine the graphon corresponding to the graph (see Figs.~\ref{fig:graph-and-graphon}).

%%%

\subsubsection{Main goals for general grapheme dynamics}
\label{sss.goalsgen}

Our task is to verify the following (which we will work out in Sections~\ref{s.evol}--\ref{s.duality} for the simple finite-graph dynamics given in Section~\ref{ss.examples}):
\begin{itemize}
\item
Let $\CG^{(n)} = (\CG_t^{(n)})_{t \geq 0}$ be a Markov pure-jump process based on a triple $([n],h,\mu)$ with $\supp(\mu) = [n]$. As $n\to\infty$, after the finite graphemes have been suitably \emph{embedded} in a continuum Polish space, this process converges to a limit process $\CG=(\CG_t)_{t \geq 0}$ in $\G^{[]}$.
\item
The limit process is described by a
\begin{equation}
\label{e336}
(\wh\CL^\ast,\wh\Pi^\ast,\wh\Gamma^\ast)\text{-well-posed martingale problem},
\end{equation}
with $\wh\CL^\ast$ a \emph{linear operator} on a measure-determining and convergence-determining subalgebra $\wh \Pi^\ast$ of $C_b(\wt \G^{[]},\R)$ playing the role of a generator, with $\wh\Pi^\ast$ an \emph{algebra of polynomials} (to build in \emph{sampling} of \emph{finite substructures}) on $\G^{[]}$ playing the role of a domain of test functions for $\wh\CL^\ast$, and $\wh\Gamma^\ast$ an \emph{initial law} on $\G^{[]}$. The solution is unique and defines a strong Markov process on $\G^{[]}$ with the Feller property.
\item
The same holds for the functional $\nu = (\nu_t)_{t \geq 0}$, viewed as process on $\G^{\langle\,\rangle}$, via its own
\begin{equation}
\label{e544}
\left(\CL^\ast,\Pi^\ast,\Gamma^\ast\right)\text{-well-posed ``martingale problem"},
\end{equation}
with $\CL^\ast$ a linear operator on $C_b(\G^{\langle\rangle},\R)$, $\Pi^\ast$ an algebra of polynomials on $\G^{\langle\rangle}$, and $\Gamma^\ast$ an initial law on $\G^{\langle\rangle}$. The solution is unique and defines a Markov process on a subset of $\G^{\langle\rangle}$. However, proving existence of the solution requires getting a solution of \eqref{e336} for some natural choice of $\CI^\ast$ and $\mu$. Note here that $\G^{\langle\rangle}$ is only a measurable subset of a Polish space $\wt \CW$ and is not closed in $\wt \CW$. Therefore the \emph{``martingale problem"} in~\eqref{e544} allows for many nice calculations, but is \emph{not} of the standard form. This is only the case when we consider it on $\G^{[]}$, as a martingale problem of a functional, and after passing to a closed subspace. Otherwise, the tools available from stochastic analysis are severely restricted. To resolve this problem we must use the property that the dynamics immediately moves into a smaller space.
\item
The process of connection matrix distributions $\nu=(\nu_t)_{t \geq 0}$ associated with $\CG$, which arises also from the process $\nu^{(n)}=(\nu_t^{(n)})_{t \geq 0}$ associated with $\CG^{(n)}=(\CG_t^{(n)})_{t \geq 0}$ in the limit $n\to\infty$, may be either \emph{autonomous} (i.e., not feel the refinement of the $\langle\rangle$-equivalence class) or may evolve in the \emph{random environment} given by~\eqref{e508} (i.e., feel the refinement). In the autonomous case, the \emph{equilibria} of $\CG=(\CG_t)_{t \geq 0}$ are described by quantities arising from the distribution of the connection-matrix of sampled subgraphs of the finite graph dynamics.
\end{itemize}

%%%

\subsubsection{From population dynamics to grapheme dynamics}
\label{sss.link}

In~\cite{AHR21}, stochastic processes in \emph{population dynamics} were used to build a dynamics on the space $\D$ of graphons, in particular, the Fisher-Wright diffusion and the measure-valued Fleming-Viot diffusion. The graph evolution rules given above also stem from population dynamics and lead to Markov processes that are already used in population genetics. The new idea is to take into account additional information, namely, the \emph{genealogical structure} of the population, which are the key tool for graphemes. We will use graphemes to resolve two sets of issues left open in \cite{AHR21}: martingale problem characterisation, proof of the \emph{strong Markov property} and the \emph{Feller property}, and construction of dynamics with \emph{non-trivial equilibria}.

To achieve this, we will exploit the framework of \emph{duality} and \emph{well-posed martingale problems}. For the latter, the key idea is to obtain \emph{functions of the elements in $\G^\ast$} of the form $((\CI^\ast,h,\mu))$, with $\ast$ and the outer bracket standing for $[],\{\},\langle\rangle$, by taking a sample of size $n$ according to $\mu$, considering functions $\varphi$ of graphs with $n$ vertices and their edges, and choosing equal weights to get elements of $\G_n$. These $\varphi$ are evaluated on the sample of size $n$ from $\CI$. We define an $n$-monomial on $\G$ by taking as value of the function the \emph{$n$-sample expectation} of $\varphi$ (see Section~\ref{ss.space}). Such functions are called \emph{polynomials} and will appear as test functions specifying the domain of the operator in the martingale problem.

In this section, we define the grapheme process $\CG$ as a \emph{functional} of a \emph{genealogy-valued diffusion} $\CU$. In Section~\ref{sss.theorem} we provide a characterisation via a well-posed martingale problem and via a limit of finite grapheme evolutions. The evolution rules described above (Fleming-Viot, Dawson-Watanabe, McKean-Vlasov) look at the \emph{vertex population} exactly like the transitions for the population of individuals in the Moran, respectively, Galton-Watson branching model with immigration/emigration from a source, and are well known in population genetics. In particular, their $n \to \infty$ limit \emph{measure-valued} diffusions have been established a long time ago see~\cite{Daw93,EK86}. The corresponding processes of evolving \emph{genealogies} have been studied more recently. The evolution of these genealogies can be described as a Markov process and can be shown to converge in the limit as $n\to\infty$, as follows.

As explained in \cite{GrevenPfaffelhuberWinter2013}, the \emph{encoding of genealogies} of a stochastically evolving population (in the $n \to \infty$ limit) runs via
\begin{align}
\label{e660}
\text{equivalence classes } [U^*,\mu]
\end{align}
with $U^* = (U,r)$, $U$ a set, $r$ an ultrametric on $U$, $\mu$ a probability measure on $U^*$, and \emph{equivalence} taken with respect to \emph{measure-preserving isometries} of the \emph{supports of the measure} $\mu$. This leads to a state space equipped with a \emph{Polish topology} known as
\begin{align}
\label{a666}
\text{$\U_1$ of genealogies, \quad $\U_1^V$ of $V$-marked genealogies,}
\end{align}
respectively, with state spaces
\begin{align}
\label{e670}
\U_\mathrm{fin},\,\U_{\mathrm{fin}}^V,
\end{align}
when $\mu$ is a \emph{finite} measure, so that we have a \emph{population size} $\bar \mu=\mu (U)$ and a \emph{sampling measure} $\wh \mu=\mu/\bar \mu$.

The interpretation is as follows: $U$ is the set of individuals, the metric is given by the genealogical distance between two individuals, i.e., twice the time back to the most recent common ancestor or $2t$ if there is none by time $t$ (in particular, we have an \emph{ultrametric}), while $\mu$, respectively, $\wh \mu$ is the so-called \emph{sampling measure}, allowing to draw samples of individuals for the vertices and hence also for the edges. Convergence in the topology of $\U_1,\U$ means that the distributions of sampled \emph{finite} subspaces converge (as finite ultrametric measure spaces). Recall that an ultrametric space has the geometric property that it can be decomposed in disjoint balls with an arbitrary radius that are open and closed.

We observe that we can turn our finite graph dynamics from Section~\ref{sss.compcase}, first into a $\G_n$- or $\G_\infty$-valued process and then into a \emph{$\G_n^{[]}$- or $\G_\infty^{[]}$-valued grapheme dynamics}, by defining an ultrametric $r_t$ on $[n]$, by setting the distance equal to twice the time back to the most recent common ancestor of the two individuals and equal to $2t$ if there is none, thereby turning the graph into a grapheme. Here, $\G_n^{[]}$ denotes the subspace of graphemes from $\G^{[]}$ for which the sampling measure is concentrated on $n$ points.

%%%

\paragraph*{Basic idea.}

In the literature, on the spaces in \eqref{a666} and \eqref{e670}, \emph{diffusions} $(\CU_t)_{t \geq 0}$ have been constructed describing genealogies for the classical measure-valued population models, called Fleming-Viot or Dawson-Watanabe diffusions, arising as limits of the genealogies in the Moran and Galton-Watson model,  (see~\cite{GrevenPfaffelhuberWinter2013,DGP12,DG19,ggr_GeneralBranching}), which we will use later on.

The idea to study grapheme processes $\CG$ with values in $\G^{[]}$ (as specified in Section~\ref{sss.compcase}) is to consider $\U_1$-valued processes
\begin{equation}
\label{e873}
\CU=(\CU_t)_{t \geq 0} \text{  with  } \CU_t =[U^*_t,\mu_t]=[U_t,r_t,\mu_t]
\end{equation}
to get a space $(U,r)$ in which we embed our countable graph from the sample sequence.

\begin{definition}[Grapheme processes associated with $\U$-valued diffusions $\CU$]
\label{def.graphproc}
$\mbox{}$\\
{\rm (a) Consider the case where $\CU_0=[\{1\}, \underline{\underline{0}},\delta_1]$, the \emph{single root case} (i.e., every vertex is on its own at time $0$), from which we construct the general case by a simple operation, namely, by picking an operation $\vdash$ on $\U$ and taking $(\CG_0 \vdash \CG^\ast_t)_{t \geq 0}$, with $\CG^\ast$ the solution of the single root initial state (see the paragraph above~\eqref{e1094}).\\
(b) At every time $t \geq 0$, choose a representative $(U_t,r_t,\mu_t)$ of $\CU_t$. Define the triplet (viewed as a pre-grapheme)
\begin{align}
\label{e678}
\text{ $(U^*_t, h_t,\mu_t)$ with $h_t(x,y) = 1$ if and only if $r_t(x,y) \leq 2t$,}
\end{align}
called the grapheme process $\CG$ associated with $\CU$, i.e., the connection-matrix is chosen so as to incorporate the genealogy. (The last restriction means that at time $t$ there are edges exactly for the pairs of vertices with the same common ancestor at time $0$.) The $2t$ bound in \eqref{e678} relates to the genealogical distance (twice the time back to the common ancestor).
\\
(c) The equivalence class $[(U_t,r_t,h_t,\mu_t)]$ depends on the equivalence class in $\G^{[]}$ only via the equivalence class of the triple $(U_t,r_t,\mu_t)$ in $\U_1$ or $\U$, i.e., on $\CU_t$, and defines a grapheme $\CG_t$ in $\G^{[]}$, and therefore defines a
\begin{equation}\label{e684}
\text{$\G^{[]}$-valued process $(\CG_t)_{t \geq 0}$ for each process $\U_1$-valued $(\CU_t)_{t \geq 0}$ prescribed.}
\end{equation}
We may think of \eqref{e684} as associating with every $\U_1$-valued, respectively, $\U$-valued path a new path with values in $\G^{[]}$, respectively, $\bar \G^{[]}$.}
\hfill$\spadesuit$
\end{definition}

\begin{remark}[$\G_n$- or $G_\infty^{[]}$-grapheme dynamics]\label{r.888}
{\rm Note that if we take the version of the finite-$n$ genealogical process, say $\CU^n$, then \emph{via the finite-graph dynamics introduced in Section~\ref{sss.compcase}} we obtain a version of the process $\CG^n$ by applying~\eqref{e678} and generating a process $\CG^n$ with values in $\G^{[]}_n$.} \hfill$\spadesuit$
\end{remark}

The connection via genealogies with the $\U_1$-valued dynamics allows us to exploit the following:
\begin{itemize}
\item[--]
The grapheme process $\CG=(\CG_t)_{t \geq 0}$ arises as a functional of the $\U_1$-valued process $\CU=(\CU_t)_{t \geq 0}$, respectively, the $\U$-valued process, and can be characterised by a martingale problem.
\item[--]
Standard tools can be used to show that $\CG$ has the \emph{strong Markov} and the \emph{Feller property}, and has \emph{continuous} paths.
\item[--]
\emph{Convergence} of the finite grapheme processes $\CG^n$ to the continuum grapheme processes $\CG$ on $\G^{[]}$ can be proved.
\item[--]
The identification of the \emph{dual} process as coalescent driven graphemes can be obtained as well.
\item[--]
The \emph{genealogical information} encoded in the space in which the graph is embedded can be used to get information on the process $\nu=(\nu_t)_{t \geq 0}$, inducing a process $\CG^\uparrow$ with values in $\G^{\langle\rangle}$.
\item[--]
Explicit formulas can be derived for \emph{equilibria} in terms of certain classical key distributions in population dynamics and statistics.
\end{itemize}

%%%

\subsection{Main theorems: Grapheme dynamics}
\label{sss.theorem}

We now state our core results, showing how finite graph dynamics scale to continuous processes with desirable properties. Section~\ref{sss.four} states four theorems plus corollaries that concern the evolution rules in class (I) in Section~\ref{sss.compcase}. They are formulated in such a way that we can easily modify them to get similar statements for \emph{all} the models treated in this paper. Section \ref{sss.stateprop} discusses state space properties that are important for the construction. Section \ref{sss.conex} provides extensions of the main theorems to the evolution rules in class (II) in Section~\ref{sss.compcase}.

The key objects appearing in the theorems will be defined, justified and explained in Sections~\ref{s.evol}--\ref{s.duality}, and rely on the theory of \emph{genealogy-valued diffusions} built up in \cite{GrevenPfaffelhuberWinter2013,DGP12,DG19}, which we explain as we go along. The proofs of the theorems are given in Sections~\ref{s.state}--\ref{s.compo}. By bringing the \emph{genealogies} explicitly into focus, we can extend and generalise the population-dynamic approach advocated in \cite{AHR21}. 

%%%

\subsubsection{Four theorems}
\label{sss.four}

We consider the $n \to \infty$ limit dynamics arising from the evolution rules (I) for the $\G^{[]},\G^{\langle\rangle}$-valued versions and for the $V$-marked $\G^{[],V},\G^{\langle\rangle,V}$-valued versions. We show below that this dynamics is the process $\CG$ alluded to earlier.

%%%

\paragraph*{Initial states.}

To properly understand the scope and implications of our grapheme evolution theorems, it is essential to define the \emph{admissible initial states} for our grapheme evolutions, and \emph{admissible initial laws} supported on these initial states. We introduce the sets of graphemes
\begin{equation}
\label{e849}
\wt \G_{\rm comp}^\ast, \wt \G_{\rm ultr}^\ast \text{ with } \ast = [],\{\},\langle\rangle \text{ and } \sim \text{ standing for `with bar' or `without bar'},
\end{equation}
i.e., all graphemes with \emph{non-degenerate} (= positive measure) \emph{completely connected components} that form a subset where $h(i,j)$ is always $1$ or $0$ when $i$ or $j$ are not in the subset, respectively, with possible embedding in a specific class of Polish metric spaces, namely, those with an ultrametric, i.e., in an element of $\U_1$ or $\U$ depending on $\sim$. Note that $\wt \G^\ast_{\rm comp}$ is an invariant set for our evolution mechanisms (except the one where we add insertion/deletion), but is not closed topologically. For $\wt \G_{\rm ultr}^{[]}$ we require $h$ to be \emph{a coarser ultrametric than $r$}. We note that $\wt \G^{[]}_{\rm ultra}$ is a Polish space that is dynamically closed and
\begin{equation}
\label{e862}
\wt \G^\ast_{\rm comp} \subseteq \wt \G^\ast_{\rm ultr}.
\end{equation}

\begin{remark}[General initial states]
\label{r.938}
{\rm The restriction of the initial state is needed but is harmless, at least for the classes of dynamics we consider (as we will show in Section~\ref{s.nccomp} but not formulate here). From every $\CG \in \G^{[]}$ that can occur as limit of our finite-graph dynamics the process enters at \emph{infinite} rate into the subset of restricted states, jumping from $\wt \G \setminus \G_{\rm comp}$ to the dust state for $h$ (preserving its $\wt \G^\ast_{\rm comp}$ component in the transition), and remains there. Indeed, only finitely many initial vertices have descendants at time $t >0$, all initial edges have been replaced at time $t>0$, and the process is in $\wt \G^\ast_{\rm comp}$. Only in such a state can edges be preserved for the future, since they appear for macroscopic sets of vertices (see Remark~\ref{r.968} below), as can be seen from the approximation with finite grapheme dynamics, or directly from the martingale problem by duality. Hence, we have here exactly the right state space for our dynamics.}\hfill$\spadesuit$
\end{remark}

In Sections~\ref{s.evol}--\ref{s.duality}, we will explain at length the various ingredients that appear below in the statemens on the grapheme diffusion.

We need the concept of \emph{order} for operators that are not classical differential operators.

\begin{definition}[Order of operators]
\label{def.1947}
{\rm An operator $\Omega$ is called first order if $\Omega(\Phi^2)-2\Phi \Omega \Phi$ $= 0$ for all $\Phi \in D(\Omega)$ (= the domain of $\Omega$), and is called second order if it is not first order and $\Omega(\Phi^3)-3 \Phi \Omega (\Phi^2)-3\Phi^2 \Omega (\Phi) = 0$ for all $\Phi \in D(\Omega)$. (This is the same algebraic characterisation as for first-order or second-order differential operators acting on twice-differentiable functions on $\R$.) See~\cite{DGP12} for more details and examples.  This algebraic definition extends the notion of order beyond the usual context of differential operators.} \hfill$\spadesuit$
\end{definition}

\begin{theorem}[Grapheme diffusions]
\label{th.963-I}
$\mbox{}$\\
\textup{(a)} The process $\CG$ defined in~\eqref{e684} and $\CG_0 \in \wt \G^\ast_{\rm ultr}$ is the unique solution of the
\begin{equation}
(\wh\CL^\ast,\wh\Pi^\ast,\wh\Gamma^\ast)\text{-martingale problem},
\end{equation}
where $\wh\CL^\ast$ (the generator) is the extension of the operator $\CL$ corresponding to the $\U_1$-valued processes, respectively, the $\U$-valued process lifted to $\G^{[]}$, $\wh\Pi^\ast$ (the domain) is the algebra of test functions on $\G^{[]}$, and $\wh\Gamma^\ast \in \CM_1(\G^{[]}_{\rm ultr})$ is the initial law. Furthermore:
\begin{itemize}
\item
The process $\CG$ a.s.\ has paths with values in $\wt \G^{[]}_{\rm comp}$ for $t >0$ and $\wt \G^{[]}_{\rm ultr}$ for $t \geq 0$.
\item
For positive times the grapheme connection function of a representative of $\CG$ a.s.\ is continuous and has completely connected components.
\item
The grapheme process $\CG$ has a dual (see Theorem~\ref{th.853}).
\end{itemize}
\textup{(b)} The $\G^{\langle\,\rangle}$-valued process denoted $\CG^\uparrow$, induced by the process $(\nu_t)_{t \geq 0}$ via the functional $\nu$ arising from the $\G^{[]}$-valued process $\CG$, is the unique solution of the (compare here Remark~\ref{r.964})
\begin{equation}
(\CL^\ast,\Pi^\ast,\Gamma^\ast)\text{-martingale problem}
\end{equation}
on the closure of $\G^{\langle\rangle}_{\rm comp}$, where $\CL^\ast$ is the operator induced by $\wh \CL^\ast$ on $\Pi^\ast$, which is a subalgebra of $\wh \Pi^\ast$ of test functions identifying the equivalence class in $\G^{\langle\,\rangle}$ of $\CG$, and $\Gamma^\ast$ is the initial law induced by $\Gamma$.\\
\textup{(c)} $\CG$ and $\CG^\uparrow$ are strong Markov processes with the Feller property (both w.r.t.\ their natural filtration) and with continuous paths, and $\wh \CL^\ast$ and $\CL^\ast$ are second order operators\footnote{Definition~\ref{def.1947}}, (i.e., $\CG$ and $\CG^\uparrow$ are diffusions). \qed
\end{theorem}

\begin{corollary}[State properties]
\label{cor.932}
$\CG^{a,c,\theta}$ moves in positive time from every initial state (including sparse states) to a state that has finitely many ($c=0$) or infinitely countably many ($c>0$) completely connected components. In the former case, the number of components is random with a distribution that can be identified with the help of duality. \qed
\end{corollary}

The process $\CG$ arises as a limit of finite-grapheme evolutions induced by the finite-graph evolutions of Section~\ref{sss.compcase}, which justifies its operator $\wh \CL^\ast$ as follows.

\begin{theorem}[Grapheme approximations]
\label{th.963-II}
$\mbox{}$\\
For $n\in\N$, let $\CG^{a,c,\theta;(n)}$ (with $a=d$ or $a=bn$ for $FV$, respectively, $DW$), be the finite-grapheme process starting with $n$ vertices and with a uniform sampling measure, and evolving according to the rules specified above. If, for $\CG_0^{a,c,\theta} \in \CM_1(\wt \G^{[]}_{\rm ultr})$,
\begin{equation}
\mathrm{LAW}\left[\CG_0^{a,c,\theta; (n)}\right] \Tno \mathrm{LAW}\left[\CG_0^{a,c,\theta}\right],
\end{equation}
then
\begin{equation}
\label{e978}
\mathrm{LAW}\left[(\CG_t^{a,c,\theta;(n)})_{t \geq 0}\right] \Tno
\mathrm{LAW}\left[(\CG_t^{a,c,\theta})_{t \geq 0}\right] \text{ in $C ([0,\infty),\wt \G^{[]})$}
\end{equation}
with $\wt\G=\G$, respectively, $\bar \G$, where $\Tno$ denotes weak convergence in $D([a,\infty),\wt \G^{[]})$, the space of c\`adl\`ag paths. \qed
\end{theorem}

\begin{corollary}[General initial states]\label{cor.1013}
If the process starts in $\G^{[]}$ outside the closure of $\G^{[]}_{\rm comp}$, then convergence still holds with an initial value given by the completely connected component part and given by the $h$-dust case (i.e., $h \equiv 0$) with the remaining weight, so that an instantaneous jump in $h$ at time $t=0$ occurs.
\end{corollary}

%%%

\paragraph*{Long-time behaviour and equilibrium.}

Next, we discuss the \emph{long-time behaviour}. Here the problem arises that the ultrametric spaces $(U_t,r_t)$ degenerate because pairs of individuals without a most recent common ancestor have distance $2t$, which diverges as $t \to \infty$. This means that the equilibrium state may have distances equal to $\infty$, and so we have to extend $\U_1$ and $\U$, or we have to transform the ultrametric $r_t$ to $1-\e^{-r_t}$ (which is a \emph{new ultrametric} for which distance $\infty$ becomes distance $1$). Therefore we now pass to transformed states, which we indicate by writing $\mathrm{LAW}^+$ instead of $\mathrm{LAW}$. Note that, viewed in $\G^{\{\}}$ or $\G^{\langle\rangle}$, our states do \emph{not} change: they remain equivalent to the untransformed states.

\begin{theorem}[Grapheme equilibria]
\label{th.963-III}
$\mbox{}$\\
(a) The grapheme dynamics with $\CG_0^{a,c,\theta} \in \G^{[]}_{\rm ultr}$ converges to a unique equilibrium in $\CM_1(\wt \G_{\rm ultr})$, with $\sim$ denoting bar or not bar, i.e.,
\begin{equation}
\label{e970}
\mathrm{LAW}^+ \left[\CG^{a,c,\theta}_t \right] \Tto \nu^{a,c,\theta},
\end{equation}
with $a=d$ or $b$ for $FV$, respectively, $DW$. For $c>0$ this equilibrium is non-trivial and has completely connected components, while for $c=0$ it is trivial and equal to $\delta_{K}$ with $K$ the complete graph.\\
(b) The equilibrium on $\G^{[]}$, respectively, $\bar \G^{[]}$ is determined by the entrance law of the dual process, respectively, the conditional dual process, run for an infinite time (see Theorem~\ref{th.853} and Remark~\ref{r.cond}).\\
(c) The equivalence class of the frequency vector (w.r.t.\ permutations) of the different completely connected components is a random element in $\ell_1(\N)$ whose law
\begin{equation}
\label{e1444}
\Gamma^{a,c,\theta}, \quad a=b,d,
\end{equation}
can be computed (see Section~\ref{sss.proc}) and uniquely identifies the law of the connection-matrix distribution in the equilibrium of $\CG^\downarrow$ induced by $\nu$ on $\G^{\langle\rangle}$. \qed
\end{theorem}

In the sequel, we need the following extension.

\begin{corollary}[Marked version]
\label{cor.904}
The statements of Theorems~\ref{th.963-I}, \ref{th.963-II} and \ref{th.963-III} also hold for the $V$-marked version of the processes on $\G^{[],V},\G^{\langle\rangle,V}$. Here, the trivial equilibrium is $\int_V \theta(\dd v) \delta_{K_v}$ with $K_v$ the complete graph whose vertices all have the same type $v \in V$. \qed
\end{corollary}

%%%

\paragraph*{Duality.}

The \emph{dual process} of $\CG$ in the $\G$-valued case is indispensable to understand the behaviour of $\CG$. This process lives on $\G^{\{\}},\G^{[]}, \G^{\langle\rangle}$ and is \emph{driven} by a \emph{Kingman coalescent process} $C^t=(C^t_s)_{s \in [0,t]}$ with $t$ the time horizon for the duality. This coalescent is a process of partitions of $\N$, where pairs of partitions merge at rate $d$, respectively, $b$, and jump to a cemetery $S$ at rate $c$ (in the cemetery coalescence is suspended), where a location is chosen according to $\theta \in \CM_1(S)$. This coalescent leads to a sequence of simple Markov jump processes when $\N$ is replaced by $[n]$, which form a \emph{consistent family} in $n$. By a result of Kingman, we know that the entrance law  from the state $\N$ exists as a projective limit. We equip $\N$ with the ultrametric by setting the distance between $i$ and $j$ to twice the time they need to enter the same partition element before time $t$, respectively, $2t$ if that does not happen by time $t$. Taking the completion to get an ultrametric space $\wt U_t$, using the equidistribution on $\wt U_t$ (see~\cite{GrevenPfaffelhuberWinter2009} for the construction of the $\U_1$-valued coalescent), and using \eqref{e678} to define $h$ and taking equivalence classes, we get the \emph{dual grapheme} $(\CC_s^t)_{s \in [0,t]}$ process for time-horizon $t$ and the dual grapheme $\CC_t^t$ by evaluating it at time $t$, written shortly as $\CC_t \in \G^{[]}$.

In the case of $\bar \G$-valued processes the situation is different. Here the dual has values in $[0,\infty) \times \G^\ast, \ast = []$ or $\{\},\langle\rangle$, where a number in $[0,\infty)$ given by an exponential functional of the coalescent path is added. We have the same dual process for the $\G^\ast$-component as before, and the $[0,\infty)$-component is an exponential of a path integral of the functional of the coalescent, leading to a Feynman-Kac duality, and the collection of the finite dual processes is \emph{no longer consistent}. At the end of Remark~\ref{r.cond} we describe a way to circumvent this obstacle.

Furthermore, the \emph{duality functions $H$} on $E \times E^\prime$, where $E,E^\prime$ are the state spaces of the process, respectively, the dual process, are based on polynomials (and use \emph{finite} coalescents) and on a Feynman-Kac potential $\mathfrak{V}:E^\prime \to \R^+$, given by the number of dual pairs (see Section~\ref{s.duality} for detailed definitions of $H$ and $\mathfrak{V}$). This is a duality for the underlying genealogy process $\CU$, as well as for the functional $\nu^{(\infty)}$ of connection-matrix distributions.

\begin{theorem}[Duality]
\label{th.853}
$\mbox{}$\\
(a) For $\CG^{d,c,\theta}$ the following duality relation with duality function $H$ and dual process $\CC$ holds:
\begin{equation}
\label{e855}
\E \left[H\left(\CG_t^{d,c,\theta},\CC_0\right)\right]
= \E \left[H\left(\CG_0^{d,c,\theta},\CC_t\right)\right].
\end{equation}
(b) For $\CG^{b,c,\theta}$ the Feynman-Kac duality with potential $\mathfrak{V}$ holds:
\begin{equation}
\label{e859}
\E \left[H\left(\CG_t^{b,c,\theta},\CC_0\right)\right]
= \E\left[H\left(\CG_0^{b,c,\theta}, \CC_t \right) \exp\left(\int_0^t  \dd s\, \mathfrak{V}(\CC_{t-s})\right)\right].
\end{equation} \qed
\end{theorem}

\noindent
Theorem~\ref{th.853} says that we can calculate the expectation in the left-hand side (which determines the law of $\CG_t^{b,c,\theta}$ uniquely) in terms of the Markov jump process $(\CC_t)_{t \geq 0}$, together with the distance growth that drives the grapheme process in the right-hand side.

In fact, we can strengthen this result to a \emph{strong} duality in the Fleming-Viot case. Here, a subtlety arises because we need the concept of \emph{pasting} of graphemes based on pasting one ultrametric space to another one (symbolic $\vdash$) in order to modify $(\wt \CU_t,\wt r_t)$ to incorporate \emph{general} $\CU_0=[U_0,r_0,\mu_0]$ that are \emph{different} from the \emph{single-root case}. The pasted $\wt r_t \vdash$ arises by adding $r_0$ and $\wt r_t$, and taking $\wt r_t \vdash (\iota,\iota^\prime) = r_0(\tau(\iota),\tau(\iota^\prime))+ \wt r_t(\iota,\iota^\prime)$, where $\tau$ is the matching of $\N$ with the $\mu_0$-sample sequence, respectively, its piece up to the number of partition elements at time $t$.

\begin{corollary}[Strong duality]
\begin{equation}
\label{e1094}
\CL\left[\CG_t^{d,c,\theta}\right] = \CL\left[\CG_0 \vdash \CC_t^{d,c,\theta,\ast}\right].
\end{equation}
\qed
\end{corollary}

\begin{remark}[Conditional duality]\label{r.cond}
{\rm For $\CG^{b,c,\theta}$ the situation is more subtle, but we get a similar structure when we condition on the full size process $(\bar \mu_t)_{t \geq 0}$. Conditioned on the latter path, the process of the $\U_1$-valued component of the state can be represented by a coalescent with coalescence rate $b(\bar \mu_{t-s})^{-1}$ and jump rate $c(\bar \mu_{t-s})^{-1}$ at the backward time $s$ of the coalescent (see~\cite{DG19} for more details).}\hfill$\spadesuit$
\end{remark}

%%%%%%%%

\subsubsection{Discussion of state space properties}
\label{sss.stateprop}

\begin{remark}[Relation of component decomposition to Fleming-Viot]
{\rm To understand the structure of $\nu_t$ we note the following. The distribution of the stochastic process of the \emph{size-ordered} vectors $((\mu_t(A_i))_{i \in \N}$, with $(A_i)_{i \in \M}$ the completely connected components, only depends on the equivalence class of $(U_t,r_t,h_t,\mu_t)$ and can be characterised by the Fleming-Viot \emph{measure-valued} diffusion on $\CM_1([0,1])$ at rate $d$, with emigration/immigration at rate $c$ from a source $\theta \in \CM_1 ([0,1])$ (see~\cite{DGV95} and Section~\ref{sss.proc}). To make this connection, it is convenient to think of the marked setting: each component is given a mark, drawn independently from $[0,1]$ according to $\theta$, and this type is inheritable and assigned at the immigration time. The weights of the types at time $t$ uniquely specify an atomic measure, and hence an equivalence class of vectors corresponding to the unmarked case.}\hfill$\spadesuit$
\end{remark}

\begin{remark}[Fisher-Wright, $N$-type Feller]\label{r.1200}
{\rm If we modify the dynamics such that vertices have one of $N$ fixed inherited types and are connected whenever two vertices have the same type, then we get $N$-type Fisher-Wright models or the $N$-type Feller branching model, and we get models described by $N$-dimensional diffusions, to which our theory applies. However, these cases can be also treated in a simpler way, based on $N$-dimensional diffusion theory. Our approach is tailored for the cases with $N=\infty$.}\hfill$\spadesuit$
\end{remark}

The process $\CG^\uparrow$ has a somewhat delicate mathematical structure and does not fit so well into the theory of general Markov processes. We explain why in two remarks.

\begin{remark}[Initial states]\label{r.968}
{\rm The initial values appearing for $\CG$ on $\G^\uparrow$ for $\nu_0$ are the elements where we have finitely many, respectively, countably many components characterised by a probability measure on $\N$ with a finite, respectively, countable support (without loss of generality chosen to be $[n]$ or $\N$), or we have the grapheme given by $([0,1],H \equiv 0$, Lebesgue measure), i.e., all vertices are on their own. In these cases we have \emph{continuous paths}. If the process starts outside $\G^{[]}_{\rm comp}$ and not in the single-root case, then it instantaneously jumps into the latter state at time $0$.}\hfill$\spadesuit$
\end{remark}

\begin{remark}[Difficulties and properties of the $\G^{\langle\rangle}$-martingale problem]\label{r.964}
$\mbox{}$\\
{\rm (1) The martingale problem, (a) on all of $\G^{[]}$, (b) on \emph{all} of $\wt \G^{\langle\rangle}$ or $\G^{[]}$, would be \emph{non-standard}, since the space $\G^{\langle\rangle}$ is not closed and hence is not Polish. However, it is a subset of $\wt \CW$, the space of graphons, which is a Polish space (even a compact Polish space) in the chosen topology. Therefore we have a legitimate stochastic process, taking values in a Polish space, but the martingale problem has to be extended to general graphons (which form a Polish space) in order to become a martingale problem of the standard form. We would need to show that the paths stay in an invariant subspace of $\G^{\langle\rangle}$, which all the solutions of the martingale problem enter. Still, not all the standard conclusions from the theory of Markov processes need to hold. Through the $h$-component, we get similar effects in all of $\G^{[]},\G^{\{\}}$, and the closure would be a space where $h$ becomes a graphon, say, $\wh \CW^{[]},\wh \CW^{\{\}}$.\\
(2) For our dynamics the peculiarity arises that, in any positive time, the path moves from points in $\G^{\langle\rangle}$ outside of $\G^{\langle\rangle}_{\rm comp}$ to points inside the \emph{closure of} $\G^{\langle\rangle}_{\rm comp}$. The latter is only a small subset of the state space, but it turns out that this subset is entered because the cliques in finite graphs take over on \emph{smaller} time scales. There is only one entrance law from states outside, namely, those with non-degenerate completely connected components with continuous paths, or starting in a point mass, and $h \equiv 0$ when the process starts outside $\G_{\rm comp}$. In particular, we can conclude from our results that there is no extension of the entrance laws to all points in the graphon space that are not in $\G^{\langle\rangle}$ or not in $\G^{[]}_{\rm comp}$, nor to the dust case $[\{1\},\underline{0},\underline{0},\delta_1]$ with continuous paths, only c\`adl\`ag paths and instantaneous jumps at $t=0$ into the closure of $\wt \G_{\rm comp}$. Furthermore, we cannot get a standard martingale problem that for $t>0$ has values in $\G^{\langle\rangle}$, since the latter is not closed and the rate to jump is infinity off $\wt \G_{\rm comp}$ or for the single-root case.\\
(3) The only option would be to introduce a stronger topology that makes $\G^{\langle\rangle}$ closed without interfering with the equivalence relation. This is done by requiring the limit to have a $\{0,1\}$-valued $h$, but then we still have to allow for \emph{infinite} jump rates on part of the space in the martingale problem. One way to construct a solution despite these deficiencies would be to use the graphon space $\wh \CW$, which is compact, verify that the path of the process stays in $\G^{\langle\rangle}$, and use the duality to construct the process (see~\cite{depperschmidt2023duality}) by verifying that for $t \geq 0$ the path is in the closure of $\wt \G_{\rm comp}^{\langle\rangle}$.} \hfill$\spadesuit$
\end{remark}

%%%

\subsection{Consequences and extensions of main theorems}
\label{sss.conex}

In Section~\ref{sss.link} we defined $\G^{\{\}, V}$-valued and $\G^{\langle\,\rangle,V}$-valued processes $\CG^{c,d,\theta}$ and $\CG^{c,b,\theta}$, which we abbreviate as $\CG^{\rm FV},\CG^{\rm DW}$ in the sequel. We have the characterisation of their path properties and their long-time behaviour as stated in Theorems~\ref{th.963-I}, \ref{th.963-II} and \ref{th.963-III}. We proceed with some supplementary results. In Section~\ref{sss.proc} we give a detailed description of the \emph{equilibrium process}, including a \emph{phase transition} in the path behaviour with respect to the parameter $c/a$, where $a=b,d$, and extract the two key equilibrium statistics of completely \emph{connected component's age and size}. Here the size depends on the functional $\CG^{\langle\rangle}$ only, but the age depends explicitly on the $[U_t,r_t,\mu_t]$-component, in particular, on $r_t$. In Section~\ref{ss.difftyp} we extend the results to \emph{type-biased resampling} (referred to as \emph{selection}) and \emph{type-mutation}, leading to rewiring of edges. In Section~\ref{ss.allow} we turn to the dynamics without completely connected components (due to removal and insertion of edges). In Section \ref{sss.poss} we derive explicit representations for the laws of the equilibrium statistics.

%%%

\subsubsection{Equilibrium processes}
\label{sss.proc}

%%%

\paragraph{Geometric description of the equilibrium process: diffusing partitions.}

The embedding in $\CI$ helps us to better understand the \emph{equilibrium process} of $\CG^{a,c,\theta}$ with $a=b,d$ and $c >0$, because there exists a decomposition into disjoint balls $\CH^i_t$ of \emph{completely connected components} for any representative of $\CG_t$ of the equilibrium process $(\CG_t)_{t \geq 0}$ such that the induced composition of $\CI_t \times \CI_t$ has the property
\begin{equation}
\label{e1497}
\CI_t \times \CI_t = \mathop{\biguplus}\limits_{i,j \in \N}\left(\CH_t^i \times \CH^j_t \right),
\qquad
h(u,u^\prime) = \left\{
\begin{array}{lll}
1, &u,u^\prime \in \CH^i, & i \in \N, \\
0, &u \in \CH^i_t,  & i,j \in \N, i \neq j.
\end{array} \right.
\end{equation}
In other words, we have a \emph{diffusion process} of (an equivalence class of) a \emph{countable number of balls} that are completely connected in the grapheme sense and that carry weights charged by $\mu_t$. If we size-order this vector of weights, then we obtain an object that is a function of the equivalence class. This gives us a block-type representation of the grapheme, diffusing over time, from which we can read off the \emph{graphon process} $(\wh \CW_t)_{t \geq 0}$ via the embedding of the grapheme in $[0,1]$, as described in Remark~\ref{r.936}.

Summarising, we look at the equilibrium path and its law, or the equilibrium law of the equivalence class of the weight vectors (represented by the \emph{size-ordered version}):
\begin{equation}
\label{e1506}
\mathrm{LAW} \left[\left(\left(\mu_t\left(\CH^i_t\right)\right)_{i \in \N}\right)_{t \geq 0}\right],
\qquad \mu^\dagger_t= (\mu_t(\CH^i_t))_{i \in \N},  \qquad t \geq 0.
\end{equation}
The \emph{frequencies of edges processes} are given by $(((\mu_t(\CH^i_t))^2)_{i \in \N})_{t \geq 0}$.

Note that \emph{$\mu^\dagger$ determines $\nu^{(\infty)}$ uniquely}. We show that in our two model classes the equilibrium $\CL[\mu^\dagger]$ is given by the \emph{Poisson-Dirichlet distribution} on the unit ball of $\ell_1(\N)$, respectively, by a \emph{standard Moran subordinator} on $\ell_1 (\N)$ introduced in the next paragraph.

%%%

\paragraph{Identification of the equilibrium.}

We are able to characterise the equilibria $\nu^{c,\theta,d}$ and $\nu^{c,\theta,b}$ in terms of the dual process introduced in Theorem~\ref{th.853}, run for an \emph{infinite} time (which is induced by a coalescent, as explained in Section~\ref{s.duality}). This dual process converges to a limit state $\CC_\infty$ as $t \to \infty$, and can be used to construct the equilibrium state of the associated diffusion $(\CU_t^{a,c,\theta})_{t \geq 0}$ of genealogies, denoted by
\begin{equation}
\label{e.1160}
\text{$\CU_\infty^{a,c,\theta}$ with values in $\U_1$, respectively, $\U$},
\end{equation}
depending on whether we consider Fleming-Viot or Dawson-Watanabe dynamics.

In the \emph{Fleming-Viot} case, this can even be done via a Kingman coalescent started in countably many dual individuals and state $(\{n\}, n \in \N)$ run for an infinite time, for which we define a weighted ultrametric measure space giving an element of $\U_1$ via completion (also called strong duality). We can form the grapheme with that space by connecting all elements that are in the same partition element.

In the \emph{Dawson-Watanabe} case, the representation is more subtle: we have to take the total mass and represent the normalised state, which can again be treated as a time-\emph{in}homogeneous Fleming-Viot process that has a representation via a time-inhomogeneous coalescent (see~\cite{DG19}).

In fact, for both we can identify the \emph{law of $\nu^{(\infty)}$} of the equilibrium, which determines what we can say about the \emph{equilibrium in $\G^{\langle\rangle}$}, and contains information on the equilibrium in $\G^{[]},\G^{\{\}}$. This is done by identifying $\CL[\mu^\dagger]$ in \eqref{e1506}: in the Fleming-Viot case with the help of the \emph{Poisson-Dirichlet distribution}, in the Dawson-Watanabe case with the help of the \emph{Moran-gamma-process}, which are both recalled below.

For convenience of notation we work on the state space $\G^{[],V}$ (recall the paragraph after Remark~\ref{r.964}). Consider the law of $\wh \Gamma^{{\rm FV},d,c,\theta}$ on $\CM_1([0,1])$ given by
\begin{equation}
\label{e1160}
\mathrm{LAW} \left[ \sum_{i \in \N} V_i \left(\prod_{j=1}^{i-1} \left(1-V_j\right)\right) \delta_{(U_i)} \right],
\end{equation}
where $(U_i)_{i \in \N}$ are i.i.d.\ with distribution $\theta$ and \emph{$(V_j)_{j \in \N}$ are i.i.d.\ with distribution $B(1,\beta^{-1})$} with $\beta= \frac{c}{d}$. The law $\wh \Gamma^{{\rm FV},d,c,\theta}$ is characterised as the unique equilibrium of a measure-valued process, the $\CM_1([0,1])$-valued Fleming-Viot process with resampling at rate $d$ and with immigration and emigration at rate $c$ from the source $\theta \in \CM_1 ([0,1])$ (for details see~\cite[Section 2(a) and 2(b)]{DGV95}). The expression in \eqref{e1160} induces a vector $(V_i \prod_{j=1}^{i-1} (1-V_j))_{i \in \N}$ of weights whose law is called the \emph{Poisson-Dirichlet distribution} with parameter $(\theta, \frac{d}{c})$ on $\ell_1(\N)$ and is denoted by $\Gamma^{{\rm FV},d,c,\theta}$. This is also the law of $\mu^\dagger_t$ in the equilibrium of $\CG^{c,\theta,d}$, and characterises the law of the functional $\nu^\infty$ of the equilibrium in $\G^{[]}$, respectively, characterises the full equilibrium in $\G^{\langle\rangle}$.

\begin{corollary}[Identification of key statistics in the equilibrium $\nu^{d,c,\theta}$]
\label{cor.1165}
$\mbox{}$\\
Assume that $c >0$. Then, under $\nu^{d,c,\theta}$, the law of $\nu^{(\infty)}$ is given via the distribution $\Gamma^{{\rm FV},d,c,\theta}$ and determines the law $\mu^\dagger$ of the size-ordered weights of the completely connected components.
\qed
\end{corollary}

There is a similar identification in the Dawson-Watanabe case, based on the equilibrium of a $\CM_{\rm fin}([0,1])$-valued Dawson-Watanabe process on $[0,1]$ with immigration and emigration at rate $c$ from the source $\theta \in \CM_{\rm fin}([0,\infty)^2)$, whose equilibrium is based on the \emph{Moran-gamma-subordinator} (see~\cite{DG96}). It is known that the multitype branching process with type space $[0,1]$ and with immigration from the source $\theta \in \CM_{\rm fin}(\R^2)$ at rate $c$ and emigration at the same rate $c$ has a unique equilibrium $\wh \Gamma^{b,c,\theta}$ of the form (see \cite{DG96})
\begin{equation}
\label{e1253}
\mathrm{LAW} \left[\sum_{i\in\N} W_i \delta_{(U_i)}\right],
\end{equation}
where the $(W_i,U_i)_{i \in \N}$ are the \emph{jump sizes} and \emph{jump times} of the \emph{Moran-gamma-process} up to ``time $\wt \theta$", with $\wt \theta=\theta([0,1])$. This is an \emph{infinitely divisible $[0,\infty)$-valued process} with $\theta([0,1])$ \emph{as time index}, which we can characterise uniquely by specifying its \emph{L\'evy-measure $\Lambda$ on $[0,\infty]$}. Indeed, define
\begin{equation}
\label{e1258}
\Lambda(\d x) =x^{-1}\,\e^{-\frac{c}{b}x} \d x, \quad x \in [0,\infty),
\end{equation}
and use this as the \emph{L\'evy-measure} of the infinitely divisible process, which is called the \emph{standard Moran-gamma-process} on $[0,\infty)$ with parameter $c/b$.

\begin{corollary}[Identification of key statistics in the equilibrium $\nu^{b,c,\theta}$]
\label{cor.1263}
Under the law $\nu^{b,c,\theta}$ on $\bar \G^{[]}$ (respectively, $\bar \G^{[],V}$), the law $\CL[\mu^\dagger_t]$ in \eqref{e1309} is given by $\Gamma^{b,c,\theta}$. \qed
\end{corollary}

%%%

\paragraph*{Phase transition for equilibrium paths: hitting complete graphs.}

The evolution of the frequencies of completely connected components in the equilibrium process has interesting features that can be captured by using the embedding of graphs and graphemes in a specific metric Polish space representing an equivalence class of an ultrametric measure space. In what follows we consider a statistic that describes a \emph{property} that can be read off from the \emph{functional $\nu^{(\infty)}$ projected from $U \times V$ to $U$} (and hence from the $\G^{\langle\rangle}$-valued process), namely, the property that the $\ell_1(\N)$-valued path of $(\mu_t^\dagger)_{t \geq 0}$ (recall~\eqref{e1512}) hits finitely many (or even a single) completely connected components.

\emph{Size-order} $(\CH^i_u)_{i \in \N}$ at a given time $u=s$, and observe the vector of weights at time $u=s+t$. If $\theta$ is diffusive, i.e., non-atomic, then for each entry there is a partition into continua and their weights follow a diffusion process for which we can define \emph{hitting times} by
\begin{equation}
\label{e1512}
\left(\mu_\infty\left(\CH^i_{s+t}\right)\right)_{t \geq 0} \text{ hits } 0 \text{ at time $t=T_i$},
\end{equation}
and ask whether, for $i \geq 2$, $T_i < \infty$ a.s.\ or $T_i=\infty$ a.s. The radius of the $(i)$-most charged ball grows like $2t$ as $t$ moves through $[0,\infty)$. At these exceptional times, if finite then the grapheme hits a state that has only finitely many components, namely, $i-1$ completely connected components for $i \geq 2$.

We have a \emph{phase transition} in the parameter $c/a$, with $a=b,d$, for the path property of the equilibrium path  \emph{to leave the state of countably many completely connected components} and hit the state of a \emph{single completely connected component}.

\begin{proposition}[Phase transition of path property: $\CG^{a,c,\theta}$, $a=b,d$]
\label{prop.1078}
$\mbox{}$\\
\textup{(a)} Depending on whether $c/d \geq 2$ or $c/d <2$, $T_i=\infty$ a.s., respectively, $T_i<\infty$ a.s.\\
\textup{(b)} The connected component $1$ may or may not hit the frequency $1$. If yes, then $T_i<\infty$ a.s.\ for all $i \geq 2$. If not, then $T_i=\infty$ a.s.\ for all $i \geq 2$, and there are only states with a countably infinite number of different completely connected components. In the former case there is a positive length random time interval during which hitting the complete graph occurs infinitely often, but the time of countably many components is a union of time intervals (the excursions in the countable regime) with full Lebesgue measure. In the latter case, paths have countably many distinct completely connected components.  \qed
\end{proposition}

%%%

\subsubsection{Extension: grapheme diffusions with mutation and/or selection of types}
\label{ss.difftyp}

In this section, we investigate what happens when we pass to the processes treated in Section~\ref{sss.theorem}, i.e., to evolution rules containing mutation and selection. In particular, we now consider $\G^{[],V},\G^{\langle\rangle,V}$ for a suitable mark space $V$.

First, we consider the case $m=0$. We use the marked version of the Fleming-Viot evolution rule, i.e., we \emph{connect vertices with the same type}, which characterises the most recent common ancestor of a connected component, and add \emph{selection} at rate $s$ with respect to types based on a fitness function $\chi$ (see \eqref{e857} and the sequel for formulas). Then we can use the same formula as before to define $\CG^{a,c,\theta,0,s}$ from the $\U^V_1$-valued process $\CU^{a,c,\theta,0,s}$. We add \emph{mutation} of types at rate $m$ and jump probabilities $M(u,dv)$ for a mutation from $u$ to $v$, where we distinguish between the cases where $M(u,\cdot)$ is diffusive, respectively, atomic. This addition requires serious changes, namely, we have to now define $\CG$ from $\CU$ by setting $h$ equal to $1$ for equal types, and we have to consider $\G_{\rm comp}^{[],V},\G_{\rm ultra}^{[],V}$ based on the decomposition of $U \times V$ (rather than just $U$) in union of balls in the $U$-component and single values in the $V$-direction, i.e., a \emph{marked subfamily decomposition}. With these adaptations the structure is similar.

In the cases $m>0$ and $m=0$, this gives us \emph{marked grapheme diffusions} $(\CG_t^{d,c,\theta,m,s})_{t \geq 0}$ and $(\CG_t^{d,c,\theta,s})_{t \ge 0}$, for which results analogous to Theorems~\ref{th.963-I}, \ref{th.963-II} and \ref{th.963-III} hold, based on the knowledge we have for the underlying $\U^V_1$-valued diffusions (see~\cite{DGP12}).

\begin{theorem}[Grapheme diffusions with selection and mutation]
\label{th.1838}
$\mbox{}$\\
All the statements in Theorems~\ref{th.963-I}, \ref{th.963-II} and \ref{th.963-III} hold for $\CG^{d,c,\theta,m,s}$ when $m=0$. For $m \neq 0$, Theorems~\ref{th.963-I} and\ref{th.963-II} hold too, but Theorem~\ref{th.963-III} holds if and only if the type-diffusion on $V$, arising by projecting $\mu_{\CI \times V}$ onto $V$, converges to an unique equilibrium. \qed
\end{theorem}

\begin{remark}[Form of equilibria]
\label{r.1180}
$\mbox{}$\\
{\rm (1) No explicit form of the equilibrium law exists for $s,m >0$ when mutation is not state-independent, in which case Gibbs measures appear as equilibria (see~\cite{DGsel14}). In the state-independent case the equilibrium is the same as immigration/emigration with $c=m$ and $\theta(dv)=M(u,dv)$ for all $u \in V$.\\
(2) For the case $s=0,c=0$ with $m >0$ and $M$ such that there exists a measure $\theta$ on the type space with $\theta M= \theta$, the equilibrium has a functional $\nu^\infty$ such that the equilibrium distribution is the Poisson-Dirichlet distribution in Corollary~\ref{cor.1165} when we replace $c$ by $m$. Of special interest are the case where $\theta$ is a diffusive measure, for which we have countably many connected components with random weights, and the case where $\theta$ is finite atomic case, for which we have a finite number of such components. In both cases we have the explicit representation of this random structure from~\eqref{e1160}.\\
(3) For $s=0,m>0$ and $c=0$ with $\theta$ having a countable number of atoms and no diffusive part, there are models related to the \emph{``infinite allele models"} of population genetics, where equilibria for $\G^{\langle\rangle,V}$-valued grapheme diffusions appear for the mass of completely connected components based on the Poisson-Dirichlet distribution.}
\hfill$\spadesuit$
\end{remark}

\begin{remark}[Duality
]\label{r.1215}
{\rm There is also a duality theory for models with selection and mutation via \emph{function-valued} dual processes. The duality relation for the underlying process $\CU$ can be found in~\cite{DGP12}.} There is no simple strong duality. \hfill$\spadesuit$
 \end{remark}

The grapheme diffusions with mutation and selection provide interesting classes of models because a different type of equilibrium appears, but also from a modelling point of view because building in types with different fitness is important to get different rates of expansions of the associated completely connected components. We do not pursue this in more detail here.

%%%

\subsubsection{Grapheme diffusions with equilibria allowing for non-completely connected components}
\label{ss.allow}

In this section, we show that insertion and deletion of edges according to switches between virtual and true edges in the dynamics considered in Section~\ref{sss.theorem} allows for \emph{non-completely connected components in equilibrium}, and even for certain \emph{non-Markovian} evolutions for the $\G^{\langle\rangle}$-valued functional $(\nu_t)_{t \geq 0}$. We begin with the simplest case.

We \emph{prune} the edges in the diffusions that we had in the previous model classes with completely connected components described in Section~\ref{sss.theorem}, to which we refer as the \emph{background diffusion}, for which we have a unique associated solution $\CU$ at hand. For this candidate, we derive a martingale problem. The modified operator of the martingale problem consists of the \emph{second-order terms} considered in the main theorems and their extensions in the marked version, with additional \emph{first-order terms} that account for the addition and removal of edges at rates $a^+$, respectively, $a^-$ in the finite graph evolution (recall Section~\ref{sss.compcase}).

We obtain the $\G^{[]}$-valued process $\CG$ as before from the $\U$-valued, respectively, $\U^V$-valued processes $\CU^{\rm FV},\CU^{\rm DW}$, but only after pruning the edge-function $h_t$ as defined before, in an appropriate way via $\{0,1\}$-valued white noise (pruning). In other words, we first solve the martingale problems for $\CG^{\rm FV},\CG^{\rm DW}$ from Section~\ref{sss.theorem} of the form $([U_t,r_t,h_t,\mu_t])_{t \geq 0}$, which are processes with completely connected components, and then in the solution replace $(h_t)_{t \geq 0}$ by $(\wt h_t)_{t \geq 0}$, where
\begin{equation}
\label{e1188}
\wt h_t = \mathfrak{h}_t h_t \text{ pointwise}
\end{equation}
with
\begin{equation}
\label{e1191}
\mathfrak{h}_t = U_t\text{-indexed $\{0,1\}$-valued white noise},
\end{equation}
generated independently of the path of the underlying background diffusion. Remember, however, that the object arising in the grapheme is the \emph{law of $h_t$} under the sampling measure $\mu_t$. The \emph{connection-matrix distribution} is still a \emph{continuous} functional. However, (i) the operator changes; (ii) the properties of the states assumed by the path change as well. Nevertheless, the fact remains that the state decomposes into countably many $(c>0)$ or finitely many ($c=0$) balls of points, which can potentially be connected and which diffuse. This decomposition is explicitly using $\CU$. But the \emph{effective $\wt h$} is now a \emph{white-noise pruning} of a \emph{continuous function $h$}, and instead of completely connected components we now have only \emph{path-connected} components (paths of edges built with $\wt h$: the skeletons given by sample sequences decompose in that way).

The relation in~\eqref{e1188} defines a $\G^{[],V}$-valued grapheme diffusion
\begin{equation}
\label{e1237}
\text{$^{(a^+,a^-)} \CG^{a,c,\theta}$ with $a=b,d$ or with $(d,s,m)$.}
\end{equation}
The grapheme has the property that the countable graphs arising from a typical sampling sequence of the grapheme have the feature that they are obtained by pruning the edges in a representative of the grapheme, i.e., pruning in the countable graph arising from the background grapheme by realising a $\mu^{\otimes \N}$-distributed sample sequence. We thus find that, indeed, the evolution rules in class (II) in Section~\ref{sss.compcase} can be included in our results.

We can obtain Theorems~\ref{th.963-I}, \ref{th.963-II}, \ref{th.963-III} and~\ref{th.1838} after we adapt the martingale problem as indicated above, obtain a modified dual representation by the dual state of the background diffusion with pruned edges, and the identification of the equilibrium states by the dual of the underlying background diffusion. The grapheme again follows a \emph{diffusion} in $\G^{[]}$, respectively, $\G^{[],V}$.

\begin{theorem}[Grapheme diffusion with non-completely connected components]
\label{th.graphnon}
$\mbox{}$\\
\textup{(a)} With the above modification (i) and (ii) of the operators and the state properties of paths, Theorems~\ref{th.963-I}, \ref{th.963-II}, \ref{th.963-III} and the extension in Theorem~ \ref{th.1838} hold also for the process $\CG$ in \eqref{e1237}.\\
\textup{(b)} There exists a process (the background diffusion) with completely connected components that produces the path-connected components of the process $^{(a^+,a^-)} \CG^{a,c,\theta}$, which for $a^+ \in (0,1)$ are not completely connected. In particular, in every neighbourhood of a point $u$ in the support of $\mu_t$, the density of outgoing edges in equilibrium is $\frac{a^+}{a^+ + a^-}$, and for arbitrary states is given by the marginal of the Markov chain of the spin-flip system on $\{0,1\}$ with rate $a ^+,a^-$.\\
\textup{(c)} In equilibrium the connection-matrix distribution is an i.i.d.\ $\frac{a^+}{a^+ + a^-}$-pruning of a block matrix of $1$'s with block-frequencies given by the Poisson-Dirichlet distribution in~\eqref{e1160} associated with the background diffusion.\\
\textup{(d)} The duality relation of Theorem~\ref{th.853} holds for the dual process in~\eqref{e1094} of the associated process with $a^+=a^-=0$ by taking an i.i.d.\ pruning of the edges with probabilities according to the marginal distribution at time $t$ of the $\{0,1\}^\N$ spin-flip Markov chain with rates $a^+,a^-$. \qed
\end{theorem}

In the proof of Theorem~\ref{th.graphnon}, we formulate a \emph{conditional} martingale problem for a given path of the background process that is the solution of the martingale problem of Theorem~\ref{th.963-I}. Note that the associated graphon $\wh \CW_t$ consists of blocks of fluctuating sizes specified by their weights under $\mu_t$, but the $0$'s and $1$'s occur with frequencies determined by the $a_+,a_-$, and by the time the system has evolved into equilibrium equal $\frac{a^+}{a^+ + a^-}$. Therefore the component of a single type is \emph{not} completely connected, only \emph{path-connected}. Recall that the latter holds also on a \emph{countable Erd\H{o}s-R\'enyi random graph}. Nevertheless, the graphon has values in $\{0,1\}$ and therefore still agrees with the grapheme from our $\wt \G^{[]}$-valued process projected on $\wt \G^{\langle\rangle}$.

Theorem~\ref{th.graphnon} also holds for the extensions to the generalisations given in Section~\ref{sss.compcase}. This holds literally for part (a). Parts (b)-(d) have to be adapted as follows. In (b) we now have a \emph{time-inhomogeneous} spin-flip system in a \emph{random environment} through the dependence on the underlying genealogy process. In (c), the pruning of the equilibrium is now given by random variables $A^+ = a^+(\CU_\infty)$ and $A^- = a^-(\CU_\infty)$ with $\CU_\infty$ the equilibrium genealogy. Part (d) changes analogously to (b).

\begin{corollary}[Extension]\label{cor.1304}
The modified theorem above also holds for the more general flip rates defined in Section~\ref{sss.compcase}.
\end{corollary}

%%%

\subsubsection{Explicit representations of laws of equilibrium statistics}
\label{sss.poss}

There are statistics for graphemes that can be explicitly calculated because they are the laws or important characteristics in models from population dynamics. Return to the models described in Section~\ref{sss.theorem}. We next introduce some quantities that concern the space-time structure of a current completely connected component.

%%%

\paragraph*{Density of edges in equilibrium.}

For the equilibrium distribution on $\G^{\langle\rangle}$ or $G^{\langle\rangle,V}$ in the Fleming-Viot case, the equilibrium induced by $\nu^{\langle\rangle,{\rm FV},c,\theta,d}$ is concentrated on states that decompose into completely connected components, with vertex-frequencies that in \emph{size-ordered} form are given by
\begin{equation}
\label{e1309}
(W_i^\geq)_{i \in \N}
\end{equation}
with $W_i$ given by the Poisson-Dirichlet distribution, i.e., the \emph{size-ordered vector}
\begin{equation}\label{e1330}
\left(V_i \prod_{j=0}^{i-1} \left(1-V_j\right)\right)_{i \in \N},
\end{equation}
which is called the \emph{GEM$(\frac{c}{d})$-distribution}. A key statistics is the frequency of edges among all possible edges:
\begin{equation}
\label{e1314}
\bar E=\sum_{i \in \N} W_i^2.
\end{equation}
This object describes the \emph{relative density} of edges, and measures the deviation from completeness of the graph.

A similar object arises in the Dawson-Watanabe case, where as additional objects we have the \emph{scaled limiting vertex number} and the \emph{vertex numbers of the connected components},
\begin{equation}
\label{e1318}
(W_i)_{i \in \N},
\end{equation}
where $\sum_{i\in\N} W_i < \infty$ is a \emph{$\gamma$-distributed} random variable with parameter $(1,\tfrac{c}{b})$. Define the analogue of the quantity in \eqref{e1314} as
\begin{equation}
\label{e1129}
\bar E = \sum_{i\in\N} \left(\frac{W_i}{\sum_{j \in \N} W_j}\right)^2.
\end{equation}
The challenge is to say more about $\bar E$ using~\eqref{e1330} and \eqref{e1318}.

%%%

\paragraph{Decomposition of edges in equilibrium: space-time structure.}

We now come to properties of the equilibrium process, which depend on the \emph{embedding in an ultrametric measure space} of \emph{genealogies}. In that space, we can define clusters (or cliques) of the set of edges arising from the time-$t$ descendants of a single immigrant, define the completely connected corresponding components, their size at time $t$, and the times of their formation some random time back in the past.

\medskip\noindent
\emph{(i) Decomposition w.r.t.\ immigration time.}
We can decompose the state in the space in which the graph is embedded at time $t$ into the connected components coming from a \emph{specific immigrant} (recall that we are in equilibrium, i.e., in the $t \to \infty$ limit where all individuals  descend from one immigrant only). Each component has a finite diameter, i.e., the supremum of pairwise distances, ordered by size:
\begin{equation}
\label{e1807}
s_1>s_2>\cdots
\end{equation}
The vertices of these components are the descendants of the \emph{immigrants} at times $t-s_1,t-s_2,\ldots$, which have frequencies $\mu_t^{(1)},\mu_t^{(2)},\ldots$ Note that $s_1,s_2,\ldots$ are realisations of \emph{random variables} $S_1,S_2,\ldots$ In particular, all edges that are present at time $t$ date back to an origin later than
\begin{equation}
\label{e1813}
\text{time } t-S_1 \text{, and at time $t$ the frequency of edges is } \sum_{i \in \N} (\mu_t^{(i)})^2.
\end{equation}
In other words, $t-S_1$ is a \emph{renewal point} for the edge structure and $1-\sum_{i\in\N} (\mu_t^{(i)})^2$ is the \emph{deviation from completeness}.

We want to determine the law of the \emph{ages} and \emph{sizes} of the \emph{completely  connected components} in \emph{equilibrium}:
\begin{equation}
\label{e1132}
\CL \left[(t-S_i,\mu^{(i)})_{i \in \N}\right],
\quad \CL \big[(t-S_i)_{i \in \N}\big],
\quad \CL \left[\sum_{i \in \N} (\mu^{(i)}_t)^2 \right].
\end{equation}
The best way to analyse these objects is to work with two approaches, each giving explicit expressions as follows. In the \emph{backward view}, the coalescent driven dual process of the measure-valued process on $\N$ defined by~\eqref{e1506} is a coalescent with jumps into a cemetery at rate $c$ and coalescence at rate $d$, with the cemetery state carrying the measure $\theta$. In the \emph{forward view}, we use the excursion theory of this measure-valued diffusion, which is a \emph{$k$-dimensional diffusion} with a \emph{random $k$}, starting from the zero-state.

\medskip\noindent
\emph{(ii) Decomposition of edges in completely connected components w.r.t.\ the time of their origin.} We can label the completely connected components $k \in \N$ by using the time of their appearance. In each completely connected component we can decompose the edges according to their \emph{origin in time}. Namely, we decompose the component immigrated at time $t-s_k$ for some $k \in \N$ into disjoint balls as follows. First, we take the ball of maximal radius, which is charged by $\mu_t^{(k)}$, and decompose it into disjoint balls of maximal radius less than $s_k$, say $s^1_k$. After that, we continue with the complement within the time $s_u$-ball, to get successively $s_k^2,s_k^3,\ldots$ We denote the corresponding random variables for the $k$-th component by $(S^{k,j})_{j \in \N}$. For every $k \in \N$ and every realisation, this gives decompositions of the edges in the component dating back to time $t-s^2_k,t-s_k^3,\ldots$ As weights $\mu^{(k)},\mu^{(k,1)},\mu^{(k,2)},\ldots$ for the latter random variables, we use $M^{k,j}$ as notation. The frequencies of the edges present after $t-s^j_k$ are obtained by connecting vertices between the \emph{$s_k^j$-balls} and using \emph{weights $\mu^{(k,j)}$}, $k \in \N,j \in \N$. The frequencies of the edges are given by the product of their weights, namely, $\mu^{(k,j)}(1-\mu^{(k,j)})$. Hence, the \emph{space-time structure} of the \emph{edges in the $k$-th component} for $k \in \N$ is described by the collection of random variables with values in $\R^+$, respectively, $[0,1]$, which we can define for the $\G^{[]}$-valued grapheme:
\begin{equation}
\label{e1146}
\big(S^{k,j},M^{(k,j)}\big)_{j \in \N}, \qquad
\big(S^{k,j},M^{(k,j)}(1-M^{(k,j)})\big)_{j\in \N}.
\end{equation}
The second components are also observables for the $\G^{\langle\rangle}$-valued process, but not the first.

This leaves us with the task to calculate the laws for the objects defined in~\eqref{e1132} and~\eqref{e1146} based on duality and excursion theory of diffusions, which is not addressed here.

%%%%%%%%%%%%%%%% SECTION 3 %%%%%%%%%%%%%%%%%%%%%%%%

\section{Grapheme evolution $1$: State space, topology and test functions}
\label{s.evol}

In this section we build up the topology to get a proper \emph{state space} for our grapheme dynamics and state basic properties of the \emph{topology}. This is done by specifying \emph{algebras of test functions} on $\G^\sim$, with $\sim=\langle\rangle,\{\},[]$, that we use to define \emph{convergence of sequences} in these spaces and thus a topology. Later, in Section~\ref{ss.go}, these functions will also be the basis for the \emph{domain of the operator} of our martingale problems. In Section~\ref{ss.space} we focus on the topology, on the test functions, and on tightness issues. In Section~\ref{sss.extens} we extend the setting to varying size, i.e., $\bar \G^\sim$. In Section~\ref{ss.marked} we extend further to $V$-marked versions $\G^{\sim,V},\bar \G^{\sim,V}$.

%%%

\subsection{Test functions, topology and tightness for grapheme spaces}
\label{ss.space}

%%%

\subsubsection{Test functions on grapheme spaces and topology}\label{sss.testf}

We introduce a \emph{topology} on the sets $\G^{[]},\G^{\{\}},\G^{\langle\,\rangle}$ defined in Section~\ref{s.introduction}, in order to specify Borel $\sigma$-algebras $\CB_{\CG^{[]}},\CB_{\CG^{\{\}}},\CB_{\CG^{\langle\rangle}}$ and, if possible, turn them into a Polish space. We do so by defining \emph{convergence of sequences} of elements of the space $\G^{[]},\G^{\{\}},\G^{\langle\rangle}$. This in turn is based on convergence of the sequences of \emph{evaluations} of a specific class of test functions on the spaces $(\CI,\CB_\CI)$ for $\G^{\{\}},(\CI,r)$ for $\G^{[]}$, respectively, $\{0,1\}^{\N \times \N}$ for $\G^{\langle\rangle}$. The proofs of the propositions stated in the present section are given in Section~\ref{s.state}.

\paragraph*{Functions.}
Our functions are based on evaluations for specific functions of points in sampled \emph{finite} subspaces from $\CI$ such that they only depend on the equivalence class of $(\CI,\CB_\CI)$ and on the connection-matrix distribution via $h$.
\begin{itemize}
\item
The algebra of functions, called \emph{polynomials}, arise as functions on \emph{sampled finite subspaces}, with their \emph{edges} obtained by sampling finitely many vertices with $\mu$ from $\CI$ (without replacement if $\mu$ has atoms), forming the \emph{finite version of the space $\CI^\ast$} with the restriction of $h$ defining a \emph{finite} \emph{subgraph} that is \emph{embedded} in $\CI$. If we put the equidistribution on the vertices, then we get a \emph{sampled} \emph{finite grapheme} (see~\eqref{e791},~\eqref{e341} and~\eqref{e403} below). On that grapheme we define functions on a representative of the edges and of the vertices as points in the underlying space $\CI$, which are taken from the finite space built as just indicated. Below we deal separately with two cases: (i) corresponding to $\G^{\langle\rangle}$, respectively, (ii) corresponding to $\G^{[]}$ or $\G^{\{\}}$.
\end{itemize}

\medskip\noindent
\textbf{(i) Functions on $\G^{\langle\rangle}$.}
Consider the functional $\nu$ of the connection-matrix distribution in an $\N$-sample, i.e., look at the space $\G^{\langle\rangle}$ and use test functions on $\G^{\langle\rangle}$ that can be lifted to functions on $\G^{[]},\G^{\{\}}$. Recall that an element of $\G^{\langle\rangle}$ is of the form $\langle(\CI,\tau),h,\mu\rangle$ for \emph{some} Polish space $\langle\CI,\tau\rangle$. Our first observation is that we can study the marginal distributions $\nu^{(m)}$, $m\in\N$, of $\nu$ via the \emph{subgraph count functions}, which are the \emph{monomials $\Phi$} of degree $m$, given as follows.

Fix $C = [m]$ and $A,B \subseteq C^2 \setminus \mathrm{Diag}(C^2)$ (denoted by $C^{2 \setminus}$) such that $A \cup B=C^{2 \setminus}, A \cap B = \emptyset$ and $|A | + |B| = {m \choose 2 }$ for some $m= \mid C \mid \in\N$ (this induces an $m$-subgraph with specified vertices from $C$ and its edges). Introduce functions $\wt \varphi_{A,B} = 1_A(1-1_B)$ that are indicators of the set of $m$-graphs with the pairs in $A$ connected and the pairs in $B$ not connected. In this way we get functions on the pair $(\CI^m,h)$:
\begin{align}
\label{e796}
\CF_m= \left\lbrace\varphi_{A,B} \colon\, A,B \subseteq C^{2 \setminus},
A \cap B = \emptyset,\, |A| + |B| = {m \choose 2} \right\rbrace,
\end{align}
where $\varphi_{A,B}$ on $(\CI^m,h)$ is defined via $\wt \varphi_{A,B}$ as
\begin{equation}
\label{e976}
\varphi_{A,B} (\uu) = \wt \varphi_{A,B} \left(\uuh(\uu)\right)
= \Prodl_{(i,j) \in A} h(u_i,u_j) \Prodl_{(i,j) \in B} \left(1-h(u_i,u_j)\right).
\end{equation}
Define
\begin{equation}
\label{e802}
\CF= \bigcup\limits_{m \in \N} \CF_m,
\end{equation}
which distinguishes equivalence classes of countable graphs, i.e., graphemes. Set
\begin{equation}
\label{e1575}
\Phi^{m,\varphi_{A,B}} (\CI,h,\mu) = \int_{\CI^m} \varphi_{A,B}\; (\uuh(\uu)) \mu^{\otimes m} (d\uu),
\end{equation}
where for $\mu$ with atoms we use $\mu^{\otimes m, \downarrow}$. Then the test functions on $\G^{\langle\rangle}$ are given by
\begin{equation}
\label{e806}
\Pi^{\ast,\CF} = \text{ the algebra generated by (=\,connection polynomials).}
\end{equation}
These polynomials determine asymptotically the behaviour of the \emph{normalised subgraphs counts} in sequences sampled according to $\mu$.

For $m \in \N$, consider more generally functions in $\Pi^{\ast,\CF}$ that are specified for some symmetric $\varphi$ vanishing on the diagonal,
\begin{equation}
\label{e787}
\varphi\colon\, \{0,1\}^{[m]\times[m]} \to \R.
\end{equation}
Call this set $\wt \CF_m$, with $\uuh$ denoting the connection matrix of an $m$-sample $\uu$ from $\CI$. Then the test functions are given by
\begin{equation}
\label{e791}
\Phi^{m,\varphi} \left([\CI,h,\mu]\right) = \int_{\CI^m} \varphi(\uuh)(\underline{u})\,\mu^{\otimes m}(\d\underline{u})
= \int_{\{0,1\}^m} \varphi(\uuh)\, \nu^{(m)}(\d \uuh), \qquad \varphi \in \wt \CF_m,\, m \in \N.
\end{equation}
 Note that~\eqref{e791} is only defined when $h$ is a \emph{measurable function}. The right-hand side shows that our quantities so far do \emph{not} depend on the choice of representatives of the equivalence class. We define the algebra generated by the monomials in~\eqref{e791} on $\G^{\langle\rangle}$ as
\begin{equation}
\label{e374}
\Pi^\ast = \text{ algebra of \emph{general connection polynomials}}.
\end{equation}

On $\G^{\langle\,\rangle}$ we want a topology that is based on these functions, requiring for \emph{convergence of a sequence $(\CG_n)_{n \in \N}$ to $\CG$} the convergence of all evaluations of $\CG_n$ by $\Phi \in \Pi^\ast$ to the one of $\CG$. This works because of the following properties.

\begin{proposition}[$\Pi^\ast$ law convergence determining for $\nu$]
\label{prop.928}
$\mbox{}$\\
(a) The algebra generated by $\Pi^\ast$ uniquely determines the $m \times m$ connection-matrix distributions $\nu^{(m)}$, $m \in \N_{\geq 2}$, and is by definition convergence determining on $\G^{\langle\rangle}$.\\
(b) The moments of the random grapheme $\CG$ with state space $\G^{\langle\,\rangle}$,
\begin{equation}
\label{14}
\left\lbrace \E\left[\Phi(\CG)\right],\Phi \in \Pi^\ast \right\rbrace,
\end{equation}
determine $\CL [\nu]$ uniquely and are convergence determining for laws on $\G^{\langle\rangle}$. \qed
\end{proposition}

We can extend these functions $\G^{\langle\rangle}$ from $\Pi^\ast$ to $\G^{[]}$ and $\G^{\{\}}$, since they do not depend on anything other than the functional $\nu$. We misuse notation by still talking about the algebra $\Pi^\ast$, even though strictly speaking we have $\Pi^{\ast,\sim}$ with $ \sim= \langle\rangle,[]$ or $\{\}$. Proposition~\ref{prop.928} holds for the functional $\nu$ on the spaces $\G^{[]},\G^{\{\}}$ as well.

\medskip\noindent
\textbf{(ii) Functions on $\G^{[]},\G^{\{\}}$.}
In these space we have to look also at \emph{convergence of the underlying spaces}, i.e., we consider functions for $(\CI^\ast,\mu)$, respectively, $(\CI,r,\mu)$ in order to be able to treat $\G^{[]},\G^{\{\}}$. Note that the equivalence relation leading to $\G^{[]}$ requires isometric measure-preserving bijections between the spaces $(\CI',\CB_{\CI'}),(\CI'',\CB_{\CI''})$, under the assumption that $\CI^\prime=\supp(\mu^\prime), \CI^{\prime\prime}=\supp(\mu^{\prime\prime})$. This \emph{equivalence} is denoted by $[]$. The topology on $\G^{[]}$ is therefore obtained by defining \emph{convergence of sequences} $\CG_n \to \CG$, with $\CG_n,\CG \in \G^{[]}$ as below in the paragraph on topology. We proceed similarly in $\G^{\{\}}$, replacing isometric by homeomorphic. This means that we need a set of test functions of $[\CI,r,\mu]$, respectively, $[\CI^\ast,\mu]$ that are \emph{separating}.

Take $\G^{[]}$ and consider test functions that are averages with respect to the sampling measure $\mu$ of functions of matrices of $m$-samples, namely, the \emph{monomials} $\Phi^{m,\varphi}$ on $\U_1$, i.e., functions of $[\CI^*,h,\mu]$. More precisely, on $\G^{[]}$ we choose a bounded continuous test function $\varphi(\underline{\underline{r}})$, with $\underline{\underline{r}}$ a continuous $\R$-valued function on $\CI \times \CI$, given by \emph{distances $r(\cdot,\cdot)$} between points in $\CI$, which determines the equivalence class of $(U,r,\mu)$ in $\U_1$ and is of the form (see~\cite{GrevenPfaffelhuberWinter2009}) generating the algebra that we call $\wh \Pi$:
\begin{equation}
\label{e341}
\Phi^{m,\varphi}([(\CI,r),h,\mu]) = \int_{\CI^m} \varphi (\uur)\, \d\mu^{\otimes m},
\qquad m \in \N,\, \varphi \in C_b(\R^2).
\end{equation}
In order to obtain the test functions for $\G^{[]}$, we take the function from $\Pi^\ast$ that determines the distribution of $h$ under $\mu$ sampling, consider $\varphi(\uuh)$ in~\eqref{e341} instead of $\varphi(\uur)$, take the integral of the product $\varphi_r(\uur) \varphi_h(\uuh)$ to get a monomial on $\G^{[]}$, and consider the generated algebra $\wh \Pi^\ast$.

For $\CI^*$ and $\G^{\{\}}$,
\begin{equation}
\text{in~\eqref{e341} use as integral a continuous function $\varphi(\underline{u})$ on $\CI^m$ (instead of $\varphi(\underline{\underline{r}})$)},
\label{e1730}
\end{equation}
induced by products of functions, say $(\varphi_i)_{i \in [m]}$, that are bounded and continuous on $\CI$. Take the product over $i \in [m]$. This identifies $((\CI,\tau),\mu)$ and generates an algebra $\wh \Pi$ on the equivalence classes of Polish measure spaces. To identify the grapheme including $h$, this has to be \emph{complemented by a further factor of the form in~\eqref{e791}}. We note that, in order to evaluate the function in another representative, the formula in the analogue of~\eqref{e341} involves the homeomorphism.

The above construction on $\G^{[]},\G^{\{\}}$ (we use the same notation for the algebras whenever no confusion can occur) leads to
\begin{equation}
\label{e403}
\wh \Pi^\ast = \text{ algebra of polynomials of $[(\CI,r),h,\mu]$, respectively, $\{\CI^\ast,h,\mu\}$}.
\end{equation}

\begin{proposition}[Law determining]
\label{prop.1743}
The algebra $\wh \Pi^\ast$ is separating for $\G^{[]}$ and $\G^{\{\}}$, and, when taking random elements as input expectations, determine the laws on $\G^{[]}$, respectively, $\G^{\{\}}$. \qed
\end{proposition}

%%%

\paragraph*{Topology.}

We next define \emph{convergence} of sequences to a limit point, which has two ingredients.

\medskip\noindent
\textbf{(I)} The \emph{first} requirement for a sequence of graphemes from $\wt \G$ with $\sim\,=\langle \rangle,[],\{\}$ to converge to a limit grapheme is that the corresponding \emph{connection-matrix distributions} \emph{converge}. Formally, the sequence $([\CI_n^{\ast},h_n,\mu_n])_{n \in \N}$ from $\G^{\langle\rangle}$ converges to an element $[\CI^*,h,\mu] \in \G^{\langle\rangle}$ if and only if
\begin{equation}
\label{e810}
\Phi \left([\CI_n^{\ast},h_n,\mu_n]\right) \ntoo \Phi \left([\CI^\ast,h,\mu]\right) \qquad  \forall\,\Phi \in \Pi^{\ast,\CF},
\end{equation}
and hence for all $\Phi \in \Pi^\ast$.

\medskip\noindent
\textbf{(II)} For $\G^{\langle\rangle}$, the \emph{second} requirement is strengthening the graphon topology as follows. The set of $\{0,1\}^\N$-valued connection-matrix distributions is compact in the weak topology. But what we need in $\G^{\langle\rangle}$ are connection-matrix distributions of an infinite sample sequence guaranteeing that $(\nu_n)_{n \in \N}$ has limit points $\nu$ that arise from a sampling sequence evaluating a $\{0,1\}$-valued function on $\CI^\ast \times \CI^\ast$. For $\G^{\{\}}$ and $\G^{[]}$, the \emph{second} requirement is that the limit points $\nu$ from (I) have a suitable $(\CI^*,\mu)$ in which they can be embedded, and that these triples arise as the \emph{limit} along the sequence as well. In other words, we have to specify a topology for the space of equivalence classes of $(\CI^*,h,\mu)$ from either $\G^{\{\}}$ or $\G^{[]}$. For that purpose we set
\begin{equation}
\label{e964}
\begin{aligned}
&\text{$\T$ is the space of \emph{equivalence classes $[\CI^*,\mu]$} of Polish measure spaces},
\end{aligned}
\end{equation}
arising as limit of elements with finitely many points arising from elements of $\G_\infty$, and
\begin{equation}
\label{e1371}
\text{$\M$ is the space of \emph{equivalence classes} $[\CI,r,\mu]$,}
\end{equation}
arising of limits of \emph{finite} metric measure spaces. For convergence of a sequence $[\CI^\ast, h,\mu]$, respectively, $\{\CI^\ast,h,\mu\}$ we also have to require, beyond~\eqref{e810}, that
\begin{equation}
\label{e957}
\text{$((\CI_n^{\ast},\mu_n))_{n \in \N}$ converges to $(\CI^*,\mu)$ in the topology of $\T$,
respectively, $\M$ with $(\CI,r,\mu)$}.
\end{equation}
The topologies for $\M$, respectively, $\T$ are defined by specifying \emph{convergence of sequences}, which is obtained by requiring convergence of evaluations of the chosen class of test functions $\wh \Pi$ given in~\eqref{e341}, respectively,~\eqref{e1730} for $\G^{[]}$, respectively, $\G^{\{\}}$.

To get convergence of the joint structure of $h$ and $(\CI^\ast,\mu)$, we have to \emph{enrich} the algebra of polynomials to $\wh \Pi^\ast$, so that they separate equivalence classes of Polish measure spaces, and subsequently define the topology based on the evaluations from $\wh \Pi^\ast$. Hence we define \emph{convergence in $\G^{[]},\G^{\{\}}$} by requiring
\begin{equation}
\text{\emph{convergence of the evaluations of all functions from~\eqref{e403}}}.
\label{e1775}
\end{equation}

Thus, in (I) and (II) we have turned $\G^{[]},\G^{\{\}},\G^{\langle\rangle}$ into \emph{topological} spaces and we have Borel-$\sigma$-algebras to obtain measure spaces. The important point is to get, whenever possible, \emph{Polish} spaces, as state spaces for random variables and stochastic processes in their standard form.

\begin{proposition}[State spaces]
\label{prop.813}
$\G^{[]}_{\rm ultra}$ is a Polish space, and so is the closure of $\G^{\langle\rangle}$, called $\wh \CW$, and of $\G^{[]},\G^{\{\}}$, called $\wh \CW^{[]},\wh \CW^{\{\}}$. \qed
\end{proposition}

\begin{remark}\label{r.1738}
{\rm We leave open the issue that if we introduce on $\G^\ast, \ast=\langle\rangle,\{\},[]$ a stronger topology for convergence of the $h$-component in the grapheme topology, guaranteeing that the functional $\nu$ satisfies an entropy condition on its range that is equivalent to having $\{0,1\}$-valued graphons, then the closure in this stronger topology gives a Polish space $\G^\ast$.} \hfill$\spadesuit$
\end{remark}

Since the closure of $\G^{\langle\rangle}$ corresponds to the space of graphons, the set $\G^{\langle\rangle}$ is strictly smaller than the closure and we have to therefore check compact containment conditions for our processes in order to obtain processes with values only in $\G^{\langle\rangle}$. Since in our theorems we obtained processes with values in $\G^{[]}$, this guarantees the required property after we pass to $\G^{\langle\rangle}$ and we know that we truly obtain a process with values in $\G^{\langle\rangle}$.

Does the same hold for $\G^{\{\}}$, or must we restrict to smaller spaces, like spaces of algebraic trees (see~\cite{LW21})? What remains is to settle metrisability (separability and completeness follow from the definition), i.e., to \emph{exhibit a metric} generating the topology. There is no obvious reference for this. However, we observe that if we restrict $\T$ to spaces that have a representative that is an \emph{algebraic tree}, then this closed subspace can be equipped with a concrete metric via an embedding, and for our processes we would get a dynamics in this Polish space. We will not pursue this task further in the present paper.

We had been able to do the construction for $\G^{\langle\rangle}, \G^{[]}$, since the topology can be based on a countable set of test functions that are \emph{universal} for all representations of an equivalence class, which is not available for $\G^{\{\}}$. But this comes at the price that we have to check on $\G^{\langle\rangle}$ whether our stochastic processes have paths with values in $\G^{\langle\rangle}$, which is not complete in the metric based on the subgraph counts.

\begin{remark}[$\G^{[]}$ versus $\G^{\{\}}$ or $\G^{\langle\rangle}$]
\label{r.969}
$\mbox{}$\\
{\rm  (1) These spaces are problematic for the $h$-compo\-nent, namely, the lack of completeness of $\G^{\langle\rangle}$, respectively, the fact that its completion does not represent countable graphs (they may be randomised), so that we cannot formulate standard martingale problems on the full space, except on the $\wt \G^\ast_{\rm ultra}$-versions, as we have demonstrated. \\
(2) There is a clear advantage to work with $(\CI,r)$ and on $\G^{[]}$. On $\G^{\{\}}$ we have the following further drawback. On $\G^{[]}$ the test functions are natural and are comparable between different representatives, since we deal with $\varphi$ that are \emph{universal} in the representatives. This is not the case on $\G^{\{\}}$ with $\varphi(\uu)$, where we have to use homeomorphisms between different representatives that are used to define $\varphi$ on a different representative than a tagged representative from the equivalence class. This problem can only be circumvented on $\G^{\{\}}$, when we embed the space of a representative of the equivalence class homomorphically into $([0,1], \CB_{[0,1]},\dd u)$ and work with \emph{uniformly} continuous functions on $[0,1]$. However, this yields bounded continuous functions on $\CI^*$, but as a function of $\CI$ we loose uniformity, because the embedding in $[0,1]$ does \emph{not} necessarily have this uniformity. This in turn means that we can define continuous bounded polynomials as test functions to define our topology, but we are lacking concrete \emph{uniformly continuous subclasses} of such functions that are also measure and convergence determining. This in turn means that we cannot exhibit on $\G^{\{\}}$ a countable measure-determining set of uniformly continuous test functions in $C_b(\G^{\{\}},\R)$, induced by a function on $\CI^*$ that is invariant under equivalence, as we can on $\G^{[]}$. This fact is an obstacle, for example, when we want to prove the strong \emph{Markov property} or a (generalised) \emph{Feller property} for stochastic processes with values in these spaces.}\hfill$\spadesuit$
\end{remark}

The polynomials on $\G^{\{\}}$ are a \emph{subclass} of the polynomials on $\G^{[]}$, since fewer Polish spaces are equivalent. In particular, the martingale problems on $\G^{\{\}}$ are more complicated than those on $\G^{[]}$, since we deal with invariants for larger sets.

%%%

\subsubsection{Tightness in grapheme spaces}

We often have to decide whether sequences of graphemes converge to a \emph{limit element} in the respective spaces $\G^\sim$ with $\sim\, = [],\{\},\langle\rangle$ when the evaluations of polynomials in $\wh \Pi^\ast$ converge to a limit number. In order to get laws and convergence of graphemes with different non-trivial limits of $h$ (i.e., not equal to $\underline{\underline{0}}$ or $\underline{\underline{1}}$), we need \emph{extra conditions}.

For convergence we must have that $((\CI_n^{\ast},h_n,\mu_n))_{n \in \N}$ is determined by the polynomials as test function and converges to a limit in $\G^\sim$ with $\sim\, = [],\{\},\langle\rangle$. Since we have defined the convergence of elements in $\G^{[]}$ as convergence of evaluations of polynomials, i.e., elements of the algebra $\wh\Pi^\ast$ giving us the natural \emph{topology} on $\G^{[]}$, it remains to prove that the \emph{limit of the evaluations} is an \emph{evaluation of a limit point} in the state space.

Recall that we require $\CI^*$ to be Polish only, so $\CI^*$ need not be compact or locally compact, nor does $\CM_1(\CI^*)$ need to be compact. Therefore we possibly \emph{cannot} sample from a limit object. This poses conditions on $([\CI_n^{\ast},h_n,\mu_n])_{n \in \N}$, namely, \emph{tightness properties} of the spaces in which we embed the graphemes of index $n$. It is here that the \emph{embedding} enters into the game. Furthermore, we have to guarantee that limiting elements equal a $\{0,1\}$-valued $h$, which requires that the connection-matrix distributions viewed as measures on $[0,1]^\N$ are tight in the stronger topology, strengthened by an \emph{additional requirement} guaranteeing $\{0,1\}$-valued rather than $[0,1]$-valued limits. Such a tightness criterion is the entropy criterion for sequences of finite samples, requiring the entropy of these samples to be $o(n^2)$ (see~\cite{GdHKW23} and~\cite{hatami2012entropy}), which restricts the closure of the finite graphemes to a subspace of $\wh \CW$. We will see in Section~\ref{ss.mainideas} that on $\G^{[]}_{\rm ultra}$ this is immediate, since in that case the $h_n$ are continuous and the metric is an ultrametric, so we need not involve the entropy criterion.

If we consider $\G^{[]}$, then we have metric measure spaces $(\CI,r)$. We also know that here we deal with the closed subspaces $\U_1,\U$, which evolve in $\G^{[]}_{\rm ultra}$. We can use the theory in~\cite{GrevenPfaffelhuberWinter2009,GrevenPfaffelhuberWinter2013,DepperschmidtGrevenPfaffelhuber2011} to get Polish spaces, convergence criteria and tightness criteria, the latter amounting to the so-called
\begin{equation}
\label{e1284}
\text{\emph{no dust} condition}.
\end{equation}
On $\U_1,\U$ this amounts to requiring that, for every $\varepsilon>0$, the sampling measure puts weight $\geq 1- \varepsilon$ on \emph{finitely many ancestors} a time $\varepsilon$ back. Indeed, this gives us a condition for the tightness of the spaces in which we embed the elements of our sequence. Furthermore, we need convergence of the connection-matrix distributions. However, if the evaluations using $\Pi^\ast \subseteq \wh \Pi^\ast$ converge and the spaces of embedding have a limit in which we can embed, then we have found the limit in $\G^{[]}$. Therefore with~\eqref{e1284} we get a tightness condition in $\G^{[]}$.

Next, we consider random graphemes and their laws. We can derive a tightness criterion because these live in $\CM_1(\G^{[]})$, which takes on the form $\U_1,\U$, namely,
\begin{equation}
\label{e.1784}
\begin{aligned}
&\text{the number of ancestors time $\varepsilon$ back that have descendants}\\
&\text{of total mass $\geq 1-\varepsilon$ is stochastically bounded.}
\end{aligned}
\end{equation}
On $\U$ we need, in addition, that the total masses are stochastically bounded. Finally, we need that the laws of the connection-matrix distributions are concentrated on the set of functions with values in $\{0,1\}$, for example, when we form local averages over $h$.

Replacing $\G^{[]}$ by $\G^{\{\}}$, we get into different territory and there is not much available in the literature. By the embedding in $[0,1]$ mentioned above, we can introduce the topology in the same way and obtain a topological space (in fact, a Polish space), but the algebra of polynomials is lacking nice properties. For $\G^{[]}$ we can work with \emph{distance-matrix distributions} to represent the state (analogously to the connection-matrix distributions $\nu^{(m)}$, $m \in \N$, used on $\G^{\langle\rangle}$). On $\G^{\{\}}$ we have \emph{sampled tuple distributions}, which via the embedding become sampled tuple distributions in $([0,1],\CB_{[0,1]})$. A formulation independent of the equivalence class is subtle, because we can only say that the polynomial transforms via the integral transformation formula on the equivalence class, hence is a function of that class, and so we can use the representation on a subset of $([0,1],\CB_{[0,1]})$. This means that we have test functions involving the space and the homeomorphisms, a much less canonical setting than functions of distances $\varphi$ that are the same on every representative, as for $\G^{[]}$. As pointed out before, we cannot gain much by pursuing this line further.

Turning to $\G^{\langle\rangle}$, we note that the elements of $\G^{\langle\rangle}$ correspond to a subset of probability measures on $\{0,1\}^{\N \choose 2}$ that is not closed. Hence, even though $\CM(\{0,1\}^{\N \choose 2})$ is a Polish space, and as a measure space on a compact set is classical, we require restrictions that lead to a \emph{subset}. Namely, the elements of this subset arise from i.i.d.\ sampling on \emph{some Polish space}, and the \emph{empirical function} $h$ is \emph{$\{0,1\}$-valued}. In other words, we need to check that limit points of sequences obtained from the connection-matrix distributions can indeed be realised, i.e., arise from an embedded sequence, and lead to a \emph{$\{0,1\}$-valued} function $h$. This need not be easy, as is exemplified by graphons with grey-shades, corresponding to sequences $(h_n)_{n\in\N}$ that do not have a limit $h$ with values in $\{0,1\}$, and hence do not define a grapheme based on the space $([0,1]$,Borel, $\mu$), because the limiting subgraph counts determine a unique $h$ with values in $[0,1]$. Therefore $\G^{\langle\rangle}$ is a subtle subset that is homeomorphic to a subset of the Polish space $\CM_1(\{0,1\}^\N)$. However, we can use the well-known fact that the space of graphons with the topology induced by the subgraph counts can be generated by the cut metric. Therefore we get with $\G^{\langle\rangle}$ a separable metrisable subset of a Polish space, which, however, is \emph{not} closed. For our dynamics there are topologically closed subspaces $\G^{[]}_{\rm ultr}$ that are closed under the dynamics and induce subspaces in $\G^{\langle\rangle}$.

For a sequence of graphemes in $\G^{\langle\rangle}$ to converge, it \emph{suffices} to find a sequence of elements from $\G^{\{\}}$, a limit candidate from that space, such that the corresponding \emph{subgraph counts converge}. The limit candidate for the embedding we find by showing that the sequence has a limit in $\G^{\{\}}$, including a limiting function $h$. This means that from tightness criteria in $\T$ we get \emph{sufficient} criteria for tightness in $\G^{\langle\rangle}$.

We can deal with $\G^{[]}$ in the same way, which gives us accessible sufficient tightness criteria. With~\eqref{e.1784} we get a criterion guaranteeing that we have limit points in $\G^{[]}$, and convergence when the subgraph counts converge as well, since the sequence of equivalence classes of metric measure spaces has a limit grapheme in $\G^{[]}$ and a limiting $\{0,1\}$-valued $h$. The latter is often immediate on $\G^{[]}_{\rm ultra}$. However, we know from the general theory of graphons that the topology can be generated by a metric (the \emph{cut distance} from~\eqref{e385}), so we get values in a subset of a Polish space.

%%%

\subsection{Extension to varying grapheme sizes}
\label{sss.extens}

We have to argue that the previous observations extend to varying grapheme sizes. To do so, we have to introduce polynomials $\Phi$ on $[U,r,\mu]$ of the form $\Phi = \bar \Phi \, \wh \Phi$ with $\bar \Phi \in C_b(\R^+,\R)$ and $\wh \Phi$ a polynomial of $[U,r,\wh \mu]$, in order to be able to identify the pair $\bar \mu$ and $[U,r,\wh \mu]$. Hence for $[U,r,\mu]$ we work with functions $\bar \Phi$ of $\bar \mu$, that are \emph{bounded} and continuous. In our models we will be able to work with polynomials based on $\mu$ rather than on $\wh \mu$, but we have to argue that we have the necessary \emph{integrability properties} for $\bar \mu$, and must include growth properties in the degree in order to ensure that they are separating for a subset of $\U$ that occurs as possible values of our random variables. In our examples based on the Galton-Watson process, we can work with $\bar \Phi(\bar \mu)=\bar \mu^n$, $n \in \N$, and consider the subset of $\mu$ for which $\bar \mu^n$ is integrable and the Laplace transform of $\bar \mu$ exists in some neighbourhood of $0$. This is described in detail in~\cite{DG19} and~\cite{ggr_GeneralBranching}.

%%%

\subsection{Extension to marked graphemes}
\label{ss.marked}

For the case of $V$-marked graphemes, take a Polish space $(V,r_V)$ and consider the \emph{$V$-marked grapheme}
\begin{equation}
\label{e568}
\left[(\CI \times V)^*,h,\mu_{\CI \times V} \right]
\end{equation}
with $\wt \kappa,\kappa$ denoting mark kernels, respectively, mark functions from $\CI$ to $\CB_V$, respectively, $V$ and
\begin{equation}
\label{e572}
\mu_{\CI \times V} = \mu \otimes \wt \kappa, \qquad \mu \in \CM_1 (\CI^\ast),
\end{equation}
where $\wt \kappa$ is a kernel assigning marks to vertices (typically, $\wt \kappa(\cdot,0)=\delta_{\kappa(\cdot)}(0)$ for a \emph{mark function} $\kappa$ that measurably maps $\CI$ into $V$), and consider the spaces
\begin{equation}
\label{e.578}
\begin{aligned}
&\wt \G^V = \text{ space of $V$-marked graphemes},\\
&\G^V \text{ with $h(i,i^\prime) =1 \Longrightarrow \kappa(i)=\kappa(i^\prime)$}.
\end{aligned}
\end{equation}
Recall that the corresponding state space
\begin{equation}
\label{e890}
\G^{\langle\cdot\rangle}, \G^{V,\langle\cdot\rangle}
\end{equation}
is a coarse-graining of $\G^{[]},\G^{[],V}$. In the marked case $\wt \G^V$, $\nu$ must be replaced by the
\begin{equation}
\text{\emph{marked connection-matrix distribution} $\nu_V$,}
\label{e1920}
\end{equation}
where we keep track of the edges and of the type of the vertices. Here,
\begin{equation}
\text{the entry $1$ is replaced by $(1,(v_1,v_2))$ \emph{and} the entry $0$ by $(0,(v_1,v_2))$, with $v_1,v_2 \in V$. }
\label{e.1925}
\end{equation}

\begin{proposition}[Determining properties]
\label{prop.1}
$\mbox{}$\\
(a) Any element in $\G^{[],V}$ or $\G^{\{\},V}$ uniquely determines $\nu_V$, and similarly after $\G$ is replaced by $\bar \G$.\\
(b) Consider the case of a non-atomic sampling measure $\mu$. Then every $\nu_V$, with the sub-matrix independence property for any two sub-matrices connected with two disjoint index sets, determines an element in $\G^{\langle\,\rangle,V}$. This element may, however, be trivial, i.e., equal the complete graph or the totally disconnected graph. \qed
\end{proposition}

We continue with the state space
\begin{equation}
\label{e830}
\begin{aligned}
&\G^{[],V} =  [(\CI \times V)^*,(h,\kappa),\mu_{\CI \times V}],\\
&\kappa\colon\, \CI \to V, \quad \mu_{\CI \times V} \in \CM_1((\CI \times V)^*), \quad h\colon\, \CI^2 \to \{0,1\},\\
&\mu_{\CI \times V} (\cdot,\cdot)\colon\,\mu (\cdot) \otimes \delta_{\kappa (\cdot)}^{(0)},\quad \mu \in \CM_1 (\CI^*).
\end{aligned}
\end{equation}
The equivalence class is formed with respect to $h$-preserving, $\kappa$-preserving and measure-preserving bijections between the supports of the measures. The objects that characterise the marked connection-type structure for finite $n$ are
\begin{equation}
\label{e839}
k_n\colon\,\CI^n \to \{0,1\}^{n(n-1)} \times V^n, \quad k_n (\CG) = \big((h(\iota_i,\iota_j))_{i,j \in [n]},
(\kappa(\iota_i))_{i \in [n]}\big),
\end{equation}
which gives us $(\nu^{\otimes n}_{\ast \kappa_N})_{n \in \N}$, the sequence of connection-type matrix distributions of order $n$. We proceed analogously when we work with $(\CI,r)$ and $\G^{[]}$.

We can generalise further by allowing, instead of mark functions, \emph{mark kernels}
\begin{equation}
\label{e1108}
\kappa\colon\, (\CI,\CB_\CI) \to (V,\CB_V).
\end{equation}
Then the space $\G^{\langle\rangle,V}$ is based on the \emph{connection-type matrix distribution $\nu_V$} as invariant. The state spaces depend on whether we choose $\CI^*$ or $(\CI,r)$:
\begin{equation}
\label{e847}
\G^{[],V}, \quad \G^{\{\},V}.
\end{equation}
We introduce a topology induced by defining convergence of evaluations of functions on $\G^{[],V}$ taken from an \emph{algebra of polynomials}, denoted by $\wh \Pi^{\ast,V}$.

Let us first look at the simpler case $\G^{\langle\,\rangle,V}$ again. We define $\Pi^{\ast,V}$ acting on $\G^{\langle\cdot\rangle,V}$. A \emph{monomial} $\Phi \in \Pi^{\ast,V}$ is given by a function $\varphi \in C_b(\{0,1\}^{n(n-1)},\R)$, $\Psi \in C_b(V^n,\R)$, $n \in \N$, and
\begin{equation}
\label{e.855}
\begin{aligned}
\Phi(\CG) &= \Phi \left([(\CI \times V)^*,h,\kappa,\mu_{\CI \times V}]\right)\\
&= \int_{(\CI \times V)^n} \varphi\big((h(\iota_k,v_k)(\iota_\ell,v_\ell))_{k,\ell \in [n]}\big)\,
\psi\big((\kappa(\iota_k))_{k \in [n]}\big)\, \mu^{\otimes n}_{\CI \times V} \underline{(\iota,v)}\\
&= \int_{\left(\{0,1\} \times V^2\right)^{n(n-1)}} (\varphi \otimes \psi)\,\d \nu^{(n)}_V
\end{aligned}
\end{equation}
with $\underline{\iota} =(\iota_1,\ldots, \iota_n),\underline{v}=(v_1,\ldots,v_n)$, and $\nu_V$ the connection-type matrix distribution on the edges and their end points.

In order to get to the topology on $\G^{[],V},\G^{\{\},V}$, we need to control the embedding in the space $\CI^*$, and also need finer test functions
\begin{equation}
\label{e985}
\wh \Pi^{\ast,V}
\end{equation}
to fully characterise the elements of $\G^{[],V}$. We consider~\eqref{e.855} and replace the test function $\varphi$ by \emph{products} of the two test functions
\begin{equation}
\label{e1849}
\varphi_h\big((h((\iota_k,v_k),(i_\ell,v_\ell))_{k,\ell \in [n]}\big) \varphi_r\big((r(\iota_k,i_\ell))_{k,\ell \in [n]}\big),
\end{equation}
where $\varphi_h$ denotes the $\varphi$ in~\eqref{e.855}, and $\varphi_r\colon\,[0,\infty)^{n^2} \to [0,\infty)$ is continuous.

If we turn to $\wt \G^{\{\},V}$, then we replace $\varphi_r$ by a continuous function on $\CI^n$. We cannot rely on the results in~\cite{GrevenPfaffelhuberWinter2009,GrevenPfaffelhuberWinter2013,DepperschmidtGrevenPfaffelhuber2011}, because for $\G^{\{\},V}$ we have no nice subset that is measure-determining and convergence-determining and uniformly continuous. For the dynamics in the present paper this is not really crucial, since we work on $\G^{[],V}$ anyway.

\begin{proposition}[The state space: $V$-marked]
\label{prop.1202}
$\G^{[],V}_{\rm ultra}$ and the closure of $\G^{\langle\rangle,V},\G^{[],V}$ are Polish.
\qed
\end{proposition}

\noindent
As before, the question is whether the same holds for $\G^{\{\},V}$ and what is the situation for $\G^{\langle\rangle,V}$. Here we refer to the discussion in the non-marked case.

%%%%%%%% SECTION 4 %%%%%%%%%%%%%%%%%%%%%%%%%%%%%%

\section{Grapheme evolution $2$: Examples and martingale problems}
\label{s.exmart}

In this section we give the details of the construction behind the Fleming-Viot and the Dawson-Watanabe grapheme diffusion on the state space $\G^{[],V}$ by using \emph{well-posed martingale problems}, focussing mainly on the specifications of the \emph{operators} acting on the domains of test functions $\wh \Pi^\ast$, respectively, $\wh \Pi^{\ast,V}$, which we  specified in Section~\ref{s.evol}.

Section~\ref{ss.fromto} specifies the \emph{finite-grapheme} Markov pure-jump processes, their rates and transitions and their deterministic part (the growth of distances), and gives their \emph{generators}. Section~\ref{ss.go} defines the \emph{operator} of the martingale problem for the grapheme diffusions in $\G^{[]},\G^{[],V}$ in the case of the dynamics with completely connected components.  Section~\ref{ss.ext} proceeds with the case of dynamics with \emph{non}-completely connected components, which arises after we allow independent insertion and deletion of single edges at rates $a^+,a^-$, or its \emph{generalised version} with $\CU$-dependent $a^+,a^-$.

%%%

\subsection{From finite-graph processes to grapheme processes: two classes of examples}
\label{ss.fromto}

%%%

\subsubsection{Fleming-Viot and Dawson-Watanabe: finite grapheme processes}
\label{sss.fincase}

We use the rules presented in Section~\ref{sss.compcase} for a jump process of finite graphs, specify all the parameters, extend the description systematically to the \emph{marked} version, and give the graph the structure of a grapheme by using a specific embedding of the graph in a measure space. We do this by constructing an equivalence class of finite ultrametric measure spaces from our evolving graphs. We define the Markov jump processes with deterministic component on $\G_\infty$ and introduce a notation for the corresponding grapheme processes. Recall that, due to the increase of distances in time, the evolving environment also has a deterministic continuous component and not only jumps. The formulation with types and marked graphemes is needed only when we extend to rules with mutation and selection. Otherwise we can think of the unmarked case. Nonetheless, it is useful to have the marks available to see some parts of the construction more clearly.

%%%

\paragraph*{Key objects.}

The finite \emph{$V$-type graph} of size $n$ is described as a grapheme, i.e., as the equivalence class
\begin{equation}
\label{e402}
\CG =[\CI^\ast \times V,h,\kappa,\mu_{\CI \times V}]
\end{equation}
of a quadruple $(\CI^\ast,h,\kappa,\mu)$, where $\CI$ is a set of vertices ($\CI^\ast = (\CI,\CB_\CI)$ is a measurable space) and $\mu$ is a probability measure on $\CB_\CI$. The mark function $\kappa\colon\,\CI \to V$ is the type-\emph{function} that associates with each vertex a type drawn from the set $V$, where $(V,\CB_V)$ is a measurable space based on a Polish space $V$. (Recall that for elements in $\U^V_1,\U^V$ this can in general be a transition kernel from $(\CI,\CB_\CI)$ to $(V,\CB_V)$.) The measure $\mu_{\CI \times V}$ on $\CI \times V$ is defined as $\mu \otimes \kappa$. In the case where we have a mark \emph{function}, the connection-type matrix
\begin{equation}
\label{e407}
h\colon\, \CI \times \CI  \to \{0,1 \},\, \text{  respectively, $(\CI \times V)^2 \to (\{0,1\} \times V)^2$}, h \text{ measurable}
\end{equation}
defines the edges $(a,a^\prime)$, with $a,a^\prime \in \CI$, and the types $v,v^\prime \in V$ with $v=\kappa (a),v^\prime= \kappa(a^\prime)$, by requiring that
\begin{equation}
\label{e1399}
\text{$(a,a^\prime)$ is an edge if and only if $h(a,a^\prime)=1$, respectively, $h((a,v),(a^\prime,v^\prime))=1$.}
\end{equation}
The measure $\mu$ is the \emph{sampling measure}, and allows us to draw $m$-samples from the graph by independently drawing $m$ vertices $(a_1,\ldots, a_m)$ according to $\mu$ without replacement. Two bijections between $\mathrm{supp}(\mu)$, $(\CI,r,h,\kappa,\mu)$ and $(\CI^\prime,h^\prime,\kappa^\prime,\mu^\prime)$, are equivalent if there exists a bijection $\varphi$ of $\mathrm{supp}(\mu)$ to $\mathrm{supp}(\varphi)$ such that $\mu$-a.s.\ $h^\prime=h \circ \varphi^{-1}, \kappa^\prime=\kappa \circ \varphi^{-1}$ and $\mu^\prime= \mu_{\ast \varphi}$. The set of all such objects is called
\begin{equation}
\label{e414}
\wt \G^V, \qquad  \G^V = \left\lbrace \wt \CG \in \G^V \colon\, h (a,a^\prime) = 1
\Longrightarrow \kappa (a) = \kappa (a^\prime) \right\rbrace.
\end{equation}

%%%

\paragraph*{The embedding in a discrete metric space.}

We have defined two classes of Markov pure-jump processes of finite graphs. From our evolving \emph{finite} graph at time $t$ we can build a \emph{finite ultrametric measure space} with an \emph{$h$-function}, namely, $(\CI_t,h_t,\mu_t)$. For example, interpret invasion as a birth event. Then we can define \emph{ancestors} and \emph{descendants} in the obvious way. For the set $\CI$ of current vertices we can define the \emph{distance of $(i,i^\prime)$} as twice the time back to the last common ancestor of $i$ and $i^\prime$. In this way we obtain a \emph{finite ultra-metric space}, giving us $(\CI_t,r_t)$. Hence our \emph{finite graph} gets embedded in a \emph{metric space} and we have
\begin{equation}
\label{e2005}
\text{constructed $((\CI_t,r_t),h_t,\mu_t)$ and defined a (finite) grapheme process.}
\end{equation}
If we add the types, then we similarly obtain the marked grapheme and its evolution.

%%%

\paragraph*{Finite-grapheme dynamics.}

We now turn the graph dynamics into a grapheme jump processes with deterministic part taking values in $\wt \G^{[],V}$: (i) the \emph{Fleming-Viot process} and the \emph{Dawson-Watanabe process} with \emph{immigration and emigration}; (ii) the same processes as in (i) but with \emph{insertion or deletion} of edges.

As possible initial states we consider the space $\G^{[],V}$, which requires that edges occur only between vertices marked with the same type. The processes are defined via the following three rules, each specifying at which \emph{rate} which \emph{jump} occurs.

%%%

\paragraph*{$\blacktriangleright$ Fleming-Viot evolution rule.}

\begin{itemize}
\item
At rate $d {n \choose 2}$ sample a pair $(a,a^\prime)$ from $(\CI \times \CI) \setminus D_{\CI\times\CI}$ based on $\mu$. If $\kappa(a)=\kappa(a^\prime)$, then nothing happens, while if $\kappa(a) \neq \kappa(a^\prime)$, then with probability $\tfrac12$
\begin{equation}
\begin{aligned}
&a^\prime \text{ takes over the type of $a$}, \\
&a \text{ takes over the type of $a^\prime$},
\end{aligned}
\end{equation}
the edges of the looser are removed, and the looser is connected to all the vertices the winner is connected to, $\mu$ is fixed and $\nu$ changes to $\nu-\delta_{(a,\kappa(a))} +\delta_{(a,\kappa(a^\prime))}$, respectively, $\nu-\delta_{(a^\prime,\kappa(a^\prime))}+\delta_{(a,\kappa(a))}$. Another way of saying this is that individuals $a^\prime,a \in \CI$ die and are replaced by two children, of individual either $a$ or $a^\prime$, inheriting the type (in population dynamics this resampling is known as the Moran model). This choice also defines the evolution of $(\CI_t,r_t)_{t \geq 0}$. Namely, distances of two vertices grow at rate $2$, and at a birth event the distance of the newborn to the ancestor is $0$ and is the same for all other vertices. We start from the space where all distances are $0$, and all vertices are identified with one point with full mass $\mu$. If we have types, then the $n$ points get different types drawn from the set $V$ according to some diffusive measure $\kappa_0$ on $V$.
\item
At rate $cn$ a vertex is a sampled and removed (together with all its connections), a new vertex $a^\prime$ enters with the type $v^\prime$ chosen according to the measure $\theta$ on $V$, $\mu$ is fixed and $\mu_{\CI \times V}$ changes to $\mu_{\CI \times V} - \delta_{(a,\kappa(a))} + \delta_{(a^\prime,v^\prime)}$ with probability $\theta (dv^\prime)$.
\end{itemize}

%%%

\paragraph*{Extensions of the Fleming-Viot rule with types and emigration/immigration.}

Consider marked graphemes $[\CI, h,\mu,\kappa]$, where $\kappa\colon\,\CI \to V$ is a mark function (i.e., a measurable function from $\CI$ to $V$), and consider additional rules for marked populations and their genealogical processes with values in $\U_1^V$. Of interest are cases where \emph{mutation} of types and/or \emph{selection} w.r.t.\ types is added, leading to new processes
\begin{equation}
\label{e857}
\left(\CG_t^{d,c,\theta,m,s}\right)_{t \geq 0}, \qquad \left(\CG_t^{b,c,\theta,m,s}\right)_{t \geq 0},
\end{equation}
that allow for statements like in Theorems~\ref{th.963-I}, \ref{th.963-II}, \ref{th.963-III} and~\ref{th.1838} for the Fleming-Viot, Dawson-Watanabe and McKean-Vlasov evolution rules, extended by additional \emph{mutation of types}.  The type-$b$ component mutates to type $b^\prime$
\begin{equation}
\label{e1687}
\text{at rate $m$ with transition probability $M(b,db^\prime)$, $b,b^\prime \in K$},
\end{equation}
and edges are connected and disconnected accordingly. To keep the notation simple, we only discuss the case when $M(b,\cdot)$ is \emph{diffusive}, so that always \emph{new} types appear. Furthermore, we have weak \emph{selection} of types with fitness function $\chi$, i.e., at rate $s/n$, $s \in (0,\infty)$, a vertex pair $x,y$ is picked, and with probability
 \begin{align}\label{e1925}
&\frac{\chi(x)}{\chi (x)+\chi(y)}\colon &x \to y \text{ replacement takes place,}\\
&\frac{\chi(y)}{\chi (y)+\chi(x)}\colon &y \to x \text{ replacement takes place.}
\nonumber
\end{align}

%%%

\paragraph*{$\blacktriangleright$ Dawson-Watanabe evolution rule.}

In the Dawson-Watanabe case, consider supercritical or subcritical branching. If $s>0$, then at rate $s$ a new vertex is born and is connected with every vertex. If $s<0$, then at rate $-s$ a vertex is removed (including all its edges). Furthermore, a vertex changes type at rate $m$ and with transition probability kernel $M(\cdot,\cdot)$ as above, and all old edges are removed and edges are added between all vertices carrying the new type.

%%%

\paragraph*{$\blacktriangleright$ Marked Fleming-Viot evolution rule with immigration.}

\begin{itemize}
\item
Every initial individual gets its own type, drawn from the set $[0,1]$. All individuals of the same type are connected.
\item
Pairs are sampled from the population at rate $d$. If a pair is sampled and the two individuals are of different type, then a coin is thrown to decide which of the two types is adopted by the pair, all connections are updated, one individual is \emph{disconnected} from all the individuals of its component and is \emph{connected} to all the individuals of the new component. Note that this is the same as saying that one individual duplicates the other and dies.
\item
At rate $c$ a vertex is \emph{removed} from its component and gets a \emph{new type} with law $\theta$ on $[0,1]$. In particular, the measure $\theta$ gives rise to different types of mechanisms:  if $\theta$ is non-atomic, then incoming types are \emph{always new} and create a new component, while if $\theta$ is atomic, then this is not so and $\theta$ is in fact a parameter.
\end{itemize}
The accompanying changes in the measure $\mu_{\CI \times V}$ are left to the reader. We obtain the same dynamics of connection matrices as before, but now the \emph{inheritable type of a vertex} allows us to impose new evolution rules.

%%%

\paragraph*{$\blacktriangleright$ Additional insertion and deletion of edges.}

We discuss two choices:
\begin{itemize}
\item
Simple model: At rate $a^+ {n \choose 2}$ a pair $(a,a^\prime)$ is chosen according to $\mu$ and the edge $(a,a^\prime)$ is connected, at rate $a^- {n \choose 2}$ a pair $(a,a^\prime)$ is chosen according to $\mu$ and the edge $(a,a^\prime)$ is disconnected, $\mu$ and $\mu_{\CI \times V}$ are fixed, and $h$ is jumping, as indicated.
\item
Modify the above by considering rates $a^+(\cdot),a^-(\cdot)$ that are functions of the population size of the type and/or of the age of the population, which is the maximal distance in the subpopulation. (For example, larger and older populations may be less pruned to loose a connection and more pruned to reinstall a lost connection). Here, $a^+(\cdot),a^-$ are bounded continuous functions.
\end{itemize}

%%%

\subsubsection{Finite grapheme processes and their generators}
\label{sss.graph}

The rules specified above via \emph{jump transitions} and their \emph{rates} uniquely define a \emph{Markov pure-jump process} on $\G_n$ or $\G_\infty$ with additional deterministic growth of distances.

To find the operator of the grapheme evolution, we use the observation that we can view the dynamics of the finite graphs as driven by a specific process, namely, the Moran process. A priori we have no guidance from the finite jump processes of abstract graphs on how to best embed it in a Polish space. Even for a metric structure on $\G_n^{\langle\,\rangle}$ we have to work with the finite embedded graph (whose embedding we have defined in Section~\ref{ss.fromto}), which is given via the \emph{genealogical Moran process}, respectively, the \emph{genealogical Galton-Watson process} for the case of Dawson-Watanabe. This gives us an \emph{enrichment} of the \emph{abstract finite-graph dynamics} via a topological structure on each element of the \emph{sequence of finite graphs} $(\CG^{(n)})_{n\in\N}$, which is moving in (a subset of) the space $\U_1$, respectively, $\U$ of equivalence classes of ultrametric measure spaces, and represents the \emph{genealogy} of the population of vertices evolving according to the Moran model. Namely, the distance of two individuals is twice the time to the most recent common ancestor. The sampling measure is the equidistribution, the weight is $2$ at birth and $0$ at death.

This gives us our $\G_\infty$-valued Markov processes, with $\sim$ either bar or no bar, and dropping the parameters $d,c,\theta$ we write
\begin{equation}
\label{e434}
(\CG^{{\rm FV},(n)}_t)_{t \geq 0} \text{, respectively, } \left(\CG_t^{\rm DW,(n)}\right)_{t \geq 0}.
\end{equation}
These simple processes can be defined as solution to a \emph{well-posed martingale problem}, with the test functions being the polynomials $\wh \Pi^\ast$ from Section~\ref{sss.testf}, as follows.

The \emph{operator} (here a classical generator) of the martingale problem is of the following form, where we split into resampling and emigration/immigration:
\begin{align}
\label{e1447a}
&\wh \CL^{\ast, {\rm FV},d,c,\theta,(n)}
= \wh \CL^{\ast,{\rm FV},d,(n)} + \wh \CL^{\ast,{\rm FV},c,\theta,(n),{\rm em-im}} \text{  on  } \G_n^{[]},\\
&\wh \CL^{\ast, {\rm DW},b,c,\theta,(n)}
=  \CL^{\ast,{\rm FV},b,(n)} +  \CL^{\ast,{\rm DW},c,\theta,(n),{\rm em-im}} \text{  on  } \bar \G^{[]}_\infty.
\label{e1447b}
\end{align}
Recall that we used graphemes embedded in an ultrametric space $\U_1$ or $\U$ that is a \emph{continuum}. To that end we had passed to $\G_n^{[]}$ and had given the finite system a structure as a \emph{finite ultrametric probability measure space}. Consequently, we have operators arising from the structure of the ultrametric measure space in which we \emph{embed our graph dynamics}, and obtain a grapheme for each $n$ that in the limit as $n \to \infty$ converges to a continuum grapheme.

Consider first the case without immigration or emigration. Then the operator (generator in this case) $\wh \CL^{\ast, \rm FV,(n)}$ is given on $\wh \Pi^\ast$ (recall~\eqref{e341}) by the formula
\begin{equation}
\label{e1458}
\wh \CL^{\ast, {\rm FV},(n)} = \wh \CL^{\ast, {\rm res},d} + \wh \CL^{\ast, {\rm grow}},
\text{ respectively, } \CL^{\ast, \rm DW, (n)}=\wh \CL^{\ast, \rm res,b} + \wh \CL^{\ast, \rm grow}.
\end{equation}
The first-order term, giving rise to the deterministic increase of distances with time, is given by
\begin{equation}
\label{e1462}
\wh \CL^{\ast, \rm grow} \Phi^{(n),\varphi} = 2 \langle \nu^{(n)}, \bigtriangledown (\varphi)\rangle.
\end{equation}
The second-order term is given via $\theta_{k,\ell}$, the operator replacing variable $\ell$ by variable $k$ (see~\eqref{e911}),
\begin{equation}
\label{e1465}
\wh \CL^{\ast, \rm res,d}\Phi^{(n),\varphi}
= \frac{d}{2} \sum_{k,\ell \in [n]} \big( \langle \nu^{(n)},
\varphi \circ \theta_{k,\ell} \rangle - \langle \nu^{(n)},\varphi\rangle \big).
\end{equation}
The generator $\CL^{\ast,{\rm DW},(n)}$ is similar, except that in~\eqref{e1465} we drop the term $\langle\nu^{(n)},\Phi^{(n),\varphi}\rangle$.

In both cases the polynomial for sampling \emph{without} replacement is given by
\begin{equation}
\label{e14699}
\Phi^{(n),\varphi} = \langle \nu^{(n),\varphi}, \varphi \rangle,
\end{equation}
and $\varphi$ is a function of $\uur=((r(u_i,u_j)))_{i,j \in [n]}$ with $(u_i)_{i \in [n]}$ the $n$-sample from $\CI$. Replacing $\varphi (\uur)$ by $\varphi(\uuh)$, with $\uuh$ the connection matrix, we have as operator for the functional $\nu$ on $\Pi^\ast$ also
\begin{equation}
\label{e1474}
\wh \CL^{\ast, {\rm FV},d,(n)} \Phi^{(n),\varphi} = \CL^{\ast, {\rm res},d} \Phi^{(n),\varphi}.
\end{equation}

We also need the operator for \emph{immigration} and \emph{emigration}, which reads
\begin{equation}
\label{e1478}
\CL^{\ast, {\rm FV},c,\theta,(n),{\rm em-im}} \Phi^{(n),\varphi}
= \sum_{k \in [n]} \big(\langle\nu^{(n)}_{(k)},\Phi^{(n),\varphi}\rangle
- \langle\nu^{(n)},\Phi^{(n),\varphi}\rangle \big),
\end{equation}
where $\nu_{(k)}$ arises from $\nu$ by sampling the mark of the $k$-th member from a source $S$ with measure $\mu \otimes \theta$ on $\CI^\ast \otimes (S)^\ast$ rather than $\mu \otimes \kappa$.

Note that if in the formulas above we instead sample \emph{with replacement}, i.e., draw from $\mu^{\otimes n}$, then we get
\begin{equation}
\label{e2210}
\wh \CL^{\ast, \rm res,d}\Phi^{(n),\varphi}
= \frac{d}{2} \sum_{k,\ell \in [n]} \big(\langle \nu^{(n)},
\varphi \circ \theta_{k,\ell} \rangle - \langle \nu^{(n)},\varphi\rangle \big) + O\left(\frac{1}{n}\right).
\end{equation}

%%%

\subsection{Construction: Fleming-Viot and Dawson-Watanabe grapheme operators}
\label{ss.go}

In order to write down the martingale problem arising in the $n \to \infty$ limit for the processes introduced in Section~\ref{sss.fincase}, we define the \emph{operator} for the $\wt \G^{[]},\wt \G^{[],V}$-valued Fleming-Viot process ($\sim\,=$ no bar) and Dawson-Watanabe process ($\sim\,=-$). This dynamics has two components, one from the \emph{abstract graph} (evolution of the connection-matrix distributions) and one from the evolution of the spaces in which we \emph{embed} the abstract graph, here associated with elements from  $\U_1,\U$, respectively, $\U_1^V,\U^V$. For both components of the evolution we have to specify the operators of the martingale problem.

We start with the pure Fleming-Viot and Dawson-Watanabe operator, where the \emph{second-order term} represents the \emph{essence of the grapheme diffusion process}. We first turn to the \emph{dynamics} of the \emph{connection matrix}, which induces a dynamics for the functional $\nu$ on $\G^{[]},\G^{[],V}$, and deduce from that the operator on $\G^{\langle\rangle},\G^{\langle\rangle,V}$. Afterwards we turn to the evolution of the space in which we embed the vertices and combine this with the operator on $\G^{[]}$. Finally, we turn to the \emph{first-order terms}.

%%%

\subsubsection{Fleming-Viot and Dawson-Watanabe operators on functions of connection matrix}
\label{sss.flemwa}

Consider test functions on $\G^{[]}$ of the form
\begin{equation}
\label{10}
\Phi (G) = \Phi \left([\CI,h,\mu]\right) = \int_{\CI^m} \varphi\left(\left(h(u_i,u_j)\right)_{i,j \in [m]}\right)
\mu(\d u_1)\ldots\mu(\d u_m), \quad \varphi \in C_b \left([0,\infty)^m\right),
\end{equation}
with $G \in \G^{[]}$. We denoted the set of test functions of this form by $\Pi^\ast_{\rm mon}$ and the generated algebra by $\Pi^\ast$. Consider the map $\wt \theta_{i,j}$ acting on $(u_1,\ldots, u_m)$ via the replacement of $i$ by $j$ defined by
\begin{equation}
\label{e911}
\wt \theta_{i,j}\big(u_1,\ldots,u_m\big) = \big(u_1,\ldots ,u_{i-1},u_j,u_{i+1},\ldots, u_{j-1},u_j,u_{j+1},\ldots,u_m\big).
\end{equation}
Thus, $\wt \theta_{i,j}$ maps an $m$-sample to an $(m-1)$-sample, and $\wt \theta_{i,j}$ induces the map $\theta_{i,j}$ on matrices $(h(u_k,u_\ell))_{k,\ell \in [m]}$. Set, with $\underline{u}=(u_1,\ldots,u_m)$,
\begin{equation}
\label{12}
\CL^{\ast,\rm FV,res,d}\Phi = \sum_{i,j \in [m]} \int_{\CI^m}
\left[\varphi\left(\theta_{i,j}(h(u_k,u_\ell)_{k,\ell \in [m]})\right)
-\varphi\left((h(u_k,u_\ell))_{k,\ell \in [m]}\right)\right] \mu^{\otimes m}(\d\underline{u}),
\end{equation}
where we note that it suffices to consider $\varphi$ of the form in~\eqref{e976}. This extends easily to an operator that acts on $\Pi^\ast$ and determines the evolution of $(\nu_t)_{t \geq 0}$.

\begin{remark}[Lifting from $\G^{[]},\G^{\{\}}$ to $\G^{\langle\,\rangle}$]
\label{r.1429}
{\rm Note that in~\eqref{10} for $G \in \G^{[]}$ we can write, for a unique $G^{\langle\,\rangle} \in \G^{\langle\,\rangle}$,
\begin{equation}
\label{e.1431}
\Phi(G)=\int_{\{0,1\}^{m^2}} \varphi(\uuh)\, \nu^{(m)}(\d \uuh)
=\Phi^{\langle\,\rangle}(G^{\langle\,\rangle}),
\end{equation}
with $\Phi^{\langle\,\rangle}$ a polynomial on $\G^{\langle\,\rangle}$, i.e., $\Phi^{\langle\,\rangle} \in \Pi^{\ast,\langle\,\rangle}$. Therefore in~\eqref{12} we can define an operator $\CL^{\langle\,\rangle, \rm FV}$, with the operation $\theta^{\langle\,\rangle}_{i,j}$ instead of $\theta_{i,j}$, to obtain the
\begin{equation}
\label{e1435}
(\CL^{\langle\,\rangle,\rm FV},\Pi^{\ast,\langle\,\rangle},\Gamma^{\langle\,\rangle})
\text{-martingale problem on $\G^{\langle\,\rangle}$},
\end{equation}
which is solved by a solution of the $(\CL^{\ast, \rm FV},\Pi^\ast,\Gamma^\ast)$-martingale problem \emph{lifted from $\G^{[]}$ to $\G^{\langle\,\rangle}$}.}
Recall, however, that $\G^{\langle\rangle}$ is not a Polish space, so standard theory may not apply, or may apply only after restriction to a subspace, for example, $\G^{\langle\rangle}_{\rm ultra}$, the equivalence classes in $\G^{\langle\rangle}$ induced by elements in $\G_{\rm ultra}^{[]}$.
\hfill$\spadesuit$
\end{remark}

Next, we must \emph{extend} this operator from $\Pi^\ast$ to $\wh \Pi^\ast$ by using the well-known operator on $\wh \Pi^{\ast,\downarrow}$ for the dynamics $(\CU_t)_{t \geq 0}$ on $\U_1$, respectively, $\U$, and lifting these to operators on $\G^{[]}$, respectively, $\bar \G^{[]}$, and also their marked versions on $\G^{[],V},\G^{\{\},V}$, in order to obtain the evolution of the \emph{embedded grapheme} via the \emph{evolution of the random environment} (specification of the embedding). Of course, we also have to pass to $\wh \Pi^\ast$, the algebra where the joint law of $\nu$ \emph{and} $\CU$ is determined.

%%%

\subsubsection{Operators for functionals of the genealogy}\label{sss.op}

We next specify the operator $\wh \CL^\ast$ on $\wh \Phi^{\ast,V} \in \wh \Pi^{\ast,V}$, i.e., we \emph{extend the operator} to the \emph{algebra $\wh \Pi^{\ast,V}$} generated by $\Pi^{\ast,V} \cup \Pi^{\ast,V,\downarrow}$. We recall the operator for the \emph{genealogy-valued Fleming-Viot process} $\CU$ associated with our graphemes, which is the sum of four terms, and extend it to $\wh \Pi^\ast$ or $\wh \Pi^{\ast,V}$, which will be immediate after we look at the object in the right way, namely, by using that $h$ is given by a specific distance function in our definition.

The limiting genealogical process $\CU$ for Fleming-Viot and Dawson-Watanabe population processes have been constructed as follows (see \cite{GrevenPfaffelhuberWinter2013,DG19,ggr_GeneralBranching}). On the domain $\wh \Pi^{\ast, V,\downarrow}$ (i.e., the functions in $\wh \Pi^\ast$ restricted to functions not depending explicitly on $h$), call this operator $\wh \CL^{\ast,\rm FV}$. This specifies a $(\wh \CL^{\ast,\rm FV},\wh \Pi^{\ast,V,\downarrow}, \wh \Gamma)$-martingale problem for the process $\CU$ on $\U_1$ or $\U$, for which a solution has been constructed in the above references and for which a \emph{duality relation} has been checked, ensuring uniqueness of solutions. The dual is a \emph{function-valued} process driven by a \emph{coalescent}. At the same time, we obtain a martingale problem for the process $(\nu_t)_{t \geq 0}$, with $\nu_t$ arising as a functional induced by $h_t$ for a given measure space $((\CI_t,r_t),\mu_t)$ (see the references above). The four terms of the operator are as follows.

\medskip\noindent
(I) The so-called \emph{growth operator}, with growth referring to the increase of distances with time, is obtained after taking the test function defined earlier, replacing $h$ in~\eqref{10} by $r$, and acting on the test function as in~\eqref{e1462}.

\medskip\noindent
(II) The \emph{resampling operator} $\CL^{\ast, V, \rm res}$ defined above describes, in the $\U^V_1$-valued processes, the change of the genealogy due to the birth of new individuals, at distance $0$ from the ancestor, and the sharing with other individuals of the distance to the most recent common ancestor. This is a second-order operator (acting on products of test functions in the same way as differential operators act by the product rule). The extension for the limit process runs well when we work with $(\CI,r)$ (hence on $\G^{[]}$), since in that case we have a similar form for the extended operator $\wh \CL^{\ast, {\rm FV}, {\rm res}}$, namely,
\begin{equation}
\text{replace connections $h$ by distances $r$ in \eqref{10} and \eqref{12} to get }
\CL^{\ast,{\rm FV}, {\rm res}} \text{ on } \wh \Pi^{\ast,V,\downarrow}_r,
\label{e1614}
\end{equation}
where the subscript $r$ indicates that the functions do \emph{not} depend on $h$. In this way we also get the dynamics of the spaces in which we embed our abstract graph dynamics, which gives us the extension of the second-order term $\CL^{\ast,V,{\rm res}}$. To justify the latter, we use the approximation with finite processes, for which it was proved in~\cite{GrevenPfaffelhuberWinter2009} that the same formula holds (by using an observation we explain in Section~\ref{sss.func}).

\medskip\noindent
(III) The operator of \emph{emigration} and \emph{immigration} is first order, and is defined on a subclass of the polynomials on $\CU_t$ by~\eqref{e1478}. This definition works on $\wh \Pi^{\ast,V}$ and its subsets. As above, we can use \cite{DG19} to obtain the formula on $\wh \Pi^{\ast,\downarrow}$ on $\bar \G_n, \bar \G_\infty$. We have to first do this on $\Pi^\ast$, and afterwards lift to $\wh \Pi^\ast$. Here, we use again the approximation by finite-$n$ models explained in the above reference.

\medskip\noindent
(IV) The first-order operators capturing \emph{mutation} and \emph{selection}, denoted by $\CL^{\ast, \rm FV,mut}$ and $\CL^{\ast, \rm FV,sel}$, respectively, are defined on $\wh \Pi^{\ast,\downarrow}$ in the literature (see~\cite{DGP12}). We have to extend these operators first to $\Pi^\ast$ first, and afterwards to $\wh \Pi^\ast$. We begin with the action on $\wh \Pi^{\ast,\downarrow}$.

The ingredients are the rates $m$ and $s$, the mutation kernel $M(b,db^\prime)$ from $K$ to $(K,\CB_K)$ in~\eqref{e1687}, and the fitness matrix or fitness function $\chi\colon\,K \to [0,1]$ from~\eqref{e1925}. Finally, we have to extend from $\wh \Pi_r^{\ast,V}$ to $\wh \Pi^{\ast,V}$, where now \emph{products} of functions $\varphi_h$ and $\varphi_r$ depending on $h$, respectively, $r$ appear.

In the sequel we use the polynomials defined in~\eqref{e.855}. For the \emph{selection operator} on $\wh \Pi^{\ast,\downarrow}$ we have
\begin{equation}
\label{e2162}
(\CL^{\rm sel}\Phi)(G) = s \sum_{k=1}^n \int_{(U \times V)^{n+1}} \varphi(\uur (\underline{\iota}))
\left[\psi(\underline{v}) \chi (v_k)-\psi(\underline{v})\chi(v_{n+1})\right]
\mu^{\otimes n+1}_{U \times V} \left(d(\underline{\iota},\underline{v})\right).
\end{equation}
The \emph{mutation operator} on $\wh \Pi^{\ast,\downarrow}$ reads as
\begin{equation}
\label{e2146}
(\CL^{\rm mut}\Phi)(G)=\sum_{k=1}^n \left[(M_k \Phi)(G)-\Phi(G)\right],
\end{equation}
where, for $k \in [n]$, we set
\begin{equation}
\label{e2150}
\left(M_k \Phi \right)(G)
= m \left\lbrace \int_K \Phi_k (G^{b^\prime}) M(b_k,db^\prime)-\Phi (G) \right\rbrace,
\end{equation}
with
\begin{equation}
\label{e.2154}
\Phi_k (G)=\int_{(U \times K)^n} \varphi(\uur(\underline{\iota}))\,
\psi((\underline{b}^k))\, \mu^{\otimes n}_{U \times V} (d(\underline{\iota},\underline{b}^k)),
\qquad \underline{b}^k = \left(b_1,\ldots,b_{k-1},b^\prime,b_{k+1},\ldots,b_n\right).
\end{equation}
On $\Pi^\ast$ the operator acts as above, leaving $\varphi(\uuh)$ fixed but acting on $\psi (\underline{b})$. We extend this to $\wh \Pi^\ast$ by following the same principle as before.

%%%

\subsubsection{Dawson-Watanabe operator}

It remains to specify the Dawson-Watanabe operator. Here, only the second-order term changes, namely, in~\eqref{12} the second part with the minus term is dropped (see~\cite{DG19} for details).

%%%

\subsection{Extension to non-completely connected components}
\label{ss.ext}

Here the states are of a somewhat different nature, even though in the background there still is an autonomously fluctuating process in the form of a decomposition of the vertices into open and closed (potentially completely connected) subsets such that, \emph{given} the whole \emph{(time-)path}, the state is given by a \emph{pruning} of the edges of the \emph{potential edges} by the noise, where the potential edges are given by the processes with completely connected components studied previously.

We look at such \emph{edge process} in a \emph{random medium}. This means that, starting from our desired process $\CG$, we now take the process that is associated with the $\G^{[]}$-valued process for the case where $a^+=a^-=0$, which we call $(\CG_t^{(0,0)})_{t \geq 0}$ and which is associated uniquely with a process $(\CU_t)_{t \geq 0}$ (see~\eqref{e678}). Then we take a field $\CA_t$ indexed by $U_t \times U_t$ of i.i.d.\ $\{0,1\}$-valued Markov spin-flip chains with rates $a^+,a^-$ to generate a representative of the state of the desired process $\CG^{(a^+,a^-)}$. To make this rigorous, we have to argue on $\G^{[]}$ and write down the operator of $\CG^{(a^+,a^-)}$. We modify the action of this operator on $\Pi^\ast$ (the $h$-dependent functions) and also modify~\eqref{e678} by adding a further requirement on the evolution via the field $(\CA_t(u_1,u_2))_{t \geq 0}$, namely, that no flip has occurred.

The evolution of the condition is well understood, since it is simply the underlying $\U_1$ (or $\U$)-valued process $(\CU_t)_{t \geq 0}$, which is autonomous since there is no feedback from the status of being an active or a non-active edge to the underlying process $\CG^{(0,0)}$. The effect of the pruning is captured in a \emph{conditional} martingale problem involving an induced process for the connection-matrix distributions. Given the condition, this is a distribution that arises as the pruned distribution of $(\nu_t)_{t \geq 0}$ for the solution $(\CG_t^{(0,0)})_{t \geq 0}$ of the martingale problem \emph{without} the additional mechanism of insertion and deletion, where the frequency of pruning is given via the value $a_t$ of the frequency of current virtual states among all eligible states. We have to justify this splitting of the dynamics.

For the finite approximations, we can split the process into an autonomous process $\CU^n$, creating the potential edges, and an evolution of the edges by the field $\CA$ acting on the potential edges as a spin-flip dynamics, giving $\CG^n$. The evolution of the edges does not influence $\CU^n$. Therefore the splitting of the condition (autonomous) and the flipping of virtual edges works well.

We have to show that this also holds in the limit dynamics, which can be deduced by using \emph{only} the martingale problem. For that we observe that the operator can be split into the action of the operator corresponding to $\CG^{(0,0)}$ and the operator acting on functions of the state of the edge, i.e., the specification active or virtual on functions that are products of functions of each of the components, one of which observes the active-virtual information, representing the spin-flip dynamics. On $\wh \Pi^\ast$ restricted to functions not depending on $h$,  i.e., $\wh \Pi^{\ast,\downarrow}$, we find the generator of $\CU$. On $\Pi^\ast$ we obtain the spin-flip generator acting on $h$. On $\wh \Pi^\ast$ it suffices to consider the product of a function from $\wh \Pi^{\ast,\downarrow}$ and a function from $\Pi^\ast$.

The additional operator acts on functions of the form as in~\eqref{e.855} like
\begin{equation}
\label{e.2147}
(\CL_t^{\rm ins,del} \Phi) (G)= \bigg(a^+\sum_{k,\ell=1}^n (\Phi^{+,t}_{k,\ell}-\Phi)
+a^- \sum_{k,\ell=1}^n (\Phi^{-,t}_{k,\ell} -\Phi)\bigg)(G),
\end{equation}
where, for $G=(U,r,h,\mu)$, we define at time $t$, with the states at time $t$,
\begin{equation}
\label{e.2151}
\Phi^{+,t}_h(G)=\int_{U^n} \varphi_h \big(\uuh^{k,\ell,\sim}(\uu)\big)
\varphi_r(\uur(\uu)) 1(r(u_k,u_\ell)<2t) \mu^{\otimes n} (d \uu), \qquad  \sim\, \in \{+,-\},
\end{equation}
and
\begin{equation}
\label{e2155}
\big(\uuh^{k,\ell,+}(\uu)\big)_{(i,j)} =1 \; \text{ for } \; k=i,\,j=\ell,\,\sim\,=+,
\end{equation}
and
\begin{equation}
\label{e2158}
\big(\uuh^{k,\ell,\sim}(\uu)\big)_{(i,j)} = \uuh (\uu) \text{ otherwise},
\end{equation}
and similarly for $\Phi^{-,t}$.

We can avoid \emph{time-dependent terms} if we introduce \emph{marks} characterising the \emph{founder} of a family surviving at time $t$ for all $t >0$ and requiring identity of type in the indicator of~\eqref{e.2151} instead of having the same time-$0$ ancestor.

%%%%%%%% SECTION 5 %%%%%%%%%%%%%%%%

\section{Proofs for state spaces and topologies}
\label{s.state}

In Section~\ref{ss.algeb} we derive the properties of the \emph{algebras of test functions} $\wh \Pi^\ast, \wh \Pi^{\ast,V}$ on the spaces $\wt \G^{[]},\wt \G^{[],V}$, respectively, $\wt \G^{\langle\rangle}, \wt \G^{\langle\rangle,V}$, with $\sim$ being bar or no bar. In Section~\ref{ss.polish} we derive the properties of the topologies that were introduced, in particular, the property of being \emph{Polish}.

%%%

\subsection{Algebra of test functions}
\label{ss.algeb}

In this section we prove Proposition~\ref{prop.928}--\ref{prop.1743} and Proposition~\ref{prop.1}--\ref{prop.1202}. To do so, we exploit the literature for the respective properties on $\U_1,\U,\U_1^V,\U^V$, respectively, the more general version on $\M_1,\M,\M^V_1,\M^V$, of which the former are closed subspaces (see \cite{GrevenPfaffelhuberWinter2009,DepperschmidtGrevenPfaffelhuber2011,DG19}). What remains on $\G^{[]}, \G^{[],V}$ (or $\bar \G^{[]},\bar \G^{[],V}$), and is not covered by the literature on genealogies of populations for functions on $\M_1,\M$ and their marked versions, is to provide the same information for functions of $[U_t,r_t,\mu_t]$ that are used to study the \emph{connection-function part}, i.e., the functional $(\nu_t)_{t \geq 0}$ induced by $h$, and to get the complete object $([(\CI_t,r_t), h_t,\mu_t])_{t \geq 0}$.

However, on $\G^{[]}_{\rm ultr}, \bar \G^{[]}_{\rm ultr}$ the function $h$ can be viewed as an ultrametric, and so $(\CI,r,\mu)$ and $(\CI,r,h)$ are of a similar structure, so that we can get the properties of both $\wh \Pi^{\ast,\downarrow}$ and $\Pi^\ast$ from the literature on $\U_1,\U$. It remains to get the combined structure from the algebra that is generated by taking products of elements of the two algebras, as studied in~\cite{DGP23} in detail. In our context we can use that the $h$-component is a continuous functional of the $r$-component of the state. We have to also extend to $\bar \G$. However, since $\bar \G$ consists of objects that are pairs of a total mass component and an element of $\G$, this can be achieved in the same way as for the space $\U$ from $\U_1$ (see, for example, \cite{DG19,ggr_GeneralBranching}). In the general case $\G^{[]},\bar \G^{[]}$, we use the results on $\M_1,\M$ (recall~\eqref{S3} and the sequel).

%%%

\subsection{Polish spaces}
\label{ss.polish}

In this section we prove Propositions~\ref{prop.813} and~\ref{prop.1202}. We have to show that, by adding $h$ to the structure $((U,r),\mu)$ to get $[(U,r,h,\mu)]$, we get a Polish space. This fact is based on the knowledge that the equivalence classes of a metric measure space $[U,r,\mu]$ from $\U_1$ or $\U$ form a Polish space in the Gromov weak topology \cite{GrevenPfaffelhuberWinter2009}. We have that also $\G^{[]}_{\rm ultr}$ is a Polish space in the topology specified in Section~\ref{s.evol}. Namely, since $h$ is just inducing a second ultra-metric by defining the $h$-path connected components as the $1$-balls, this space is like a \emph{multi-metric} measure space (as studied in~\cite{DGP23}), for which the claim holds. Note, however, that here we are in the simpler setting of $\G_{\rm comp} \subseteq \G_{\rm ultr}$, with the $h$ that we used for our models, because one is a coarse-graining of the other, so that convergence of the $r$-topology implies convergence of the $h$-topology on those sets of states (because we have ultrametric spaces). The same applies in the more general setup where $h$ can be viewed as a metric but not an ultrametric, where the second metric generates a coarser Borel-algebra.

%%%%%%% SECTION 6 %%%%%%%%%%%%%%%

\section{A prelude on duality}
\label{s.duality}

In this section we consider the dynamics without additional insertion/deletion and introduce: (i) the dual process $\CC$ driven by the coalescent process $C^t=(C^t_s)_{s \in [0,t]}$, defined for every time horizon $t$, (ii) the duality function $H$. In Section~\ref{s.nccomp} we will extend the duality to the general setting. 

%%%

\paragraph{(i) The dual process.}
\iffalse{The dual process is driven by the coalescent, together with an \emph{entrance law} at time $0$ from the state given by $\{\{n\}$, $n \in \N\}$, based on the Kingman coalescent starting with countably many initial individuals. The latter defines a \emph{$\G^{[]}$-valued} dynamics. We also need a representation of the dual $\CC$ that reformulates the one used for the strong duality, since the latter is not convenient to formulate the $H$-duality. This process is an $n$-coalescent combined with an $n$-distance matrix and an $(n \times n)$-connection matrix, which is an equivalent coding of the $\G_n^{[]}$-valued grapheme associated with the coalescent.}\fi

The state space for the dual process is $E^\prime_n=S_n \times D_n$, with $S_n$ the set of partitions of $[n]$, $D_n$ the set of $(n \times n)$-$\mathbb{R}_+$-valued distance matrices. The dynamics is such that each pair of partition elements, \emph{coalesces} at rate $d$, and the distance of each pair of different vertices \emph{grows} at speed~$2$. This defines an $E^\prime$-valued Markov jump process, $(\CC_s=(c_s,\uur_s,\underline{\underline{h}}_s))_{s \ge 0}$, with a deterministically evolving part (distance growth). Its generator is given as follows: for $(C,\underline{\underline{r}}) \in S_n\times D_n$, $g:S_n\times D_n$ bounded and continuously differentiable in distances,
\begin{equation}
\label{e:001}
\begin{aligned}
&G^{\mathrm{dual}}g\big(C,\uur\big)\\
&= d\sum_{\varpi_1\not=\varpi_2\in C}\Big(g\big(C\setminus\{\varpi_1,\varpi_2\}
\cup\{\varpi_1\cup\varpi_2\},\uur\big)-g\big(C,\uur\big)\Big)
+ 2\sum_{1\le i<j\le n}\frac{\partial}{\partial r_{i,j}}g\big(C,\uur\big)\mathbf{1}_{i \nsim_C j},
\end{aligned}   
\end{equation}
where as usual we write $i\sim_c j$ if and only of there is $\varpi\in c$ with $i,j\in\varpi$.

\paragraph{(ii) The duality function.}
The duality function combines the duality relation for the two components of $\CG_t$, namely, for $\CU_t$ in $\CC_t^t$ and for the connection function $h$ in $\CC_t^t$ (the first is well known in the literature, the second is specific for graphemes). Both have to be integrated into a single duality function, link two $\G^{[]}$-grapheme processes, and give the strong duality.

First, we define the duality function for the process $(\CU_t)_{t \geq 0}$ and the dual $(C^t_s)_{s \in [0,t]}$. This is obtained by putting in the duality function for $\CG$ and $(C^t_s)_{s \in [0,t]}$ in~\eqref{e2174} below the function $\varphi_h \equiv 1$. Define $H\colon\, E \times E^\prime \to \R$, with $E=\G^{[]}$ and $E^\prime= \mathop{\bigcup}\limits_{n \in \N} E^\prime_n$, as
\begin{equation}
\label{e2174}
H\big([U,r,h,\mu],(C,\uur^\prime)\big):=\int_{U^n} \varphi_r \left(\uur^C (\uu)
+\uur^\prime\right)\varphi_h\left((h^C(i,j))_{1\le i<j\le n}\right)\, \mu^{\otimes n} (d \uu),
\end{equation}
where $n\in\mathbb{N}$ such that $C\in S_n$, $r^C$ is the distance matrix of the sample $\uu$ given by $r^C_{i,j}=r_{C(i),C(j)}$, with $C(i)$ being the partition element containing $i \in [n]$, and $h^C(i,j)=1$ if and only if $j \in C(i)$.

Next, we turn to $h$, and hence to the functional $(\nu_t)_{t \geq 0}$, and combine them into a duality for the process to $(\CG_t)_{t \geq 0}$. The duality w.r.t.\ $H$ as well as the strong duality for $\CG$ are, for \emph{our} dynamics, in fact consequences of the known \emph{strong} duality relation for $\CU$. Recall that the connection function arises as a \emph{function} of the distances, both for the process and for the dual process. Namely, a decomposition of the sets $U_t,\wt U_t$ (for process, respectively, dual process) gives the completely connected components as $2t$-balls in both cases. These are given by the disjoint decomposition into open and closed $2t$-balls, represented by partition elements in the dual, which have a closure that is a \emph{completely connected component}, and their (disjoint) union is $\wt U_t$. Therefore the strong duality for $\CU$ and $\CC$ \emph{implies} the \emph{strong duality of the grapheme process} $\CG$ and $\CC$ in our setting.

However, if we want to get information about the uniqueness of the martingale problem for $\CG$, then we need to obtain the equivalence of $[U_t,r_t,h_t,\mu_t]$ and $[\wt U_t,\wt r_t,\wt h_t,\wt \mu_t]$ by checking that expectations of the functions given by $H$ are the same, \emph{just} using the forward and the backward (dual) operators (see~\eqref{e2189} below). This has been done in the literature for the $\CU$-part, because we saw that the $h$-part has the same structure, so that the relation must hold for our operator on $\G^{[]}$, because the merging of variables allows us to extend from $\wh \Pi^{\ast,\downarrow}$, and from $\Pi^\ast$ functions to products, and hence to $\wh \Pi^\ast$. The general criterion we use here is as follows. If the family $\{H(\cdot,c),c \in E^\prime \}$ is law determining, then we have to check the following generator relation (which is done in~\cite{GrevenPfaffelhuberWinter2013} for our case):
\begin{equation}
\label{e2189}
(G^{a,c,\theta} H)(\cdot,c)(u) = (G^{\rm dual}H)(u,\cdot)(c), \qquad u \in E,c \in E^\prime,
\qquad a=d,
\end{equation}
while for $a=b$ we must add the term $\mathfrak{V}(\cdot)H(u,c)$ to the right-hand side. 

%%%%%%%%%% SECTION 7 %%%%%%%%%%%%%%%%%%%

\section{Proofs for grapheme diffusions with completely connected components}
\label{s.nccomp}

In Section~\ref{ss.mainideas} we explain the main ideas, list the tasks that remain, and give the key arguments for Theorems~\ref{th.963-I},~\ref{th.963-II},~\ref{th.963-III},~\ref{th.853} and~\ref{prop.1078}. In Section~\ref{ss.proofcor} we prove their corollaries. Theorem~\ref{prop.1078} is extended in Section~\ref{ss.add}, and details on the references are worked out in Section~\ref{ss.proofth1}.

%%%

\subsection{Main idea for the proof}
\label{ss.mainideas}

As in~\cite{AHR21}, for the special case where $c,\theta$ are both $0$ we have chosen simple evolution rules for finite graphs $([n],h,\mu)$ arising from population dynamics of individuals in $\CI$. In addition to this reference, we now use that we also have complete information about the evolution of the embedding space $([\CI_t,r_t,\mu_t])_{t \geq 0}$ process. Namely, these processes have been treated in the literature, as far as the analogues of Theorems~\ref{th.963-I} and \ref{th.853}, and Theorem~\ref{th.1838} for the genealogy $\CU$, are concerned.

We must complete the proof of the claims on $\G^{[]}$, where the \emph{connection function $h$} is an additional part besides the \emph{genealogy} $\CU$. Moreover, the \emph{joint law} of these two components has to be treated also. This has to be done by using that the process $\CG$ is given as a \emph{functional} of the \emph{genealogy process} $\CU$ associated with our population genetic models, namely, $(\CU_t^{a,c,\theta})_{t \geq 0}$ with $a=b,d$ for Fleming-Viot, respectively, Dawson-Watanabe. In the sequel we specify the theorems in the literature that relate to our theorems, and handle the two problems just mentioned. We do so in the last two points in Section~\ref{sss.process}, where we exhibit the specific structure of the processes we deal with, which substantially facilitates the task and the details of the argument given in Section~\ref{sss.func}. Additional ideas are needed in the presence of mutation, which are postponed to Section~\ref{ss.add}.

%%%

\subsubsection{The process: from genealogies to graphemes}
\label{sss.process}

In this section we look at the process $\CU_t=( [\CU_t,r_t,\mu_t])_{t \geq 0}$ associated with our models, recall their characterisation as solution to a well-posed martingale problem, and describe path properties, equilibria and convergence to them and converging finite approximations, in order to see how we can include the connection function $h$ into the picture. This involves both initial states and dynamics.

A key property of the dynamics is that a large part of $\G^{[]}$ is transient and is left instantaneously toward $\G^{[]}_{\rm comp}$, and the states are completely connected components built by the time-$t$ descendants of finite or countably many  (for $c,\theta >0$) founders from the time-$0$ population. We can start in the marked case in \emph{any} initial configuration of types, but at positive times our evolution leads to a state with finitely or countably many types, and consequently to completely connected components only. This means that our evolution takes place in a subset of $\G^{[],V}_{\rm ultra} \subseteq  \G^{[],V}$ and, in fact, stays in $\G^{[],V}_{\rm comp}$ when started there:
\begin{equation}
\label{e.2000}
\G^{[],V}_{\rm comp},\G^{[],V}_{\rm ultra} \text{ are \emph{dynamically closed} subsets of } \G^{[],V}.
\end{equation}
Only the latter is \emph{closed topologically}.

%%%

\paragraph{(i) Initial states.}
First we have to decide what \emph{initial states} $[(\CI_0,r_0),h_0,\mu_0]$ are \emph{possible} for the grapheme evolution, since we have defined it as functional of $\CU$. It turns out that we can choose the initial state only with $\CU_0 \in \U_1$, respectively, $\U$ for reasons of \emph{consistency}. Namely, in order to obtain \emph{c\`adl\`ag paths} of graphemes, we have to take $h$ such that the initial $h_0$ is determined by $\CU_0$. In particular, we still have to give a version of the formula relating $\CG$ to these initial states, which are different from the single-root case for $\CU_0$.

First, we consider a special case: start the genealogical part in $[\{1\}, \underline{\underline{0}},\delta_1]$, i.e., \emph{all} initial particles form a single root, and we follow the arguments for the dynamics based on \eqref{e678} to get $\CG$ from $\CU$. How do we get general initial states? We embed the process $\CU$ in a $[0,1]$-marked process $\CU^{[0,1]}$. We \emph{dissolve} the root by taking as type space $V=([0,1],\CB_{[0,1]})$ and starting in $\mu_{U \times V,0}=\delta\otimes \pi$, with $\pi$ the Lebesgue measure on $[0,1]$, so that every initial individual gets a different type. In this way, at time $t>0$ we have finitely or countably many atoms in $\mu_{U_t \times V,t}$, and this enables us to distinguish all \emph{founders} by their type in $[0,1]$. We call founders at time $s >0$ all individuals having descendants at some time $t>s$. In other words, introducing \emph{inheritable types} allows us to simulate all different $\CU_0$-states by using a marked model and identifying specific marks.

We have to adapt \eqref{e678} when we do \emph{not} start the process $\CU$ in a single root. First, if the initial state is a more general $\CG_0=[(U_0,r_0),h_0,\mu_0] \in \G_{\rm comp} \subseteq \G^{[]}_{\rm ultr}$, then we add a further condition, namely, that the types of two vertices have to agree if and only if there is an edge between them. We then define the process $\CG$ by using the functional $\CU^\ast_t$ instead of $\CU_t$ in the analogue of~\eqref{e678}, where $\CU^\ast_t$ is the $t$-truncation of $\CU_t$. Therefore $h_0$ has a form that represents the partition of the set of founders for the later states at time $t >0$ in completely connected components (i.e., $1$ if and only if both arguments are from one element of the partition), so that components of the first initial condition are merged according to $\CG_0$. This gives us the Fisher-Wright model of~\cite{AHR21}, also for non-trivial $c,\theta$.

For general $h_0$, we decompose $h_0=h_0^1+h_0^2$, with $h_0^1$ associated with completely connected components and $h_0^2$ the remainder, and we follow the diffusion corresponding to $h^1_0+\wt h^1_0$, with $\wt h^1_0$ the $0$-component on the points where $h^0_2$ was supported (which is \emph{an instantaneous jump} to a state treated above). This gives us a c\`adl\`ag path, but one that jumps at an infinite rate at time $0$.

\paragraph{(ii) Dynamics.}
We next formulate our basic grapheme diffusions arising from graph evolutions, as described in Section~\ref{ss.examples} and above. Take the $\G^{[],V}$-valued or $\bar\G^{[],V}$-valued processes $\CG$, driven by the $\U^V$-valued processes $(\CU_t^{d,c,\theta})_{t \geq 0}$, $(\CU_t^{b,c,\theta})_{t \geq 0}$ as in Definition~\ref{def.graphproc}, and \eqref{e678}:
\begin{itemize}
\item
Fleming-Viot process with rate $c \geq 0$ of immigration and emigration from a source $[[0,1], \text{Euclidean distance}, \theta]$, $\theta \in \CM_1([0,1])$, and sampling rate $d\colon\,\CG^{{\rm FV},c,d,\theta}$.
\item
Add selection to get $\CG^{{\rm FV},c,d,\theta,s,m}$ in the case $m=0$, written $\CG^{\rm FV,c,d,\theta,s}$, which allows for the same reasoning as above despite the marks, since the latter are passed on to the descendants when $m=0$.
\item
Dawson-Watanabe process with rate $b >0$ on $[0,1]$, with rate $c \geq 0$ of immigration and emigration, source $\theta \in \CM_{\rm fin}([0,1])$: $\CG^{{\rm DW},c,b,\theta}$.
\end{itemize}
We need information on $\CU$. To cover non-trivial $c,\theta$, we use~\cite{GrevenPfaffelhuberWinter2013,DGP12,DG19} and~\cite{GSW2016} to get \emph{existence} and \emph{uniqueness} of the solution to the martingale problem and \emph{continuous paths} of the $\U_1$-valued, respectively, $\U$-valued \emph{driving processes} $\CU$ and their \emph{genealogy} underlying $\CG$. Therefore, in our definition of $\CG$ with $h_t$ based on $\CU_t$ at each time $t$, we have obtained as grapheme diffusion a \emph{unique stochastic processes} with values in $\G^{[]}$ with \emph{continuous paths} (since $h$ is continuous in that case), and it is known from the references mentioned that $(\CU_t)_{t \geq 0}$ has the strong Markov and Feller property, and for $t \to \infty$ converges to a limiting \emph{equilibrium} state. A further observation is the following: the process $\CU$ is approximated in \emph{path space} by individual-based models (with $n$ individuals) in the limit as $n \to \infty$. By our formula relating $\CU$ and $\CG$, we have, in particular, also defined a \emph{finite approximation} of $\CG$ via grapheme processes with values in $\G_\infty$, called $\CG^n$, in path space. For detailed references, see~Section~\ref{ss.proofth1}.

In order to prove Theorems~\ref{th.963-I},~\ref{th.963-II} and ~\ref{th.963-III}, we want to view the processes induced by $\U$ as \emph{strong Markov processes} solving a \emph{well-posed martingale problem} on state space $\G^{[]}$, and get the $t \to \infty$ limit as an \emph{equilibrium} of a Markov process. Furthermore, this process has to be proven to be second-order, i.e., to be a \emph{diffusion}. Here, the order is defined by following the generalisation of the concept of order of differential operators to operators acting on functions $\G^{[]},\G^{\{\}},\G^{\langle\,\rangle}$ based on $\U_1$ and $\U$ in~\cite{DGP12}, which we recalled in Remark~\ref{def.1947}. The issue is whether the criteria that were used hold for the lifted operators as well. This means that, based on our definition of a version of $\CG$ for a given genealogy process $\CU$, we need to show that the respective properties on $\U_1$ or $\U$ indeed \emph{imply} the properties of the process $\CG$ on both $\G^{[]},\G^{\langle\rangle}$, In turn, this means that we need to identify the proper quotations of results on $\CU$ and formulate the \emph{implication} for the evolution of the \emph{connection-matrix distribution} process and the \emph{joint law} of the process $(h_t)_{t \geq 0}$ with the underlying $(\CU_t)_{t \geq 0}$. The transfer arguments are set up as follows.

The key points in the argument for our grapheme processes are the following: (i) lift the results on $\CU^{d,c,\theta},\CU^{b,c,\theta}$ from $\U_1,\U$ to $\CG^{[]}$ in order to obtain a solution to the martingale problem for $\CG^{[]}$, by taking a functional from the underlying $\U$-valued process, as specified in~\ref{def.graphproc} and~\eqref{e678}, respectively, for general initial conditions; (ii) establish \emph{duality} based on the duality established for $\CU$, which makes the solutions unique; (iii) lift the \emph{approximation results} on $\U$-valued processes to $\G^{[]}$, i.e., lift $\CU^{(n)} \to \CU$ as $n \to \infty$ to $\CG^{(n)} \to \CG$ as $n\to\infty$, i.e., convergence of the grapheme Fleming-Viot process. In particular, establish the statements about the process $(\nu_t)_{t \geq 0}$ ($\nu_t$ is the law of $h_t$ under the sampling measure $\mu_t^{\otimes 2}$) based on the properties of the $\CU$-processes.

All these points are resolved in the following proposition.

\begin{proposition}[From $\CU$ to $\CG$]
\label{prop.2351}
The following hold for the various model classes considered, except for insertion/deletion (treated in Section~\ref{s.compo}).\\
(a) If $\CU^n$ converges to $\CU$ in path space, then $\CG^n$ converges to $\CG$ for the $\G^{[]}$-valued process, as well as for the $\G^{\langle\rangle}$-valued functional.\\
(b) If $\CU$ satisfies the $(\CL^\CU,\wh \Pi^{\ast,\downarrow},\Gamma_\CU)$-martingale problem and the latter has a unique solution, then the $(\CL^\CG,\wh \Pi^{\ast},\Gamma_\CG)$-martingale problem associated with $\CU$ has a unique solution $\CG$ that is given by the construction in Section~\ref{sss.link}. Furthermore, the functional $(\nu_t)_{t \geq 0}$ solves the martingale problem of~\eqref{e1435} (see Remark~\ref{r.1429}, where the growth operator is suppressed and the domain of the operator is $\Pi^\ast$).\\
(c) The equilibrium of $\CG$ is obtained by applying the construction of Section~\ref{sss.link} to the equilibrium state of $\CU$.\\
(d) The strong dual representation of $\CU$ implies the strong dual representation of $\CG$, in particular, the $H$-duality relation lifted from $\U_1$ to $\G^{[]}$.\\
(e) The second-order property of the operators of $\CG$ can be transferred from $\CU$. \qed
\end{proposition}

The details on how to prove these results are given next.

%%%%

\subsubsection{Proofs: functionals of connection-matrix distributions and joint laws}
\label{sss.func}

%%%

\paragraph*{The structure of grapheme spaces.}

We are in a similar situation as in the theory of measure-valued processes, where, for example, the Fleming-Viot process in any positive time enters the set of atomic measures (which is a trap), after which it has a simpler structure. Here we consider the version of the process on $\wt \G_{\rm ultra}^{[]}$ (recall the definition~\eqref{e849}). The task to move from $\CU$ to $\CG$ is facilitated by the structural property of such a grapheme $\CG$: $h$ is in fact a \emph{coarser} ultrametric than $r$, the one of the space $(U,r,\mu)$ in which we embed the graph. Namely, we can view $h$ as a function defining an ultrametric taking three values, $0$, $1$ and a further value $>1$, and define a \emph{closed subspace} $\U^\ast_1$ of $\U_1$ (and similarly for $\U$). In \emph{our} models, the ultrametric structure on the state given by $\CU_t \in \U_1$ is a \emph{refinement} of the state from $\U_1^\ast$ given by $h$, namely, $(U_t,h_t,\mu_t^\downarrow)$, the arrow indicating restrictions of $\mu_t$ to the coarser $\sigma$-algebra of the ultrametric $h$.

The more general properties are the following :
\begin{itemize}
\item
For $t>0$ our process has only completely connected components, i.e., lives in $\G_{\rm comp}^{[]},\G_{\rm comp}^{\langle\rangle}$.
\item
The solution of the martingale problem for the $h$-functional can itself be seen as a $\U_1^\ast$-valued process for some closed subspace $\U^\ast_1$ of $\U_1$, defined by the property that the metric takes only three values, $0$, $1$ and $>1$, with operators similar to the ones obtained in population genetics for the genealogy process. Being a coarse-graining, the solution is compatible with the underlying $r$.
\end{itemize}
The above observation defines a \emph{grapheme space $\G^\ast \subseteq \G^{[]}$}: (i) where our processes turn out to live and where the underlying embedding in a metric space $\CU=[U,r,\mu]$ is required to be from $\U_1$ (or $\U$); (ii) $h$ is required to be coarser than $r$. In that case we can give an embedding explicitly, which allows us to view \emph{$\G_{\rm comp}^{\langle\,\rangle}$ or $\bar \G_{\rm comp}^{\langle\rangle}$ as contained in a closed subspace of the space $\U_1^\ast$} of ultrametric probability measure spaces, such that the ultrametric takes on only three different values via an embedding. Here are the details of the construction.

When all connected components are completely connected, we can project, respectively, embed the subset of graphemes $\G_{\rm comp}^{[]} \subseteq \G^{[]} ,\G_{\rm comp}^{\langle\rangle} \subseteq \G^{\langle\rangle}$ into $\U_1$ via a map $\psi$ as
\begin{equation}
\label{S1}
\G_{\rm comp}^{[]} \overset{\psi}{\hookrightarrow} \U_1, \quad \text{ the image $\U_1^{\rm graph}$
of $\G_{\rm comp}^{[]}$ is contained $\U^\ast_1$,}
\end{equation}
where $\U_1$ is the space of equivalence classes of ultrametric measure spaces. Here, on $\G^\ast$, $\psi$ is given by
\begin{equation}
\label{S2}
[(\CI,r),h,\mu] \overset{\psi}{\longrightarrow} [\CI,r^\ast,\mu^\ast] \in \U^\ast_1,
\end{equation}
where $\mu^\ast$ is the restriction of $\mu$ to the $\sigma$-algebra associated with $r^\ast$, and $r^\ast$ is an ultrametric given by
\begin{equation}
\label{S3}
r^\ast(x,y)= \left\{
\begin{array}{ll}
0, & \mbox{  if  } x=y, \\
1, & \mbox{  if  } h(x,y)=1, \\
\infty, & \mbox{  if  } h(x,y)=0 \mbox{  and  } x \neq y.
\end{array} \right.
\end{equation}
Via the transformation $r \to 1-\e^{-r}$ we obtain an element of the classical $\U_1$ (with \emph{finite} values for the metric), with distances taking values $\{0,1,\infty\}$, respectively, $\{0,e^{-1},1\}$. On $\G^{[]}_{\rm ultra}$, this mapping is a coarsening of the state, while on $\G^{\langle\rangle}_{\rm ultra}$ it is an embedding and
\begin{equation}
\label{S4}
\U^\ast_1 \text{ is a closed subspace of $\U_1$}.
\end{equation}
In the general case $\G{[]},\G^{\langle\rangle}$ the function $h$ can be viewed as metric (not necessarily an ultrametric) when the third value is less than $2$.

\begin{proof}[Proof of Proposition~\ref{prop.2351}]\label{pr.2287}
We proceed in five steps (i)-(v).

\medskip\noindent
(i) We first consider initial states in $\G^{[]}_{\rm comp}$, or the single-root case, generating a closed set topologically and dynamically.

\emph{Case $\G^{[]}$.} The above observation (together with the fact that balls in ultrametric spaces are open and closed) yields that the weak convergence of $\CU_t^{(n)}$ to $\CU_t$ in $\U_1$ as $n \to \infty$ implies the convergence of $\CG_t^{(n)}$ to $\CG_t$ when the component $h^{(n)}$, respectively, $h$ is taken into account. This is due to the definition of $\CG^{(n)},\CG$ as specific functionals of $\CU^{(n)},\CU$. In particular, if we look at the finite-$n$ versions of the functional $\nu$ giving the process $(\nu_t)_{t \geq 0}$, then we see that this process converges based on what is known about the underlying processes $\CU^{(n)}$ and their limit $\CU$, namely, their $2t$-ball decomposition is also the decomposition in completely connected components, and the former converges as object in $\U_1,\U$.

\emph{Case $\G^{\langle\rangle}$.} Consider evolution rules where the population dynamics has genealogies that evolve as \emph{$\U_1$-valued processes}, arising as limits of individual-based models, which are \emph{strong Markov processes $(\CU_t)_{t \geq 0}$} and are characterised by a well-posed martingale problem. We have defined a function $h$, based on $\CU$, so that we obtain a functional~\eqref{S2} giving the connection-matrix distribution $\Psi$ that maps elements of $\U_1$ to $\G^{\langle\,\rangle}$, and gives rise to a \emph{well-defined grapheme evolution $(\CG^{\langle\,\rangle}_t)_{t \geq 0}$} on the state space $\G^{\langle\,\rangle}$, of which we need that it arises as the limit of finite-grapheme evolutions. In particular, the subgraph counts converge to the ones obtained as moments under the connection-matrix distribution. The latter comes from the fact that the frequency of $k$-subgraphs can be expressed in terms of $(\mu_n(t)(H^i))_{i \in \N}= (\mu_n^\ast(t)(H^i))_{i \in \N}$ with $H^i$, $i \in \N$, the completely connected components, and from the fact that the expressions converge when the processes $(\CU_t^{(n)})_{t \geq 0}$ converge to $(\CU_t)_{t \geq 0}$, since the completely connected components are open and closed balls in the metric given on $\CU_t$.

We have to identify the possible \emph{initial states} for the $n\to \infty$ limit, i.e., we need to know how c\`adl\`ag or continuous path moves inside $\G^{[]}_{\rm comp}$, or move from the complement to the inside as the entrance law from the single-root case as continuous path, or instantaneously jump inside at time $0$ and afterwards evolve like the latter entrance law. Finally, we can have mixtures of these three extreme types of behaviour.

The above means that we have to show that if $H$ is not $\equiv 0$ or $\equiv 1$ within sets of a decomposition of the basic set, or is a mixture of both, then there is the instantaneous jump at time $t=0$ into the $H \equiv 0$ state, taking place on the complement of the union of the completely connected components. In other words, looking first at the case $c>0$, only completely connected components can avoid immediate destruction, by the growing of completely connected components, founded by single vertices whose descendants survive until time $t >0$ that are all connected due to the resampling mechanism, but to \emph{no} vertices that are \emph{not} in their clique. Thus, descendants of immigrants are connected precisely with all their subfamilies, which reduces the need to show this for all initial survivors for $t >0$, based on the $c=0$ mechanism. A similar but more subtle mechanism works for the case $c=0$. Here we have to focus on the edges between different \emph{weakly} completely connected components, where connections to other components or points are allowed. For the part of $h \equiv 0$, we have the same situation as with immigrants. We can focus on the completely connected components. We need that the change by resampling leads to a loss of connection to the outside of the weakly completely connected component from the lost vertices, and duplication of the connections to the outside that go out from the duplicated vertex. Therefore the unbiased subsequent changes are of order $n$, whereas the number of changes of vertices in a weakly completely connected component is of order $1$. Therefore, as $n \to \infty$ we hit $0$ with the number of edges between weakly completely connected components before the vertices begin to fluctuate. This jump follows also from the convergence of $\CU^{(n)}$ to $\CU$, with finitely many surviving initial individuals or immigrants entering before a time $s$ for positive time in $\CU$, absorbing the full measure and thus forcing out the complement of the closure of $\wt \G^{[]}_{\rm comp}$ by right continuity.

Thus, altogether, we have proved part (a) of Proposition~\ref{prop.2351} on $\G^{[]}$ and on $\G^{\langle\rangle}$.

\medskip\noindent
(ii) Next, we show that the structure introduced above implies that $\CG$ on $\G^{[]}$ solves a well-posed martingale problem for $\CG_0 \in \G^{[]}_{\rm comp}$ or the single-root case. The action of the operator on $\wh \Pi^{\ast,\downarrow}$ arises from the construction of $\CG$ as a functional of $\CU$, and $\CU$ itself evolves according to an operator not containing information on $h$, and follows the action of the operator of the process $\CU$ acting on $\Pi^{\ast,\downarrow}$, which gives the first piece of our martingale problem for $\CG$. We have to add functions from the algebra $\Pi^\ast$ that determine $h$, as follows: $h$ is a function on $[U_t,r_t,\mu_t]$, the change of the latter gives the action of the operator on $\Pi^\ast$, with the same action of merging variables because $h$ can be viewed as giving an ultrametric, and the growth operator is identically zero because $h$ registers only when a distance is in $(0,\infty)$, so there is no change for $t>0$. We next take the union of the algebras $\wh \Pi^{\ast,\downarrow}$ and $\Pi^\ast$, and consider the generated algebra $\wh \Pi^\ast$ (recall Section~\ref{s.evol}) where we give the action on products of a function from each of the sets, which still is the merging of variables. This gives us the martingale problem that is solved by $\CG$, for $t>0$.

We have to also discuss general $h_0$ and $\CU_0$ not in the \emph{closure of $\wt \G_{\rm comp}$}, and do so based on the martingale problem. We see from the approximation that in our scaling the connections between \emph{weakly}  completely connected components disappear at a larger rate than the latter grow to full frequency of completely connected components. Since the martingale problem implies the duality via the generator criterion, and since we have right-continuous paths, we must have the $h$-component not in the closure of $\G^{[]}_{\rm comp}$ (which is dynamically closed) jump to dust, i.e., to $h \equiv 0$, at infinite rate at time $0$, as prescribed.

Next, we turn to $\G^{\langle\rangle}$. The claim is that $(\CG^{\langle\,\rangle}_t)_{t \geq 0}$ is itself a Markov  process  on $\G^{\langle\,\rangle}$, specified by a well-posed martingale problem that, however, is not defined on a closed subspace of $\wh \CW$, except when restricted to the closure of $\G^{\langle\rangle}_{\rm comp}$. For every stochastic process $(\CU^\ast_t)_{t \geq 0}$ taking values in $\U^\ast_1$, via the embedding $\psi$ of~\eqref{S1}--\eqref{S2} restricted \emph{on} $\G^{\langle\rangle}_{\rm comp}$, we obtain a stochastic process of graphemes and hence of connection-matrix distributions
\begin{equation}
\label{S5}
\left(\nu_t\right)_{t \geq 0} = \left(\psi^{-1}(\CU^\ast_t)\right)_{t \geq 0}.
\end{equation}
Consider the $(\CL^\ast, \Pi^\ast, \CG_0)$-martingale problem, where $\Pi^\ast$ is the set of \emph{connection polynomials} on $\G^{[]}$, and $(\CL^\ast,\Pi^\ast)$ arises as the \emph{extension of $(\CL,\Pi)$ from $\U_1^\ast$ to $\G^{[]}$}, with $\Pi$ the polynomials in $\U^\ast_1$. Namely, the coarsening of $\CU$ induced by the map $\Psi$ on $\CG^{\langle\rangle}$ gives a $\U^\ast_1$-valued process whose operator we can determine from the one of the underlying $\CU$. We see from~\eqref{S3} that in our case, for $t >0$,
\begin{equation}
\label{13}
\CL^{\ast,\rm FV}(\Phi) [\CU] = \CL^{\rm FV} ([\CU])-\CL^{\rm growth} ([\CU]), \qquad \CU \in \U_1,
\end{equation}
with $\CL^{\rm FV}$ the operator for the process $(\CU_t)_{t \geq 0}$ on $\U_1$. This means that for this functional not the metric but \emph{only the sampling measure} evolves further for $t>0$. To incorporate $t=0$, note that the metric jumps instantaneously to $1$ for pairs of descendants of a founder as time starts running, and to a larger number otherwise. However, this information we have \emph{only in the associated $\CU$}. Therefore we have to view the process as an entrance law from time $0$ in the dust state, where $h=0$ on $\G^{\langle\rangle}$ and not in the union of completely connected components. This process is Markovian on $\G^{\langle\rangle}$, which completes the first task. This latter follows because the operator does not rely on information about $\CU$, and uniqueness follows from duality.

This settles the second part of part (b) of Proposition~\ref{prop.2351}.

\medskip\noindent
(iii) We have a moment duality, linking the operator $\CL^{\ast, \rm FV}$ to the coalescence operator $\CL^{\ast, \rm coal}$ on $\Pi^\ast$, when we work with the polynomials on $\U_1^\ast$, since the corresponding Fleming-Viot processes $\CU$ are in duality including the growth operator. The Fleming-Viot operator and the growth operator separately satisfy the generator criterion and hence, in the backward \emph{and} the forward direction, the growth operator is dropped when passing from $r_t$ to $r_t^\ast$, which carries over to the functional that solves the $H$-duality for positive time. More precisely, for the duality relation of the underlying $\U_1(\bar \U)$-valued process $(\CU_t)_{t \geq 0}$ of $(\CG_t)_{t \geq 0}$, we know that the dual on $\U_1$ is driven by a \emph{coalescent} on the partitions of $\N$. Indeed, the right-hand side has the \emph{coalescent} with \emph{emigration} at rate $c$ into a cemetery (the entrance law from infinity, where the partition elements at time $t$ give the connected components of the grapheme $\CC_t^\ast$ for the coalescent $\CC_t$). We have to combine this with a duality covering the joint law of $h$ and the $\CU$-component, by having in the duality function~\eqref{e2174} both $\varphi_r$ and $\varphi_h$ different from just constants. But this is immediate, since in the action of the dual operator of the second-order part the merging of variables is the same in both parts.

This settles the $H$-duality.

We want to reformulate the duality in the \emph{framework of $\G^{[]}$} to a stronger statement, i.e., we want to translate the duality relation on $\U^\ast_1$ into a \emph{strong duality relation on $\G^{[]}$}. The duality relation for $\CU^{\rm FV}$ on $\U^\ast_1$ says that if we transform the distances of the state $\CU$ as in~\eqref{S3}, then the new state, denoted by $\CU^\downarrow$, satisfies
\begin{equation}
\label{e943}
\mathrm{LAW} \left[\CU_t^{\rm FV, \downarrow}\right] = \mathrm{LAW} \left[\CC_t\right],
\end{equation}
where $\CC_t$ is the ultrametric probability measure space spanned by the \emph{entrance law of the coalescent}. By the definition of $h$ as functional of $\CU$, respectively, $\mathcal{C}$ in the forward, respectively, the backward direction, this gives (recall that here $h$ is also an ultrametric)
\begin{equation}
\label{e949}
\mathrm{LAW} \left[\CG_t^{d,c,\theta}\right] = \mathrm{LAW}\left[\CG_0 \vdash \CC_t^{\ast,d,c,\theta}\right], \qquad t > 0.
\end{equation}
This dual representation is associated with a duality relation that has duality function $H$, and is connected to a \emph{generator relation}, which is known and proved on $\U_1$ in the literature, and above for the $h$-part, and which guarantees uniqueness of the martingale problem on $\G^{[]}$.

From the duality relation w.r.t.\ $H$ we obtain the strong duality by using the \emph{tightness} properties for the collection of the finite dual processes (see~\cite{GrevenPfaffelhuberWinter2013}) and the fact that $H(\cdot,\cdot)$ determines states and the law of states. Transforming both sides of the duality relation, we get \eqref{e949}, and the above embedding in $\U^\ast_1$ carries over to $\G^{\langle\rangle}_{\rm comp}$, and trivially to $\G^{\langle\rangle}$, since we add the immediate jump to the former by modifying~\eqref{e949}, replacing $\CG_0$ by $\CG_0$ projected on the closure of $\G^{[]}_{\rm comp}$ in the right-hand side.

This settles part (d) of Proposition~\ref{prop.2351}.

\medskip\noindent
(iv) From~\eqref{e949} we conclude that the equilibrium process of $\CG^{\langle\rangle}$ arising as functional of $\CG^{[]}$ is described by the equilibrium of a known \emph{measure-valued process}, namely, the Fleming-Viot process with emigration/immigration. Indeed, project the process $\CG^{[]}$ on its measure component $(\mu_t)_{t \geq 0}$, and afterwards to the subalgebra given by the $r^\ast$-balls for $r^\ast=1$, to get the weight vector of the $2t$-balls. Take the same operations in the right-hand side of~\eqref{e949}. This leads to the same right-hand side in~\eqref{e949} as in the duality relation for the well studied \emph{measure-valued} Fleming-Viot process with emigration/immigration, allowing us to use~\cite{DGV95}, where the identification of the equilibrium as well as convergence to equilibrium is given.

This settles part (c) of Proposition~\ref{prop.2351}.

\medskip\noindent
(v) The fact that the operator for $\CG$ is second order follows from the fact that the first-order terms can be verified as such with~\cite{DGP12}, whereas the second-order term acting on the $h$-component and the $r$-component is of the same form and hence~\cite{DGP12} \emph{again} can be used to get Proposition~\ref{prop.2351}(e).

This completes the proof of Proposition~\ref{prop.2351}.

Let us look in more detail at the martingale problem for the underlying process $\CU$ and the facts that are needed as assumptions for Proposition~\ref{prop.2351}. The \emph{existence} of a solution to the martingale problem on $\U_1$ arises via the finite-population approximation, by considering a suitable dynamics with $n$ individuals and letting $n \to \infty$. Since the dynamics we have given for the vertices and the edges can be viewed as the individual-based Moran process or the Galton-Watson process, both enriched with edges, the convergence properties for graphs fit with the convergence properties for populations. We can therefore use the argument that the genealogies $\CU^{d,c,\theta,(n)}, \CU^{b,c,\theta,(n)}$ of the finite $n$-populations converge to a limit, namely, $\CU^{d,c,\theta},\CU^{b,c,\theta}$, respectively. This is carried out in~\cite{GrevenPfaffelhuberWinter2013} for the Fleming-Viot process, in~\cite{DG19} for the Galton-Watson process, and in~\cite{DGP12} for the model with mutation and selection, leading to a limiting process on $\U$. For the model with immigration/emigration we have to use the results in~\cite{DG19,GSW2016,depperschmidt2023duality,ggr_GeneralBranching}, but the necessary tightness nevertheless comes out of the same ideas in the proof (see, in particular, \cite[Section 2]{GSW2016} and \cite{DG19}, where this is exhibited for a particular convergence result, but the technique carries over to our setup).

The \emph{uniqueness} for the martingale problem for $\CU$ follows from the duality, as well as the $t\to \infty$ convergence of that limiting dynamics to \emph{equilibrium}, as verified in the same references for the underlying $\CU^{d,c,\theta},\CU^{b,c,\theta}$, where the latter is proved via convergence of the dual, which gives the claim right away (the former is a general fact for standard martingale problems).

In the proof of Theorem~\ref{th.963-II}, which concerns the extension to mechanisms including \emph{selection}, $\CL$ requires marks. The addition of \emph{mutation} requires not just marks, but marks that are \emph{not} inherited anymore. This is addressed separately in Section~\ref{ss.add}.

Having \emph{only additional selection} works fine, namely, the approximation by finitely many vertices follows for the underlying $\CU$ from~\cite{GrevenPfaffelhuberWinter2013,DG20}, respectively,~\cite{DG19}, which give tightness in path space, in particular, the dust-free condition for the grapheme and convergence to a limit grapheme. Selection is handled via~\cite{DGP12}. We give the list of relevant references in Section~\ref{ss.proofth1}.

Results from the literature together with Proposition~\ref{prop.2351} settle the claims in Theorems~\ref{th.963-I},~\ref{th.963-II},~\ref{th.853},~\ref{r.cond}, and in some cases of Theorem~\ref{th.1838}, namely, without mutation.
\end{proof}

%%%

\subsection{Proof of two corollaries}
\label{ss.proofcor}

In this section we complete the proof of Corollaries~\ref{cor.1165} and~\ref{cor.1263}. We have to argue that the weight of $\mu_t$ of the completely connected components is given by the measure-valued diffusions with emigration and immigration in~\cite{DGV95}, respectively,~\cite{DG96} with values in $\CM_1([0,1])$, respectively, $\CM_{\rm fin}([0,1])$. Start at time $0$ in the Lebesgue measure. Then the states are atomic measures for every $t>0$, and for $t \to \infty$ converge to an equilibrium that can be specified explicitly, as we did in~\eqref{e.1160} and~\eqref{e1253}. How to connect this to our process $\CG$?

We have to code the weight vector given by the mass of the $2t$-balls by an atomic measure in $[0,1]$, respectively, $[0,\infty)$, by associating with the ball a unique number in $[0,1]$. Take the size-ordered mass vector and assign i.i.d.\ uniformly distributed numbers from $[0,1]$ to the corresponding balls. Then take the atomic measure of the observation with the weight under $\mu_t$ of the corresponding ball.

Next, we look at the dynamics of this object under the process $(\CU_t)_{t \geq 0}$ underlying the process $(\CG_t)_{t \geq 0}$. For that we consider the action of the generator on a subclass of test functions of the martingale problems for the process $\CG$, which are given by polynomials $\Phi^{n,\varphi}$, where the function $\varphi$ identifies the mass content of the different $2t$-balls. We can choose here the indicators of these balls, which are continuous functions due to ultrametricity. Then we obtain, via formulas~\eqref{10},~\eqref{12}, respectively, Section~\ref{sss.op} for the Dawson-Watanabe models, the same expression as in the given references for the measure-valued process with immigration and emigration treated in~\cite[Chapters 2(a) and 2(b)]{DGV95}, documenting the necessary results and their origin in the literature. Similarly, for the $\bar \G$-valued case and the Dawson-Watanabe model use~\cite{DG96}: the results from there prove our claims.

%%%

\subsection{Adding mutation}
\label{ss.add}

If in the extension of our basic processes $\CG^{c,d,\theta}$, respectively, $\CG^{c,b,\theta}$ (with values in $\G^{[],V}$ beyond the process $\CG^{c,d,\theta,s}$ with selection) we add \emph{mutation} to the selection, then in order to get $\CG$ from the underling genealogy process $\CU$ we have to change the definition of $h$ as a function of $\CU$, which must now be based on the \emph{types} rather than on the \emph{founder} at time $0$ only, and which we might or might not describe via marks. Recall that now two vertices are connected when both carry the same type. In particular, the ultrametric induced by $h$ is no longer related as easily as before to the ultrametric $r_t$. Typically, \emph{different} founders with \emph{part} of their descendants give rise to the completely connected components. Nonetheless, it is still true that completely connected components are finite or countable unions of balls of $U \times V$ that are open and closed, because the vertices descending from an individual without mutation are connected, since the time of the last mutation. Because they have the same type, this corresponds to a ball of radius twice the time back to the common ancestor and lost mutant. The part suffering mutation in that time span founds a new family, which we have to take out of the ball. Iterate this scheme in the order of the relevant mutations in the tree trunks.

We also need to give attention to the topology of the mark set, to suit our purposes. Before we do so we observe the following. The union of balls taken out leaves a countable union of balls in $U \times V$ that are completely connected. Hence the vertices descending from an individual without mutation without further mutation are a completely connected subset, because they have the same type. This correspond to a ball in $U$ of radius twice the time back to the common ancestor, but all marked with a specific type. Hence we are still able to conclude the approximation via finite models by using the tightness properties that we have for $\U_1,\U$ from the literature. To control the distribution of masses, in particular, their uniform summability, we control the number of ancestors of at least $1-\varepsilon$ of the mass over bounded time intervals, using backward the dual and forward the diffusion of the frequency of the part of the population descending from a finite set of ancestors. Here we use that the limit of the process projected on $V$ is a well-known measure-valued diffusion, which for $t >0$ has atomic states and is the limit of its finite counterparts. The $\U_1$-valued process projecting on the genealogy has a subfamily decomposition in $2(t-s)$-balls (disjoint and open and closed) for $s \in [0,t)$, of which at most countably many are charged by the sampling measure (compare with~\cite{DGsel14}). Also this decomposition arises as the limit of the decomposition of the finite counterparts.

We next give more detail on the decomposition above and the necessary topology of $K$. The claim is that $U \times V$ is decomposed for finite $n$, and in the limit consists of at most countably many balls in $U$, each marked with a single type. These are completely connected components, since the descendants of a mutant surviving at the reference time are elements of a ball with radius $2(t-t_{\rm mut})$, with $t_{\rm mut}$ the time of mutation. The mutants occurring in this ball and surviving have a different type, and give rise to a sub-ball taken out of all the descendants at time $t$. If we \emph{strengthen the topology} on $V$ to the \emph{discrete} topology, then we obtain a decomposition into open and closed balls, thereby creating the completely connected components.

The remaining parts of the argument are the same as in the models we treated before.

Regarding the equilibrium states, it is known from~\cite{DGP12} that the underlying process $(\CU_t^{c,d,\theta,m,s})_{t \geq 0}$ has a unique equilibrium. For the process arising when we project on the mark component, this observation becomes a statement (in the case we are treating) about a measure-valued process on the type space $V$. Therefore, in the case $m=0$, it is clear that we can transfer the argument to the grapheme process, which is a function of $\CU$. With mutation the argument becomes a bit more subtle as we saw above, since now the function giving $\CG$ from $\CU$ is at first sight of a somewhat different character, and we have to show that the process run for time $t \to \infty$ converges to a state obtained by applying the new function applied to $\CU_\infty$ as well. But, this follows from the fact that the level sets of $h$ as finite or countable union of balls are continuous.

%%%

\subsection{Completion of the proof of the main theorems}
\label{ss.proofth1}

In this section we complete the proof of Theorems~\ref{th.963-I}, \ref{th.963-II} and \ref{th.963-III}, the duality Theorem~\ref{th.853}, and the extension Theorem~\ref{th.1838}. We do so by providing more details on the \emph{references} given in Sections~\ref{sss.process},~\ref{sss.func} in order to justify the \emph{claims }made for the \emph{genealogy process $\CU$}. There we saw how to make the step from $(\CU_t)_{t \geq 0}$ to $(\CG_t)_{t \geq 0}$. Relevant references and theorems to the properties of genealogy-valued processes $\CU$ necessary for a particular claim on $\CG$ are listed below for each of our theorems.

A general introduction to the ideas and orientation in the field can be found in~\cite{DG20}, which gives some orientation to be able to work with the references we suggest below. A further general remark is that the proof in~\cite[Section 2]{GSW2016} contains many steps and arguments that are \emph{easily adapted to different situations}.

Here is the list of references:
\begin{itemize}
\item Theorem~\ref{th.963-I}:
\cite{GSW2016} Section 1.3.1, in particular, Theorem 1.12, Theorem 1.17 and the arguments in Section 2.
\item Theorem~\ref{th.963-II}:
\cite{GrevenPfaffelhuberWinter2013} Proposition 2.22, \cite{GSW2016}, Section 2.
\item Theorem~\ref{th.963-III}:
\cite{GSW2016} Theorem 1.17, Section 2.
\item Theorem~\ref{th.853}:
\cite{GSW2016} Theorem 1.17, Section 2.
\item Theorem~\ref{th.1838}:
\cite{DGP12} Theorem 1, Theorem 3.
\end{itemize}

%%%%%%% SECTION 8 %%%%%%%%%%%%%%%%%%%

\section{Proofs for grapheme diffusions with non-completely connected components}
\label{s.compo}

In this section, we add \emph{insertion and deletion} of edges in order to obtain grapheme diffusions with equilibria whose (\emph{path} connected) components are not necessarily \emph{completely} connected. The key steps are as follows.

We show that for the process denoted by $\CG^{(a^+,a^-)}$ (being a  solution of the martingale problem) there exists a unique underlying process $\CU$ that generates a unique grapheme process \emph{without} insertion and deletion (as we know already), say $\CG^{(0,0)}$. In the model considered in this paper, the edges no longer evolve as a function of the $(U,r,\mu)$, because additional randomness enters.

We analyse the situation by \emph{conditioning} $\CG^{(a^+,a^-)}$ on the entire path of the underlying process $([U_t,r_t,\mu_t])_{t \geq 0}$, which evolves \emph{autonomously} as a strong Markov process, and study the evolution $(h_t)_{t \geq 0}$, or better the connection-matrix distribution $(\nu_t)_{t \geq 0}$, which itself does \emph{not} influence the evolution of $(\CU_t)_{t \geq 0}$. Note that conditioning on $\CU$ is \emph{equivalent} to conditioning on $\CG^{(0,0)}$, since the latter is a \emph{function} of $\CU$ and also contains $\CU$.

Thus, we have a martingale problem generating the autonomous evolution of the path $(\CU_t)_{t \geq 0}$, referred to as the solution of the martingale problem of the condition, i.e., this process exists uniquely with continuous paths and is strong Markov. (The same holds for $\CG^{(0,0)}$, as we already know.) Given a complete path of $\CU$, we obtain a conditional martingale problem from~\eqref{e.2147}--\eqref{e.2151} for $(h_t)_{t \geq 0}$ or $(\nu_t)_{t \geq 0}$, by inserting into the expression the data from the given path $\CU$. This gives us a specific time-inhomogeneous spin-flip system operator in random environment, with the latter given by the condition.

The strategy to treat the martingale problem in Section~\ref{ss.proofth2},~\ref{ss.proofcoralt} is as follows:\\
(i) \emph{Given a path} of $(\CU_t)_{t \geq 0}$, we let $\nu_s$ be the time-$s$ marginal of the law of the functional $\nu$ of the $\N \times \N$ connection-matrix distribution given by the evolution $(\nu_s((\CU_t)_{t \geq 0}))_{s \geq 0}$, which we characterise by the \emph{conditional} martingale problem derived above (for which we must show that it is well-posed and depends \emph{measurably on $\CU$}).\\
(ii) We show that the full martingale problem \emph{implies} that the condition $\CU$ evolves according to its own well-posed martingale problem.\\
(iii) We show that $(\nu^\infty_s)_{s \geq 0}$ must solve uniquely the \emph{conditional martingale problem} for almost surely all realisations of the path $\CU$. Integration over the condition (recall (i) for measurability) gives the unique solution $\CG^{(a^+,a^-)}$ of the martingale problem on $\G^{[]}$.\\
(iv) With the help of the conditional duality, we can treat the equilibrium as well: we show that the equilibrium is the $t \to \infty$ limit of $\CG^{(a^+,a^-)}_t$ and arises by integrating the equilibrium of $\CG_\infty^{(0,0)}=[(U_\infty,r_\infty,\mu_\infty,h_\infty)]$ for the state $[(U_\infty,r_\infty,\mu_\infty,h_\infty)]$.

%%%

\subsection{Completion of the proof: martingale problem}
\label{ss.proofth2}

In this section we complete the proof of Theorem~\ref{th.graphnon} concerning the martingale problem part and properties of the solution, i.e., parts (a), (b) and (d). The fact that we deal with a conditional martingale problem requires some care, even though the basic idea is simple. We proceed in three steps to realise the program described above.

\medskip\noindent
\emph{Step 1}.
We have to specify the evolution of the \emph{connection-matrix distribution}, which we do by giving the evolution of the finite-dimensional $(\nu^m_t)_{t \geq 0}$, $m \in \N$, and hence of $(\nu_t)_{t \geq 0}$. This arises when we restrict the action of the operator for a given path of $\CG^{(0,0)}$ to functions in $\Pi^\ast$.  Note that the \emph{connection-matrix distribution} of the \emph{virtual edges} $(\nu_t^{(0,0)})$ is given. We obtain a process that has continuous path, since $\CG^{(0,0)},\CU$ have this property. Consider the operator of the martingale problem as introduced in Section~\ref{ss.ext}. Via the corresponding conditional martingale problem (recall~\eqref{e.2147}) this operator determines \emph{uniquely}, for every path of $(\CU)_{t \geq 0}$, a process $(\nu_t)_{t \geq 0}$ with values in the connection-matrix distributions, since it is the generator of a time-inhomogeneous spin-flip system with independently acting components on $\{0,1\}^\N$. The latter is, of course, well-posed as the evolution equation of a simple Markov pure-jump process.

\medskip\noindent
\emph{Step 2}.
From the martingale problem of the process $\CG^{(a^+,a^-)}$, we identify the martingale problem of the underlying process $\CU$. This in turn determines a unique process $\CG^{(0,0)}$, where the insertion/deletion of edges is suppressed. For that purpose, we consider the action of the operator for $\CG^{(a^+,a^-)}$ on the test function. This operator depends only on $[U_t,r_t,\mu_t]$ and is, by inspection, immediately identified as the operator acting on functions of the underlying process $\CU$ in $\wh \Pi^{\ast,\downarrow}$. Note that $\CU$ agrees completely with the one we obtain as underlying process for $\CG^{(0,0)}$. Furthermore, it determines uniquely a solution in $\U_1$, and therefore $([U_t,r_t,\mu_t])_{t \geq 0}$ must indeed be a path of the process $\CU$ underlying $\CG^{(0,0)}$ \emph{and} $\CG^{(a^+,a^-)}$. The point here is that $\wh \Pi^\ast$ is the algebra generated by $\wh \Pi^{\ast,\downarrow}$ and $\wh \Pi^\ast$, each of which belongs to one of the components of the grapheme, $[U_t,r_t,\mu_t]$, respectively, $[h]$ and, due to the form of the operator for the two components, the operator is lifted immediately from the one of the $\G^{[]}$-valued martingale problem.

\medskip\noindent
\emph{Step 3}.
Here we come to the modified dual representation. We have a dual representation of the underlying process $\CG^{(0,v)}$ or $\CU$ and of the conditional martingale problem. This means that the process $\CG^{(0,0)}$ or $\CU$, and the conditional process giving $h^\ast$ based on $h^{(0,0)}$, are uniquely determined and are continuous in the initial condition. Therefore there is a unique process, arising from the recipe above, that solves the original martingale problem. To see that this is the \emph{unique} solution of the latter, it suffices to have a duality for the original martingale problem. But this follows from the form of the operator, and from the duality relations for $\CG^{(0,0)}$ and the additional generator term associated with the spin-flip system, dual to the \emph{reversed} spin-flip system.

%%%

\subsection{Completion of the proof: equilibrium}
\label{ss.proofcoralt}

What remains to be shown is that the equilibrium exists and is as described in part (c) of the theorem. However, this follows from the dual representation, which guarantees existence of the limit as $t \to \infty$. The dual of $\CG^{(0,0)}_t$ converges as $t \to \infty$ to an equilibrium that we identified via the dual, and for which we explicitly obtained the forward equilibrium state of the functional $\nu$. Subsequently, we observe that the pruning of the edges yields the frequencies of edges as the current law of the spin-flip system on $\{0,1\}^\N$, which approaches the limit state as $t \to \infty$ given by $(\frac{a^+}{a^+ +a^-}, \frac{a^-}{a^+ + a^-})^{\otimes \N}$.

The equilibrium of the conditional process is independent of the condition. Therefore the convergence of $\CG^{(0,0)}$ to an equilibrium, which is independent of the edge process, indeed means convergence to the claimed equilibrium, which is the product of the law of an $\N$-sample $h^{(0,0)}$ of $\CG^{(0,0)}_\infty$ and of the spin-flip system $h^\ast_\infty$ on $\{0,1\}^\N$. The equilibrium $v_\infty$ is obtained as the law of (denote with $\times$ the point-wise multiplication of two function on the same space $\N$)
\begin{equation}
\label{e2468}
h^{(0,0)} \times h^\ast.
\end{equation}

%%%%%%%%% SECTION 9 %%%%%%%%%%%%%%%%%%%%%%%%

\section{Discussion and further thoughts}
\label{s.discussion}

In Sections \ref{s.introduction}--\ref{s.compo} we introduced the concept of \textit{graphemes} as a novel framework for studying dynamic graphs. By embedding graphs in Polish spaces, particularly ultrametric spaces, and drawing inspiration from population dynamics models, we developed a rigorous approach to model the evolution of continuum graphs. Our work extends and modifies previous research on graphon dynamics \cite{AHR21}. While graphon dynamics offers a valuable approach to studying evolving graph limits, our grapheme framework provides a more comprehensive description by explicitly incorporating the genealogy of the graph evolution. This allowed us to address open questions in graphon dynamics, such as the existence of non-trivial equilibria and the strong Markov property, and to characterize the space-time path process of dynamic graphs.

In this section we look back at what we have achieved and offer some further thoughts and guidance on the technical context. 

%%%

\subsection{Remarks on the state spaces}
\label{sss.dyn}

We wanted to construct stochastic processes on $\G^{[]}$, $\G^{\{\}}$,$\G^{\langle\rangle}$ that arise as limits of processes of embedded finite graphs when the number of vertices tends to infinity. We focussed on $\G^{[]}$, which is the only candidate for a Polish space so that we can use martingale problems to characterise processes. The construction was carried out in detail in Section~\ref{ss.examples} and worked out further in Sections~\ref{s.evol}--\ref{s.exmart}. In view of the way we defined graphemes, the proper notion of convergence was built on \emph{convergence of sampled embedded subgraphs}. Finite-graph dynamics are \emph{Markov pure-jump processes} that are strong Markov and Feller (i.e., continuous in the initial state). In order to get \emph{limiting} processes that are also strong Markov and Feller, with \emph{regular paths} (i.e., continuous or c\`adl\`ag) on Polish state spaces (but typically not $\sigma$-compact or locally compact state spaces), many standard techniques from stochastic analysis are not readily applicable. In fact, we needed to work with a dense class of \emph{equicontinuous test functions} on our state space, and this class needed to be countable, separating and measure determining (see Section~\ref{s.evol}).

A $\G^{\{\}}$-valued process induced by a $\G^{[]}$-valued Markov process need \emph{not} be Markov. However, we built $\G^{[]}$-valued processes with good properties such that, for all initial states in $\G^{[]}$, the induced process in $\G^{\{\}}$ is the unique Markov process solving a martingale problem on a subspace of $\G^{\{\}}$. A similar fact holds for the $\G^{\langle\rangle}$-valued induced process. Recall that the latter had to be realisable via countable graphs embedded in some Polish space, which is why we preferred to view them as a \emph{functional} of the $\G^{[]}$-valued process that \emph{automatically generates existence} of a concrete embedding. Otherwise we needed to verify that the limiting subgraph counts correspond to a connection-matrix distribution arising in \emph{some} arbitrary Polish space embedding. Note, however, that for martingale problems we needed Polish state spaces, which was guaranteed on $\G^{[]}$ for the environment component without further restrictions, but on $\G^{\{\}}$ only by restricting to tree-like spaces, and on $\G^{\langle\rangle}$ by restricting to states that decompose into completely connected components (also dust would do, but that would be too restrictive). We left this issue open (see Section~\ref{s.evol} for further details).

We put forward the view of the grapheme being represented by \emph{equivalence classes} of countable \emph{partitions of a continuum} via an embedding in a Polish space, and stochastic processes in these equivalence classes. To pass to \emph{equivalence classes of paths} is a topic of ongoing research. We demonstrated that, for the class of finite-graph dynamics introduced in Section~\ref{ss.examples}, for $t >0$ we have completely connected components, which provides a concise description of the grapheme dynamics on $\G^{[]}$, and therefore is preferable to knowledge of the dynamics of the subgraph densities only. We showed that a modification of this idea works even for \emph{not} necessarily completely connected components, provided the definition of a component is properly modified. Moreover, the additional information about the embedding space can be used to obtain information about the \emph{history} behind the subgraph densities at a given time as they evolve in time. This means that we get a hold of the space-time process, even the space-time path process, of the subgraph densities up to any given time (see Section~\ref{ss.choice} for further details).

In the weak topology of measures, the state space for the dynamics of the connection-matrix distribution (which reflects the subgraph densities) is compact and hence its law is tight. However, this state space includes as possible limit points the connection matrices $\underline{\underline{0}}$ and $\underline{\underline{1}}$, which are \emph{improper} for many of our purposes because they are traps for our dynamics. Furthermore, limit points may lack the property to be a connection-matrix distribution of a sample sequence in a Polish space. We had to \emph{exclude} both, and work on a smaller admissible state space that, unfortunately, no longer is closed. This meant that we had to check \emph{compact containment properties} for the paths of our stochastic processes.

Our approach opens up the possibility to choose the embedding in a way that is \emph{adapted} to the dynamics. For dense finite graphs the dynamics typically involves changes that are \emph{not} local. Therefore choosing the embedding is a subtle issue, even though its main role is to allow us to \emph{sample} a countable graph in the $n \to \infty$ limit. We must therefore choose a stochastic process
\begin{equation}
\label{e508}
(\CI^*_t,h_t,\mu_t)_{t \geq 0}
\end{equation}
i.e., the spaces we embed in are \emph{random} and \emph{time-dependent}. This affects the definition of the topology on the state spaces $\G^{[]}$, $\G^{\{\}}$, since convergence also requires convergence of the spaces in which we embed. This allowed us to store enough information to compensate for the fact that the subgraph densities at a fixed time only reveal part of the grapheme: first, they only provide information that is \emph{averaged} over the sample sequence, and second, they give no information on the embedding of the vertices. Note that, due to the form of the equivalence classes, without loss of generality we may assume that $\supp(\mu_t) = \CI_t$.

\begin{remark}[Continuum limits]
{\rm When dealing with countable limit objects we must assume that $\CI$ is a \emph{continuum}. Our grapheme processes may start in any subset of $\G^{[]}$, including sets of non-dense graphs, but after any positive time they fall in the set of dense graphs. In the present paper we focussed on \emph{dense} graphs. For \emph{non-dense} graphs, not treated in the present paper, an adaptation of the concept of a grapheme would be needed.}  \hfill$\spadesuit$
\end{remark}

\begin{remark}[Choice of spaces]
{\rm Having both $\G^{[]}$ and $\G^{\{\}}$ to work with resembles the situation in population models, where spaces of tree-like metric measure spaces and algebraic tree-measure spaces appear, which build on weaker equivalences than measure spaces \cite{GrevenPfaffelhuberWinter2009, LW21}. The space $\G^{\langle\rangle}$, which is the most general object, still corresponds to a subset of \emph{countable graphs}, just like objects in the \emph{completion} of finite \emph{graphs} with respect to the topology of subgraph counts give rise to the space of graphons (see~\cite{Jan13}).} \hfill$\spadesuit$
\end{remark}

\begin{remark}[Exchangeabile arrays vs.\ graphemes]
\label{rem:exchangeable-arrays-graphemes}
{\rm Recall \eqref{e510}. Even though $\nu^{(m),n}$ are not exactly (jointly) exchangeable\footnote{= invariant under permutations of the indices}, the limiting distributions $\nu_m$ and $\nu$ are. The Aldous-Hoover theorem~\cite{Aldous1981} guarantees that there exists a representing measurable function $h \colon [0,1]^2 \to [0,1]$ of the form
\begin{align}
\nu = \mathrm{LAW} \left[ \{ h(U_i, U_j)\}_{i, j \in \N^2 \setminus D} \right],
\end{align}
where $U_i$ are i.i.~uniformly distributed on $[0,1]$. (See also \eqref{eq:homdensity} for a link with homomorphism densities.) The grapheme framework assumes more structure than just exchangeability of the connection matrix distribution. The grapheme framework puts the embedding space $(\CI^*,h,\mu)$ into focus. In particular, we used more restrictive invariance notions than just exchangeability (cf.~Definition~\ref{def.grapheme}). This allowed us to define stochastic grapheme-valued processes (see Section~\ref{s.exmart}).}\hfill$\spadesuit$
\end{remark}

\begin{itemize}
\item
As  shown in Section~\ref{s.state}, $\G^{[]}$ is a Polish space\footnote{In order to avoid Russell’s antinomy (self-referential definition), we assume that the elements of representatives of equivalence classes from $\G^{[]}$ are not metric spaces themselves.} and is the state space on which we construct our grapheme dynamics, because it allows us to obtain \emph{strong Markov processes} $(X_t)_{t \geq 0}$ with \emph{regular path} properties solving standard martingale problems. The key point is that the information on the \emph{space-time process} $(t,X_t)_{t \geq 0}$
with values in $\R_+ \times \G^{\langle\rangle}$, which arises as genealogies (as we explain in Section~\ref{sss.link}), provides the most natural and minimal choice of embedding in a Polish space necessary to control the space-time path processes of edges via a Markov process. The state space $\G^{[]}$ has the advantage that it applies the machinery of stochastic analysis most efficiently and is fully understood, whereas the two other state spaces have some structural deficiencies and raise a number of \emph{open questions} that remain to be resolved.

\item
On $\G^{\{\}}$ the nice process properties mentioned above may in principle fail (although we show that they do not for the classes of evolution rules we treat). This state space is conceptually important, because it satisfies the \emph{minimal structural requirements} that we need when we want to embed our countable graph in a Polish space $\CI$, and want to have a stochastic process for which the paths can be embedded in $C([0,\infty),\CI)$ or $D([0,\infty),\CI)$.

\item
The role of $\G^{\langle\rangle}$ is to link grapheme dynamics to graphon dynamics on $\wt \CW$, because an element of $\wt \CW$ is uniquely characterised by the subgraph densities of any of its representative (see Appendix~\ref{ss.graphons}), and is an element of a Polish space. We saw that we can associate with our grapheme process a graphon process by taking a $\G^{\langle\rangle}$-valued functional as long as the graphon can be interpreted as a countable graph, i.e., our graphemes always have \emph{$\{0,1\}$-valued connection-matrix distributions}. This can result in block graphons, but also in graphons with not necessarily completely connected components. A stochastic process with values in $\G^{\langle\rangle}$, which is part of a Polish space (the graphons) but is itself \emph{not} a Polish space (only in a stronger topology), must therefore have paths of embeddings where the associated graphon is $\{0,1\}$-valued without exceptional points.
\end{itemize}

In Section~\ref{sss.intermezzo} we explain what is behind the choice of the spaces $\CI^*$ or $(\CI,r)$, and what are the guiding principles for that choice when we are interested in \emph{stochastic processes} taking values in $\G^{[]}$, $\G^{\{\}}$, $\G^{\langle\rangle}$: $\G^{\{\}}$ is the state space that is needed to obtain an \emph{embedded countable graph} that we can follow as time evolves, $\G^{\langle\rangle}$ keeps the structure of the embeddable countable graph without fixing an equivalence class of embeddings, while $\G^{[]}$ provides the most manageable setting that allows for a detailed study of the corresponding process of countable graphs by means of stochastic analysis. We address the question how to choose $[\CI^*,h,\mu]$ when we start from a given (finite) graph dynamics on $\G_n$ or $\G_\infty$. The initial state plays a special role. For $t=0$, this is an arbitrary state, whereas for $t>0$ the dynamics typically moves into subsets of $\G^{[]},\G^{\{\}}$ of a special form. In Section~\ref{ss.choice}, we emphasise the intrinsic role of genealogies, and list some challenges and perspectives.

%%%

\subsection{The issue of embedding}
\label{sss.intermezzo}

%%%

\paragraph*{Representability as a grapheme and dynamics in random environment.}

For finite graphs, the natural embedding set is $\CI_n = [n] = \{1, 2, \dots, n\}$, with a probability measure $\mu_n$ to sample points (vertices).  Typically, $\mu_n$ would be the uniform measure, assigning probability $1/n$ to each vertex.  Given a sequence of such finite graphemes and a corresponding sequence of connection-matrix distributions, we want to consider the limit as $n \to \infty$, i.e.,
\begin{equation}
\label{e629}
\left[\CI_n,h_n,\mu_n\right] \longrightarrow (\nu_n^{(m)})_{m \in \N}, \qquad n \to \infty,
\end{equation}
where $\nu_n^{(m)}$ represents the distribution of the connection matrix for a sample of size $m$ from the $n$-vertex graph. We want this convergence to be such that the evaluations of our test functions (polynomials) on both sides converge, leading to a limiting sequence
\begin{equation}
\label{e633}
(\nu^{(m)})_{m \in \N} \text{ of connection-matrix distributions.}
\end{equation}
This sequence $(\nu^{(m)})_{m \in \N}$ describes the limiting probabilities of observing various finite subgraphs. However, this sequence alone doesn't guarantee that we have a well-defined ``infinite'' graph.  We need an ``embedding'' to ensure that we can actually "sample" vertices and obtain a meaningful countable graph.

This is where the embedding in a Polish space $\CI^*$ (or $(\CI, r)$ for the metric version) becomes crucial.  The embedding provides a ``substrate'' on which we can realize the limiting graph.  From the embedding and the sampling measure $\mu$, we can draw an infinite sequence of vertices $(x_i)_{i \in \mathbb{N}}$ according to $\mu^{\otimes \mathbb{N}}$ (assuming $\mu$ is non-atomic; otherwise, we sample without replacement). The connection matrix of this sampled sequence, $(h(x_i, x_j))_{i,j \in \mathbb{N}}$, will then represent a ``typical'' countable graph consistent with the limiting connection-matrix distributions $(\nu^{(m)})_{m \in \N}$.

However, we often have choices for how to embed the finite graphs.  It is not always best to simply use $\CI_n = [n]$. A more insightful approach is to choose an embedding that reflects the \textit{expected structure} of the limit.  This is where the connection to population dynamics and the use of ultrametric spaces become important. We want an embedding that anticipates the genealogical relationships that will emerge in the limit.

\textbf{Example \ref{exm:dynamic-social-network} (continued):}  Recall our social network Example~\ref{exm:dynamic-social-network}. Instead of just labeling individuals with numbers $1, \dots, n$, we could embed them in a space that reflects their ``influence distance'' at each time step.  This ``influence distance'' would be our evolving ultrametric $r_t$.  This embedding is not arbitrary: it is \textit{adapted} to the dynamics of the system.

We can think of the limiting object, $(\CI^\ast, h, \mu)$, as defining a ``random environment'' for the graph.  The space $\CI^*$ and the measure $\mu$ tell us \textit{where} to sample vertices, and the connection function $h$ tells us \textit{how} to connect them. The key is that this environment itself can be \textit{dynamic}, evolving according to its own stochastic process. This is why we consider the process $(\CI^*_t, h_t, \mu_t)_{t \geq 0}$.

If we have a limiting object, say $(\CI^\ast,h,\mu)$, then we can generate a finite sampled sequence, respectively, an infinite sampled sequence. Indeed, take a sequence $(i_k)_{k \in \N}$ by sampling i.i.d.\ according to $\mu$. Let
\begin{equation}
\label{e653}
\wh\uuh=(h(i_k,i_\ell))_{k,\ell \in \N},
\end{equation}
be the \emph{empirical connection matrix} for the \emph{skeleton} $(i_k)_{k \in \N}$, which is a graph with vertex set $\N$. If $\CI^\ast$ is a Polish space, then the empirical measures $\wh\mu^{(m)}$ satisfy
\begin{equation}
\label{e659a}
\begin{aligned}
\wh\mu^{(m)} = \frac{1}{m} \sum_{k \in [m]} \delta_{i_k}
&\mtoo \mu,\\
\frac{1}{m^2} \sum_{k,\ell \in [m]} \delta_{(i_k,i_\ell)}
&\mtoo \mu \otimes \mu,\\
\frac{1}{m^L} \sum_{k_1,\ldots,k_L} \; \delta_{(i_{k_1},\ldots, i_{k_L})}
& \mtoo \mu^{\otimes L},\\
\frac{1}{m} \sum_{i \in [m]} \delta_{(\wh \uuh_L)(\tau_i \cdot)}
& \mtoo \nu^{(L)}, \quad L \in \N.
\end{aligned}
\end{equation}
Hence we obtain the \emph{continuum grapheme} as a limit of the \emph{sampled finite graphemes}.

Thus, starting from an evolution of finite graphs in $\G_n$ for $n \in \N$, the \emph{basic question} is what spaces $\CI^\ast$ we can take to form the elements $[\CI_n^\ast,h_n,\mu_n]$ of the sequence and its limit
\begin{equation}
\label{e637}
\left[\CI^\ast,h,\mu\right],
\end{equation}
which has the embedding property (denoted by $\hookrightarrow$):
\begin{equation}
\label{e641}
\begin{aligned}
&(i) \left[\CI_n^\ast ,h_n,\mu_n\right] \hookrightarrow [\CI^\ast ,h,\mu], \quad n \in \N,\\
&\quad \CI_n \hookrightarrow \CI,\, h_n=h_{\CI_n},\,\mu_n \hookrightarrow \wt \mu_n,\\
&\quad \wt \mu_n \longrightarrow \mu,\, h_n \rightarrow h,\, h \text{ measurable and $\{0,1\}$-valued},\\
&(ii) \left(\CI^*_n,h_n,\mu_n\right) \Tno \left(\CI^*,h,\mu\right).
\end{aligned}
\end{equation}
Here, (i) holds in the sense that $\nu_n \Rightarrow \nu$ or, equivalently,
\begin{equation}
\Phi([\CI_n^\ast,h_n,\mu_n]) \ntoo \Phi([\CI^\ast,h,\mu])), \qquad \Phi \in \Pi^\ast.
\end{equation}
For (ii) (i.e., convergence of the environment sequence fitting either $\G^{[]}$ or $\G^{\{\}}$), see Section~\ref{ss.space}. Of course, more than one choice is possible for the embedding. However, the procedure is \emph{consistent}.

\begin{proposition}[Consistency]
\label{prop.consist}
Every marked grapheme is the limit of its finite samples $\nu$-a.s. \qed
\end{proposition}

How to best choose the space in which we embed? From graphon theory we know that we can always choose $([0,1],\CB_{[0,1]}, \mu)$ and we get a grapheme if the entropy of the sampled finite graphs of size $n$ grows slower than $n^2$ (see~\cite{GdHKW23} and~\cite{hatami2012entropy}). But this need \emph{not} be the best choice for an embedding of the sample sequence. The space has to be rich enough to support the partition into completely connected components induced by the connection function at any fixed time. But we want more, namely, a stochastic process with \emph{regular paths}. This would allow us to get convergence in path space for the finite-graph dynamics, which is not yet reflected in~\eqref{e641}, which we strengthen to
\begin{equation}
\label{e2000}
\CL \left[\left(\left[ \CI_{n,t},h_{n,t},\mu_{n,t}\right]\right)_{t \geq 0}\right]
\utoo \CL \left[\left(\left[\CI_t,h_t,\mu_t\right]\right)_{t \geq 0}\right]
\end{equation}
or the same with $\{\}$. At this point we can either choose a \emph{universal} space, e.g.\ $([0,1],\CB_{[0,1]})$, with a fitting $\mu$, or \emph{random time-dependent} triples $(\CI^*_t,h_t,\mu_t)$ \emph{adapted} to the dynamics, and try to prove convergence of the finite-graph dynamics by choosing for $\CG^{(n)}$ any of the finite versions of $(\CI^\ast_t,h_t,\mu_t)$. As pointed out in Sections~\ref{sss.theorem} and \ref{sss.conex}, it is indeed useful to work with the latter choice and pick the \emph{space of genealogies} associated with the graph dynamics, since this is an \emph{effective} coding of the \emph{space-time process}, which we want to understand better through our constructions.

%%%

\paragraph*{Outlook on questions and further research $1$.}

We have already addressed the following questions for graphe\-mes. When is there a \emph{measurable} (or possibly even \emph{continuous}) $h$ for $[\CI^*,h,\mu]$, given a sequence $(i^{(m)})_{m \in \N}$ sampled with $\mu$, such that $(h(i_k,i_\ell))_{k,\ell \in \N}$ yield the connection matrix of the sequence? Do we have tightness in $n$? Is there a minimal or convenient embedding space? We have answered these questions when the states arise from the processes defined in Section~\ref{ss.examples}, viewed as processes taking values in $\G^{[]}$ or $\G^{[],V}$. The issue is whether we can view the evolving grapheme as a \emph{process in random environment} with a well-posed evolution of $\nu$ in $\G^{\langle\rangle}$, and with the environment having its own evolution, either \emph{autonomous} or arising from a \emph{well-posed} Markov process. These properties hold for the examples of Markov processes of graphs given in Section~\ref{ss.examples}, since for the underlying dynamics we know from the literature that $[\CI_n,r_n,\mu_n]$ converges to $[\CI,r,\mu]$. The dynamics of the vertices is exactly as in standard models, and we showed that $[(\CI_n,r_n),h_n,\mu_n]$ converges to a limit $[(\CI,r),h,\mu)]$ and $n\to\infty$.

What we have to look for are further types of dynamics that can be tackled in the same way. There are two interesting directions:

(1) Models where a \emph{fraction of edges} is removed and reconnected at any given time, giving jump processes. Here our approach can be based on so-called $\Lambda$-Cannings processes and their genealogies (see \cite{Lancestral}) and leads to the same theorems, without an explicit identification of the equilibrium.

(2) Consider processes where events as we have treated them above do \emph{not} come about in a Markovian way. This can again be tackled by using \emph{population models}, but now with \emph{seed-banks}, which allow us to remove the non-Markovianess by introducing countably many seed-banks (see~\cite{greven2020spatial, greven2021spatial,greven2022spatial}). Here, the genealogy process has not yet been treated, but would allow for a similar analysis as given here.

%%%

\paragraph*{Summary of the challenges.}

We face the following challenges when we start with a specific finite-graph dynamics:
\begin{itemize}
\item
Tightness, non-degeneracy and existence of $h$. Are \emph{further conditions} needed to guarantee the \emph{existence of  $\CG$} with values in $\G^{[]}$, i.e., do \eqref{e637}--\eqref{e2000} hold for a given dynamics?
\item
Is $(\CI^\ast, h,\mu)$ \emph{non-trivial}?
\item
We can always find an $h$ on $\CI \times \CI$ that is measurable as a random variable. Even when we have a sampled sequence $(i_n)_{n \in \N}$ and a connection matrix $(h(i_k,i_\ell))_{k,\ell \in \N}$, a key question is in what cases we can, instead of a random variable, obtain a continuous function on $\CI \times \CI$ via a continuum decomposition on $\supp(\mu^{\otimes 2})$, so that we have an evolving partition of a continuum for the completely connected components. In our examples we obtain an $h$ that allows for a \emph{continuous} continuation to $\CI \times \CI$ or, even better, to $(\supp(\mu))^2$. The question is how this construction can be generalised. We saw a possibility in our last theorem on the model with insertion and removal.
\end{itemize}
We have already resolved these challenges for our earlier classes of dynamics, but we must see how this can be generalised as a recipe, in particular, for the model where edges switch between active and non-active (Fleming-Viot with insertion or removal). In the other examples, we can identify $h$ and $\mu$ explicitly, because we have a \emph{genealogy state} $\CU$ generating the connection matrix that evolves according to a $\U$-valued stochastic process, which itself evolves autonomously. The latter yields our $\G^{[]}$-valued process, which has a natural topology on $\CI$, as well as the property of being \emph{dust-free}.

%%%

\subsection{Choice of state space, challenges and perspectives: intrinsic role of genealogies}
\label{ss.choice}

Often we are able to generate $\CI^\ast=(\CI,\CB)$ by considering a convenient Polish space with a Borel-$\sigma$-algebra like $[0,1]$ with the Euclidean distance. We always have a metric $r$ in any Polish space, but the problem is that it might not be easy to obtain a proper dynamics for this metric, in such a way that we get a choice for all $t \geq 0$ and for every realisation of the randomness. Since the measure changes as we follow the evolution, it is appropriate to choose  also $(\CI_t,\CB_t)_{t \geq 0}$ \emph{randomly}, at the cost that we do \emph{not} get a \emph{universal} space with \emph{only $\mu$} depending on time and on the underlying randomness.

For the classes of examples of finite-graph dynamics and grapheme dynamics we look at here, there is a \emph{canonical choice} based on the underlying \emph{ultrametric measure space} $(\CU_t,r_t)$ of the $\U$-valued Fleming-Viot or Dawson-Watanabe process generated by the stochastically evolving \emph{genealogy} $\CU_t$ of the population at time $t$ of the underlying population model. We saw that the ultrametricity of the spaces in which we embed is crucial to get a continuous representative of $h$. Does this mean that the choice had in fact an \emph{intrinsic character}? The answer is \emph{yes}, the reason being that the \emph{genealogy} of the population stores all the information about the \emph{space-time path process} up to time $t$ of the $\G^{\langle\rangle}$-valued version of our process in a \emph{natural and minimal way} in the state of the $\G^{[]}$-valued process, as we next explain.

If $(X_t)_{t \geq 0}$ is a Markov process, then we can consider the \emph{space-time path} process $(\CX_t)_{t \geq 0}$ with
\begin{equation}
\label{e1265}
\CX_t = (s,X_s)_{s \in [0,t]}.
\end{equation}
In our context and for finite graphs, if at time $t$ we ask at what earlier time a vertex got its connection to its current connected component from (by looking backwards in time), then we find an ancestral path and an imbedded search tree. This defines an ultrametric space $(\CI_t,r_t)$ that has this search tree as a skeleton. This tree codes the pivotal event that two vertices are in the same connected component. In particular, this space is rich enough to encode the information on the edges and their \emph{history}, and therefore so is the canonical space $(\CI,r)$ for the state at time $t$ for every dynamics that arises from a finite-graph dynamics where edges are added or removed by connecting or disconnecting vertices from their connected components. Here, we simply use the \emph{space-time process} to see the \emph{intrinsic tree structure}. The observation above also indicates that the path in the state space $\G^{[]}$ is the minimal one that unravels the information on the \emph{space-time path  process in $\G^{\langle\rangle}$} if we choose the space given by the underlying $\U$-valued genealogy process. These are also the right spaces to get limiting dynamics as the number of vertices diverges.

%%%

\paragraph*{Outlook on questions and further research $2$.}

In fact, ongoing work on population models shows that we can embed realisations of the grapheme as countable graphs at different times in a single grounded $\R$-tree equipped with a collection of sampling measures. In this way we get a representation of the \emph{whole space-time path process} of the \emph{countable graphs}, in a single tree-like metric measure space for all times simultaneously.

Several avenues for future research emerge from this work:
\begin{itemize}
    \item \textbf{Non-ultrametric embeddings:} Exploring other types of embedding spaces beyond ultrametric spaces to capture different aspects of graph dynamics.
    \item \textbf{Non-Markovian extensions:} Investigating non-Markovian grapheme dynamics, potentially by incorporating age-dependent or history-dependent evolution rules.
    \item \textbf{Beyond dense graphs:} Extending the grapheme framework to study dynamic sparse graphs, potentially by incorporating graphex-like constructions.
\end{itemize}

In conclusion, the grapheme framework offers a promising new direction for the study of dynamic graphs. By leveraging tools from population dynamics and stochastic analysis, it provides a rigorous and versatile approach to model and analyze the complex evolution of network structures over time.

\appendix

%%%%%%%% APPENDIX A %%%%%%%%%%%%%%%%%%%%%%%

\section{Graphon dynamics}
\label{ss.graphons}

%%%

\paragraph*{Static graphons.}

Let $\CW$ be the space of measurable symmetric functions $h\colon\,[0,1]^2 \to [0,1]$. The elements of $\CW$ are called \textit{graphons} \cite{LovaszSzegedy2006graphon}. A finite simple graph $G$ with $n$ vertices can be represented as a graphon. Specifically, define $h^{G} \in \CW$ in a natural way as 
\begin{equation}
\label{graphondef}
h^{G}(x,y) = \left\{ \begin{array}{ll}
1 &\mbox{if there is an edge between vertex } \lceil{nx}\rceil \mbox{ and vertex } \lceil{ny}\rceil,\\
0 &\mbox{otherwise},
\end{array}
\right.
\end{equation}
which leads to an \textit{empirical graphon} comprised of $n \times n$ zero-one blocks, see Figure~\ref{fig:graph-and-graphon}. The space of functions $\CW$ is endowed with the \emph{cut distance}
\begin{equation}\label{e385}
d_{\square} (h_1,h_2) = \sup_{S,T\subset [0,1]}
\left|\int_{S\times T} \dd x\,\dd y\,[h_1(x,y) - h_2(x,y)]\right|,
\qquad h_1,h_2 \in \CW.
\end{equation}
On $\CW$ there is a natural equivalence relation $\sim$, which is a proxy of the graph isomorphism relation. Namely, let $\Sigma$ be the space of measure-preserving bijections $\sigma\colon\, [0,1] \to [0,1]$. Write $h^\sigma(x,y):=h(\sigma x,\sigma y)$. Then, $h_1,h_2 \in \CW$ are said to be equivalent, written $h_1\sim h_2$, when, for some $\sigma \in \Sigma$, $h_1 \equiv h_2^{\sigma}$ a.e.~w.r.t.~Lebesgue measure on $[0,1]$. Denote the set of graphon equivalence classes by $\wt\CW = \CW / \sim$. Define the \textit{cut-distance} between equivalence classes as
\begin{equation}
\label{deltam}
\delta_{\square}(\tilde{h}_1,\tilde{h}_2)
= \inf _{\sigma_1,\sigma_2 \in \Sigma} d_{\square}(h_1^{\sigma_1}, h_2^{\sigma_2}),
\qquad \tilde{h}_1,\tilde{h}_2 \in \wt\CW,
\end{equation}
where $\tilde{h}$ denotes the equivalence class of which $h$ is a representative. To summarize, the equivalence relation $\sim$ yields the \emph{quotient space} $(\wt\CW,\delta_{\square})$, which is a compact Polish space (see Theorem 5.1. in \cite{LS07}) and which is the proper object to describe countable limits of finite graphs, see Figure~\ref{fig:graph-and-graphon}. 

Moreover, it is natural to consider $\wt\CW$ as a state space for random variables and even stochastic processes.

%%%%%%%%%%%%%%%%%%%%%%%%%%%%%%%%%%%%%%%%
\begin{figure}[htbp]
\centering
\begin{tikzpicture}

% Left subfigure - Graph with 5 vertices and 8 edges
\begin{scope}[shift={(-4,0)}, scale=1.2]
    % Define vertices
    \node[draw, circle, fill=black, inner sep=2pt] (1) at (0,2) {};
    \node[draw, circle, fill=black, inner sep=2pt] (2) at (2,2) {};
    \node[draw, circle, fill=black, inner sep=2pt] (3) at (2,0) {};
    \node[draw, circle, fill=black, inner sep=2pt] (4) at (0,0) {};
    \node[draw, circle, fill=black, inner sep=2pt] (5) at (1,1) {};
    
    % Label vertices
    \node at (0,2.3) {1};
    \node at (2,2.3) {2};
    \node at (2,-0.3) {3};
    \node at (0,-0.3) {4};
    \node at (1,1.3) {5};
    
    % Draw edges
    \draw[thick] (1) -- (2);
    \draw[thick] (2) -- (3);
    \draw[thick] (3) -- (4);
    \draw[thick] (4) -- (1);
    \draw[thick] (1) -- (5);
    \draw[thick] (2) -- (5);
    \draw[thick] (3) -- (5);
    \draw[thick] (4) -- (5);
\end{scope}

% Right subfigure - Empirical Graphon
\begin{scope}[shift={(2,0)}]
    % Create adjacency matrix
    \def\adjacencymatrix{{
        {0, 1, 0, 1, 1},
        {1, 0, 1, 0, 1},
        {0, 1, 0, 1, 1},
        {1, 0, 1, 0, 1},
        {1, 1, 1, 1, 0}
    }}
    
    % Draw the graphon as a heatmap
    \begin{scope}[scale=2.5]
    \foreach \i in {0,...,4} {
        \foreach \j in {0,...,4} {
            \pgfmathsetmacro{\edge}{\adjacencymatrix[\i][\j]}
            \pgfmathsetmacro{\x}{\i/5}
            \pgfmathsetmacro{\y}{\j/5}
            \ifnum\edge=1
                \fill[black] (\x,\y) rectangle ++ (1/5,1/5);
            \else
                \fill[white] (\x,\y) rectangle ++ (1/5,1/5);
                \draw[black, very thin] (\x,\y) rectangle ++ (1/5,1/5);
            \fi
        }
    }
    
    % Add axis labels
    \draw[thick] (0,0) -- (1,0);
    \draw[thick] (0,0) -- (0,1);
    \foreach \i in {0,...,5} {
        \draw[thick] (\i/5,0) -- ++ (0,-0.05);
        \draw[thick] (0,\i/5) -- ++ (-0.05,0);
        \node[below, font=\tiny] at (\i/5,0) {$\frac{\i}{5}$};
        \node[left, font=\tiny] at (0,\i/5) {$\frac{\i}{5}$};
    }
    \end{scope}
\end{scope}

\end{tikzpicture}
\caption{\small Left: A finite connected graph $G$ with 5 vertices and 8 edges. Right: Emprirical graphon $h^G$ representation of $G$, as a heat map with two levels: black = 1, white = 0.}
\label{fig:graph-and-graphon}
\end{figure}
%%%%%%%%%%%%%%%%%%%%%%%%%%%%%%%%%%%%%%%%%%%%%%%%%%

Crucially, the space of graphon equivalence classes $(\widetilde{\mathcal{W}}, \delta_{\square})$, when viewed as a topological space, can also be seen as a compactification of the space of finite graphs $G$ (more precisely, of empirical graphons $h^G \in \CW$), when equipped with the topology of \textit{subgraph density convergence}.

Specifically, for $h \in \CW$ and $F$ a finite simple graph with $m$ vertices and edge set $E(F)$, define
\begin{equation}
\label{eq:homdensity}
t(F,h) = \int_{[0,1]^m}  \dd x_1\ldots \dd x_m \prod_{\{i,j\} \in E(F)} h(x_i,x_j).
\end{equation}
Then, the \emph{homomorphism density} (= subgraph density) of $F$ in $G$ equals $t(F,h^G)$, where $h^G$ is the empirical function defined in \eqref{graphondef}. Note that $t(F,h)$ is the same for all representatives $h$ in the equivalence class $\tilde{h}$. Any graphon $\tilde{h}\in\wt{W}$ is uniquely determined by its subgraph densities, and can be approximated by empirical graphons. The important mathematical fact is that convergence in $(\wt\CW,\delta_{\square})$ is \emph{equivalent} to convergence of all the subgraph densities (see Fig.~\ref{fig:graphon-convergence}) so that this space \emph{equals} the compactification of the subspace of empirical graphons equipped with the topology introduced by subgraph counts on $\wt \CW$. For a more detailed description of the structure of the space $(\wt\CW,\delta_{\square})$ we refer the reader to \cite{BCLSV08,BCLSV12,DGKR15}.

Graphons are suited to describe limits of \emph{dense} graphs, for which the number of edges is of the order of the square of the number of vertices. For \emph{non-dense} graphs, we get a single equivalence class of constant zero graphons only, which is not particularly informative, and a multi-scale analysis is required, leading to new objects, for example, a \emph{graphex} (see, e.g., \cite{VeitchRoy2015}).

%%%%%%%%%%%%%%%%%%%%%%%%%%%%%%%%%%%%%%%%%%
\begin{figure}[htbp]
\centering
\includegraphics[scale=0.4]{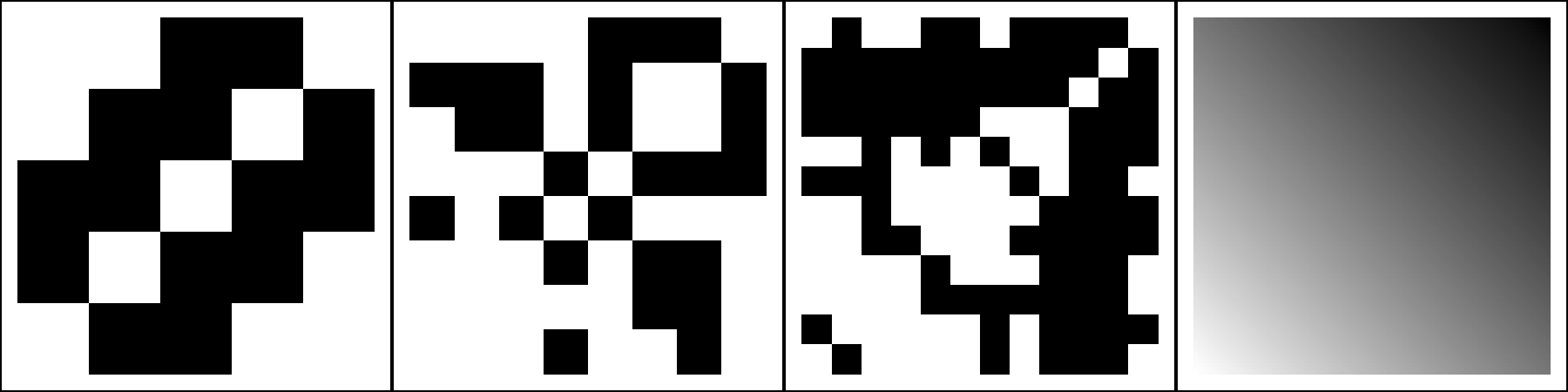}
\caption{\small Convergence of (black and white) empirical graphons to a (grayscale) limiting graphon as $n \to \infty$. From left to right $n = 5, 8, 12, \infty$.}
\label{fig:graphon-convergence}
\end{figure}
%%%%%%%%%%%%%%%%%%%%%%%%%%%%%%%%%%%%%%%%%%

%%%

\paragraph*{Graphon-valued stochastic processes.}

The first attempts to construct \emph{dynamics} on the space $\wt \CW$ of graphons were made in \cite{Crane2016, Crane2017}, using the Aldous-Hoover theory for exchangeable infinite arrays. This work led to $\wt \CW$-valued Markov processes with c\`adl\`ag paths of locally bounded variation, i.e., a mixture of a stochastic pure-jump process and a deterministic flow. A further generalisation of this construction, beyond just zero-one-valued infinite adjacency matrices, was obtained in \cite{CK18}. Specifically, when it comes to agent-based modelling, the agents typically interact in a way that depends on the current interaction strength between them, or on other features that can be associated with an edge in the graph linking them. It is natural to allow not only for discrete edge features, such as interaction versus no interaction, but also for continuous edge features. In fact, the latter makes the theory of Markovian infinite exchangeable arrays more transparent. Finally, \cite{CK18} highlights the fact that finite restrictions of infinite exchangeable Markovian arrays are not necessarily Markov processes, unless the infinite array has the Feller property, see \cite[Section~4.1 and Theorem~6.3]{CK18}.

Subsequently, building on models from population dynamics, in \cite{AHR21} \emph{diffusion-like} graphon-valued processes were constructed. Specifically, stochastic processes with continuous paths describing the non-trivial \textit{evolution of subgraph densities} were constructed. The dynamics of subgraph densities were driven by diffusion processes from population genetics. Even though \cite{AHR21} provided a large class of examples, intended as a \emph{proof of concept}, no general theory was developed and consequently many structural questions remained open. For instance, it was not clear whether or not the examples constructed in \cite{AHR21} are strong graphon-valued Markov processes that can be described by a generator acting on a dense class of test functions on the graphon space.  Moreover, all the examples had \emph{trivial} equilibria concentrated on constant graphons.

Further steps were taken in \cite{Bra20,Bra22}, where \emph{sample-path large deviation principles} were derived for dynamic Erd\H{o}s-R\'enyi  random graphs (generalising the static large deviation principle in \cite{CV11}; see also \cite{C15}), and for random graphs in which each vertex has a type that fluctuates randomly over time, in such a way that \emph{collectively} the paths of the edges and the vertex types up to a given time determine the probabilities that the edges are present or absent at that time.

In the present paper, we take a different route towards building a limiting dynamics, based on a new notion we call \emph{grapheme}, which views a countable graph as an object embedded in a measure space built from a \emph{Polish space} (and therefore also looks at finite embedded graphs). This embedding is necessary to build a \emph{stochastic process} of graphemes via \emph{martingale problems}. Like in \cite{AHR21}, we exploit the machinery developed for models from population dynamics, but we add the idea to also exploit explicitly the \emph{tree structure} behind the \emph{genealogy} of populations, which we review as we go along and which is both natural and mathematically convenient, as we will see. This approach allows for storing information on the history of the \emph{edge structure in the current state}, and hence for studying the \emph{space-time path process} of countable graphs.

%%%

\section{Connection to the literature}
\label{s.connections}

We next address the question how Theorems~\ref{th.963-I} and \ref{th.963-II} and Corollary~\ref{cor.932} relate to earlier work in the \emph{graphon literature}.

%%%

\subsection{Graphons}

\begin{remark}[Connection to graphons and~\cite{AHR21}]
\label{r.936}
$\mbox{}$\\
{\rm
(1) For each $t >0$ we can represent the state in $\G^{\langle\rangle}$ as a partition of $[0,1]$, namely, the interval lengths under the law $\mu_t$ of the completely connected components ordered by size (up to equivalence). For $c=0$ and for positive times, these evolve like a multi-dimensional diffusion. However, for $c>0$, any attempt to get this as an $\ell_1(\N)$-valued diffusion fails due to the fact that there are \emph{countably} many intervals, with countably many intervals either disappearing or being newly inserted during every positive time interval.\\
(2) The above embedding in $[0,1]$ can be used to represent the induced $\G^{\{\}}$-valued process via the space ($[0,1]$, Euclidean distance, uniform distribution). However, this does \emph{not} work for the $\G^{[]}$-process, because the embedding is not necessarily isometric.\\
(3) From the representation on $[0,1]$ we can obtain the $\wh \CW$-valued process of \emph{graphons} as a \emph{functional} of $(\CG_t^{a,c,\theta})_{t \geq 0}$, written
\begin{equation}\label{e1030}
(\CG_t^{\rm graph, (a,c,\theta)})_{t \geq 0},
\end{equation}
as a process with continuous paths and with a stationary distribution, arising as the limit of finite-grapheme processes. This process on $\wh \CW$ does \emph{not} arise as the solution of a standard well-posed martingale problem with the strong Markov property and the Feller property and with continuous paths. Whether or not such a martingale problem exists is questionable for $c >0$, while for $c=0$ existence can be shown with the help of the so-called Petrov equations (see~\cite{Pet09})\footnote{The Petrov equations are a countable system of coupled partial differential equations (PDEs) that describe the evolution of the \textit{ranked} frequencies $p_1(t) \geq p_2(t) \geq \ldots$ of types with $\sum_i p_i = 1$ in a multi-type Fleming-Viot diffusion \textit{without} mutation or immigration/emigration, i.e., the $c=0$ case): $\partial_t p_i(t) = \frac{1}{2} \left[ \sum_{j=1}^\infty p_j(t) \left( \partial^2_{p_i p_j} - \partial^2_{\partial p_i^2} \right) - \partial_{p_i} \right] \left(\sum_{k=1}^\infty p_k(t)^2 \right)$, $i \in \N$.}, establishing the existence of a process of interval partitions.

On a subspace of $\wh \CW$, for the graphons that are in $\G^{\langle\rangle}_{\rm comp}$, this can be done (see also~\cite{GdHKW23}), but it cannot be extended to $\wh \CW$ in the standard form, as we argued above.\\
(4) For $c=0$, the subclass of those $\G^{[]}$-valued processes corresponding to Fisher-Wright diffusions with immigration and emigration contains the processes given in~\cite{AHR21}.} \hfill$\spadesuit$
\end{remark}

\begin{remark}[Dynamics of subgraph-counts for $\G^{\langle\rangle}$-valued process]
\label{r.948}
$\mbox{}$
{\rm Even though the $\G^{\langle\rangle}$-valued process $(\CG_t^{a,c,\theta})_{t \geq 0}$ is uniquely determined by the equations following from the martingale problem, we can only conclude that it is a $\CG^{\langle\rangle}$-valued process because it arises as a functional of $(\CG_t^{a,c,\theta})_{t \geq 0}$. Indeed, this guarantees that we have a \emph{$\langle\rangle$-grapheme}, i.e., there exists an embedding of the countable graph in some Polish space along the \emph{whole path}. The question whether or not we could obtain the latter from properties of the subgraph-counts as paths, without reference to the $\G^{[]}$-valued process, has to be based on approximation and duality, and on an entropy criterion (see~\cite{GdHKW23}).} \hfill$\spadesuit$
\end{remark}

%%%

\subsection{Relation to the graphon literature}
\label{sss.discus}

Since in~\cite{AHR21} the graphon dynamics was based on the Fisher-Wright dynamics, which is closely related to our Fleming-Viot dynamics, it is worthwhile to connect the statements in \cite{AHR21} to what we obtained via graphemes. In our framework the parallel process arises if initially in our Fleming-Viot dynamic $(\CG^{d,0,0})_{t \geq 0}$ we assign inheritable types $0$ or $1$ to our types and connect all the completely connected components carrying $1$ with edges to all the other vertices, so that a random number of completely connected components are \emph{merged} into a single component.

\begin{remark}[Connection matrices and density representation]
\label{r.1197}
$\mbox{}$\\
{\rm
(1) The choice of embedding in a Polish space suggested by the literature on graphons is obtained by putting $(\CI,\tau,\mu) = ([0,1],d,\mu)$ with $d$ the Euclidean distance and $\mu$ the uniform distribution. A measurable function $H$ can be read off in $\CW$, the space of $h$-functions defined in Section~\ref{ss.graphons}. Namely, we can identify a graphon, which is an element of the space $\wh\CW$ of equivalences classes of $h$-functions, with the property that there exists a measure $\Gamma$ on $[0,1]$ such that, for a symmetric $[0,1]$-valued function $H$ (representing the graphon function),
\begin{equation}
\label{e706}
\Phi \left(\left[\CI^*,h,\mu \right]\right) = \int_{[0,1]^m} \Gamma^{\otimes m} (\d \underline{x})
\prod_{i,j \in [m],\,i \neq j} H(x_i,x_j),
\end{equation}
i.e., $\Phi$ is the `all connected monomial'. This object relates to our concept of grapheme $[\CI^*,h,\mu]$ as follows. If the unique $H$ is \emph{not} $\{0,1\}$-valued, then ($[0,1]$, Borel, $H,\Gamma$) does \emph{not} relate to a grapheme in $\G^{\langle\rangle}$, so graphons arising from graphemes can converge to a limit graphon without the graphemes converging themselves, since there is no space in which a \emph{countable graph} can be realized having the limiting subgraph counts.\\
(2) Conversely, we obtain graphons as \emph{functionals} of graphemes. Indeed, for the cases treated in this paper we can consider the measure $(\mu_t)_{t \geq 0}$, take the weights of the completely connected components, giving the vector $(\bar \mu_t(i))_{i \in \N}$, size-order this vector, place it on the interval $[0,1]$, and follow the separation points in time. If  (and only if) $c=0$, then for the Fleming-Viot evolution rule this gives paths of embedded \emph{nested multi-type Fisher-Wright diffusions}, whose number is finite and random, solving the Petrov equations~\cite{Pet09}. This specifies a process of graphemes based on $(([0,1],d),h_t,\mu_t)_{t \geq 0}$, with $d$ the Euclidean metric, where we can read off $(h_t)_{t \geq 0}$ from the diffusions, as well as a process of graphons with values in $\wh \CW$. In this way we can formulate another martingale problem, which produces a grapheme dynamics with the same marginal distributions as our $\G^{\langle\rangle}$-valued and $\G^{\{\}}$-valued dynamics, but \emph{different} marginals on $\G^{[]}$. Nevertheless, in both cases strong Markov processes result.\\
(3) If $c>0$, then we get a \emph{countable} set of partitions, we \emph{cannot} follow the simple procedure of (b) above because there are countably many reshufflings occurring to restall the ranking, and we do \emph{not} get a nice process of graphemes in $\G^{[]}$ based on $(([0,1],d),\wt \mu_t)_{t \geq 0}$, as we explain below. We have to bring in our genealogical structure to resolve the open issues, as demonstrated in our theorems.
}
\hfill$\spadesuit$
\end{remark}

Alternatively, we could connect our grapheme diffusion to the graphon diffusion constructed in~\cite{AHR21}, by noting that we can view $H_t$ in \eqref{e706} as densities giving $\mu_t^{\otimes m}$ when $\mu$ is the Lebesgue measure on $[0,1]$. Then $H_t$ generates the sampling of vertices and hence of edges, and thus produces a \emph{connection matrix $h_t$} that is based on the embedding in $[0,1]$. However, the evolution of the resulting $h_t$ is \emph{not} interpretable in terms of the underlying population dynamics. The function $H_t$ is a certain \emph{average} over the state $h_t$, giving the connections \emph{based on the Euclidean metric}. If we would take the ultrametric induced on a subset of $[0,1]$ and choose the same embedding of the finite graphs in our construction, then we would get different elements in $\G^{\{\}}$ and $\G^{[]}$, but the same element in $\G^{\langle\rangle}$. To get a grapheme dynamics in $\G^{[]}$ or $\G^{\{\}}$ we would need to prove that, for the embedding in~\cite{AHR21}, the convergence takes place in the sense of $\G^{[]}$ and $\G^{\{\}}$. This is difficult because the spaces for embedding are \emph{not} well adapted to the dynamics.

In \cite{AHR21}, graphons in $\wh \CW$ describe the continuum evolution as the evolution $(\mu_t)_{t \geq 0}$, which defines a $\G^{[]}$-valued dynamics, and is only necessary due to the choice $\CI^\ast=([0,1],\CB_{[0,1]})$ for all $t$, which is deterministic. Both this dynamics and our dynamics have the same $\G^{\langle\rangle}$-valued functional. The deterministic choice of $\CI^\ast$ restricts the possibilities of getting a manageable Markovian dynamics. In our approach we choose to work with graphemes instead of a graphons, getting proper evolutions of countable graphs.

Another way to get a connection with the strategy in~\cite{AHR21} is via the embedding constructed on~\cite[page 47]{AHR21} on $[0,1]^2$ in Euclidean space. Getting an analogue for our embedding in $(\CU_t)_{t \geq 0}$ works whenever we have a \emph{skeleton} and an \emph{ultrametric $r$ on $\CI$} such that we can use small balls to cover the edges, count these small balls, and set the normalised number equal to the value of $H_n$ of the size-$n$ graph approximation of the grapheme. The limiting $H$ we get in this way depends on the choice of $(\CI,r)$ due to the implicit local averaging, and the same is true for $\wh\CW$ with the Euclidean distance in $[0,1]^2$. The $H$ thus constructed acts as a graphon and is a block function in our examples, except the one with insertion and deletion, where $H \equiv 1$ with "frequency" $\frac{a^+}{a^+ + a^-}$ on a sequence of blocks, and $H \equiv 0$ on the complement of these blocks, where the block masses are fluctuating in time, which is different from the Erd\H{o}s-R\'enyi grapheme. Furthermore, this $H$ is a \emph{functional} of our dynamics and is embedded in a \emph{richer} structure.

%%%%%%%%%%% REFERENCES %%%%%%%%%%%%%%%%%%%%%%%

%%%%%%%%%%%%%%%%%%%%%%%%%%%%%%%%%%%%%%%%%%%%
\bibliographystyle{alpha}
\bibliography{graphon-large-paper}
%%%%%%%%%%%%%%%%%%%%%%%%%%%%%%%%%%%%%%%%%%%%

\end{document}